\documentclass{amsart}
\usepackage{amsfonts}
\usepackage{amsmath,amsthm} 
\usepackage{hyperref} 
\usepackage{latexsym}
\usepackage{array}
\usepackage{amssymb}
\usepackage{enumerate}
\usepackage{bbold}
\usepackage[francais,english]{babel}
\usepackage[T1]{fontenc}
\usepackage{color}
\usepackage{stmaryrd}
\usepackage{mathrsfs} 
\usepackage{graphicx}
\usepackage{soul}
\usepackage{multirow}

\theoremstyle{plain}  
\newtheorem{theorem}{Theorem}[section]
\newtheorem{lemma}[theorem]{Lemma}
\newtheorem{proposition}[theorem]{Proposition}

\newtheorem{corollary}[theorem]{Corollary} 
\newtheorem{definition}[theorem]{Definition} \theoremstyle{remark}
\newtheorem{remark}[theorem]{Remark}

\newcommand{\Rc}{\mathcal{R}}

\newcommand{\Ac}{\mathcal{A}}

\newcommand{\dd}{\mathrm{d}}

\newcommand{\Fscr}{\mathscr{F}}
\newcommand{\Hscr}{\mathscr{H}}
\newcommand{\Hc}{\mathcal{H}}

\newcommand{\Ic}{\mathcal{I}rr}

\newcommand{\jb}{{\boldsymbol{j}}}

\newcommand{\hb}{\boldsymbol{h}}
\newcommand{\kb}{{\boldsymbol{k}}}

\newcommand{\irr}{{\rm Irr}}

\newcommand{\kbsf}{{\boldsymbol{\mathsf{k}}}}
\newcommand{\hbsf}{{\boldsymbol{\mathsf{h}}}}

\newcommand{\N}{\mathbb{N}}
\newcommand{\Nc}{\mathcal{N}}

\renewcommand{\Mc}{\mathcal{M}}

\newcommand{\Pb}{\mathbb{P}}
\newcommand{\R}{\mathbb{R}}

\newcommand{\C}{\mathbb{C}}

\newcommand{\T}{\mathbb{T}}
\newcommand{\Z}{\mathbb{Z}}

\newcommand{\Uc}{\mathcal{U}}
\newcommand{\U}{\mathbb{U}}
\newcommand{\Vc}{\mathcal{V}}
\newcommand{\Wc}{\mathcal{W}}

\newcommand{\Zc}{\mathcal{Z}}

\newcommand{\eps}{\varepsilon}

\newcommand{\length}{\sharp}
\newcommand{\pig}{\boldsymbol{\pi}}
\newcommand{\Ascr}{\mathscr{A}}
\newcommand{\Rscr}{\mathscr{R}}

\newcommand{\va}[1]{|#1|}

  \newcommand{\vat}{|t|}

\newcommand{\Norm}[2]{\|#1\|\left.\vphantom{T_{j_0}^0}\!\!\right._{#2}}         
             
\author{Joackim Bernier }
\address{Univ Rennes, INRIA, CNRS, IRMAR - UMR 6625, F-35000 Rennes, France} 
\email{Joackim.Bernier@ens-rennes.fr}

\author{Erwan Faou}
\address{Univ Rennes, INRIA, CNRS, IRMAR - UMR 6625, F-35000 Rennes, France} 
\email{Erwan.Faou@inria.fr}

 \author{ Beno\^it Gr\'ebert}
\address{Laboratoire de Math\'ematiques Jean Leray, Universit\'e de Nantes, UMR CNRS 6629\\
2, rue de la Houssini\`ere \\
44322 Nantes Cedex 03, France}
\email{benoit.grebert@univ-nantes.fr}

\title[Rational normal forms for NLS]{Rational normal forms and stability of small solutions to nonlinear Schr\"odinger equations}



\begin{document}
\begin{abstract}
We consider general classes of nonlinear Schr\"odinger equations on the circle with nontrivial cubic part and without external parameters.  We construct  a new type of normal forms, namely rational normal forms, on open sets  surrounding the origin in high Sobolev regularity. With this new tool we
prove that, given a large constant $M$ and a sufficiently small parameter $\varepsilon$, for generic initial data of size $\varepsilon$, the flow is conjugated to an integrable flow up to an arbitrary small remainder of order $\varepsilon^{M+1}$. This implies  that for such initial data $u(0)$ we control  the  Sobolev norm of the solution $u(t)$ for  time of order $\varepsilon^{-M}$. Furthermore this property is locally stable:  if $v(0)$ is sufficiently close to $u(0)$ (of order $\varepsilon^{3/2}$) then the solution $v(t)$ is also controled for  time of order $\varepsilon^{-M}$. 
\end{abstract}

\maketitle
\tableofcontents

\section{Introduction}

In this paper we are interested in the long time behavior of  small amplitude solutions of non-linear Hamiltonian partial differential equations on bounded domains. In this context, the competition between non-linear effects and energy conservation (typically the $H^1$ Sobolev norm) makes the problem intricate. One of the main issues is the control of higher order Sobolev norms  of solutions for which typically a priori upper bounds are polynomials (see \cite{Bou96a,Sta97, Bou03, Soh11, CKO12})\\   Bourgain  exhibited in \cite{Bou96a} examples of growth of high order Sobolev norms for some solutions of a nonlinear wave equation in 1d with periodic boundary conditions. These examples were constructed by using as much as possible the totally resonant character of the equation (all the linear frequencies are integers and thus proportional).\\ 
On the contrary, Bambusi \& Gr\'ebert have shown in \cite{BG06} (see also \cite{Bam03,Bou03}) that, in a fairly general semi linear PDE framework, if an appropriate non-resonance condition is imposed on the linear part then the solution $u$ of the corresponding PDE satisfy a strong stability property:
\begin{equation}\label{star}
\text{if }\ \Norm{u(0)}{H^s}\leq \eps \quad \text{then } \quad \Norm{u(t)}{H^s}\leq 2\eps \quad \text{ for all }\quad  t\leq  \eps^{-M}, 
\end{equation}
where $\Norm{\,\cdot\,}{H^s}$ denotes the Sobolev norm of order $s$, $M$ can be chosen arbitrarily large and $\eps$  is supposed to be small enough, $\eps<\eps_0(M,s)$. 
 The method of proof is based on the construction of Birkhoff normal forms. To verify the appropriate non-resonance condition external parameters were used --such as a mass in the case of nonlinear wave equation-- and the stability result were obtained for almost every value of these parameters. Then this technic was applied to prove almost global existence results for a lot of semi linear Hamiltonian PDEs (see \cite{BDGS07,Bam07,Gre07,GIP09,FG10}).  However, the case of a non-linear perturbation of a fully resonant linear PDE was not achievable by this technique. Actually for the cubic nonlinear Schr\"odinger equation on the two dimensional torus it is proved in \cite{CKSTT10} that the high Sobolev norms may growth arbitrarily for some special initial data. Even in one  dimension of space, it is proved in \cite{GT12} that the quintic nonlinear Schr\"odinger equation on the circle does not satisfy \eqref{star} (but without arbitrary growth of the high Sobolev norms, see also \cite{HP17} for a generalization or \cite{CF} for a two-dimensional example).

Now consider the nonlinear Schr\"odinger equation:
\begin{equation}
\tag{NLS}\label{nls}
iu_t=-\Delta u+ \varphi(|u|^2) u , \quad x\in \T , \quad t \in \R, 
\end{equation}
where $\varphi= \R \to \R$ is an analytic  function on a neighborhood of the origin satisfying $\varphi(0)=m$ is the mass possibly 0  and $\varphi'(0) \neq 0$. 
Equation \eqref{nls} is a Hamiltonian system associated with the Hamiltonian function 
\begin{equation}
\label{HamNLS}
H_{\rm NLS}(u,\bar u ) = \frac{1}{2\pi}\int_{\T} \left(\va{\nabla u }^{2}   + 
g(|u|^2) \right) \dd x, 
\end{equation}
where $g(t)=\int^t_0 \varphi$, 
and the complex symplectic structure $i\dd u\wedge \dd\bar u$. \\
This is an example of fully resonant Hamiltonian PDE, as the linear frequencies are $j^2 \in \N$ for $j\in\Z$. Nevertheless in \cite{KP96} Kuksin \& P\"oschel  proved  for such equation the persistence of finite dimensional KAM tori, a result that requires a strong non resonant property on the unperturbed Hamiltonian. Actually they considered  the cubic term as part of the unperturbed Hamiltonian  to modulate the resonant linear  frequencies and to avoid the problem of resonances. Roughly speaking the nonlinear term generates stability.  Then Bourgain in \cite{Bou00} used the same idea to prove that  for many random small initial data  the solution of \eqref{nls} satisfies \eqref{star}. Although the method of proof is based on normal forms, the effective construction of the normal form depends on the initial datum in a very intricated way and actually the author does not obtain a Birkhoff normal form result for \eqref{nls} but rather a way to break down the solution that allows him to obtain the property \eqref{star}.   \\
 In this work we want to construct a new type of normal form, not based on polynomial functions but on rational functions (see Section \ref{rational}), transforming the Hamiltonian of \eqref{nls} into an integrable one up to a small remainder, over large open sets surrounding the origin. Then stability of higher order Sobolev norms during very long time is just one of the dynamical consequences. We stress out that since our rational  normal  form is built on  open sets, the dynamical consequences remain stable with respect to the initial datum. In particular the property \eqref{star}, although not verified on all a neighborhood of the origin,  is locally uniform with respect to $u(0)$ in $H^s$ .   \\ 
 To describe our result let us introduce some notations.
With a given function $u(x) \in L^2$, we associate the Fourier coefficients $(u_a)_{a\in\Z} \in \ell^2$ defined by 
$$
u_a = \frac{1}{2\pi}\int_{\T} u(x) e^{- i ax} \dd x.
$$
In the remainder of the paper we identify the function  with its sequence of Fourier coefficients $u = (u_a)_{a\in\Z}$, 
and as in \cite{FG10,F12} we consider the spaces 
\begin{equation}
\label{ElleLaisseUn}
\ell_s^1 = \{\,  u = (u_a)_{a\in\Z} \in \ell^2 \quad \mid \quad \Norm{u}{s} := \sum_{a \in \Z} \langle a \rangle^s |u_a| < +\infty \, \}, 
\end{equation}
where $\langle a \rangle = \sqrt{1 + a^2}$. Note that these spaces are linked with the classical Sobolev spaces by the relation $H^{s'} \subset \ell_s^1 \subset H^s$ for $s' - s > 1/2$. 

Our method also applies to equations with Hartree nonlinearity of the form (Schr\"odinger-Poisson equation)
\begin{equation}\label{nlsp}
\tag{NLSP}
\begin{array}{l}
i u_t=-\Delta u+  W u , \quad x\in \T , \quad t \in \R,\\[2ex]
 - \Delta W =  |u|^2 - \displaystyle \frac{1}{2\pi} \int_\T |u|^2 \dd x, 
\end{array}
\end{equation}
for which we have $W = V \star |u|^2$ where $V$ is the Green function of the operator $- \Delta$ with zero average on the torus. 
The Hamiltonian associated with this equation is 
\begin{equation}
\label{HamNLSP}
H_{\rm NLSP}(u,\bar u ) = \frac{1}{2\pi}\int_{\T} \left(\va{\nabla u }^{2}   + 
\frac12 (V \star |u|^2) |u|^2 \right) \dd x. 
\end{equation}
Different kind of convolution operators $V$ could also be considered, as well as higher order perturbations of \eqref{nlsp} (note however that unlike the \eqref{nls} case the cubic \eqref{nlsp} equation is not integrable in dimension 1).  As we will see, the probabilistic results obtained for \eqref{nls} and \eqref{nlsp} differ significantly.  

\medskip

Our results are divided into two parts : 

$\bullet$ {\bf Abstract rational normal forms} (see Theorem \ref{mainth}). We construct a canonical transformation $\tau$ defined on an open set $ \Vc_\eps\subset\ell_s^1$ included in the ball of radius $\eps$ centered at $0$  that puts the Hamiltonian of \eqref{nls} (resp. \eqref{nlsp}) in normal form up to order $2r$: $H\circ\tau=Z(I)+R$ where $Z$ depends only on the actions $ I = (I_a)_{a \in \Z}$ with $I_a = |u_a|^2$,  and $R=O(\eps^{2r+1})$. The proof for this result is outlined in Section \ref{sketch} and demonstrated in Section \ref{RNF}. \\ 
Of course the open set $ \Vc_\eps$ is defined in a rather complex way through non-resonant relationships between actions $|u_a|^2$ and $\eps$ (see section \ref{NR}). In particular it does not contain $u\equiv 0$ which is too resonant. Its construction relies on a {\em ultra-violet cut-off} as in classical KAM theory (particularly in \cite{Arn63}), here in an infinite dimensional setting. Moreover, these sets are invariant by angular rotation in the sense that 
\begin{equation}
\label{JeSuisUnCylindre}
\forall\, (\theta_a)_{a \in \Z} \in \R^\Z,\quad 
u = (u_a)_{a\in \Z} \in \Vc_\varepsilon
\Longrightarrow
(e^{i \theta_a}u_a)_{a\in \Z} \in \Vc_\varepsilon. 
\end{equation}
It is then necessary to show that the flow travels within these open sets. This is achieved in a second step. 

$\bullet$ {\bf Generic almost preservation of the actions over very long time} (see Theorems \ref{clouzot} and \ref{gilles}). 
For a given $\varepsilon > 0$, we set 
\begin{equation}
\label{initdata}
u(0,x) =   u^0_\varepsilon(x) = c \varepsilon \sum_{a \in \Z} \sqrt{I_a} e^{ia x}, 
\end{equation}
where $I = (I_a)_{a\in \Z}$ are random variables with support included in the interval $(0,\langle a \rangle^{-2s + 4})$, so that $u_\varepsilon^0$ belong to the space $\ell_s^1$ and $c= (2\pi)^{-1} \tanh \pi$ is a normalizing constant to ensure $\| u^0_\varepsilon\|_{\ell^1_s} < \varepsilon/2$ almost surely.  
We prove that under some assumptions on the law of $I_a$, then for essentially almost all couple $(\varepsilon,I)$ and $\varepsilon$ small enough, the initial values $u^0_\varepsilon$ of the form \eqref{initdata} are in the domain of definition of the normal forms, and thus have a dynamics that is essentially an integrable one over very long time. This implies the almost preservation of the actions $|u_a(t)|^2$ over  times of order $\varepsilon^{-M}$ with $M$ arbitrary which in turn implies that the solution remains inside the open set $ \Vc_\eps$ . In particular we deduce the almost preservation of the Sobolev norm of the solution over  times of order $\varepsilon^{-M}$, i.e. property \eqref{star}. This second step is detailed in section \ref{dyn}.

\medskip 

We  show a difference between \eqref{nls} and \eqref{nlsp} for which we obtain a stronger result. Indeed, whereas possible resonances between $\varepsilon$ and the actions $I_a$ can appear in \eqref{nls}, this is not the case for \eqref{nlsp} which thus can be seen as a more robust equation than \eqref{nls}.

As previously mentioned, the possibility of obtaining normal forms without the help of external parameters was already known in the KAM theory (see \cite{KP96} and also \cite{EGK16}). However these normal forms were constructed around finite dimensional tori.
The originality of our analysis is that we work with truly infinite dimensional objects. 

The question of building full dimensional invariant tori by using our rational normal forms is under study. 
 
It would also be nice to apply this new normal form technique to other PDEs, especially in higher dimension. Nevertheless, there is an important limitation: we use in an unavoidable way the fact that the dominant term of the non-linearity (the cubic term for \eqref{nls} and \eqref{nlsp}) are completely integrable (they depend only on actions). This is no longer true for the quintic NLS equation (see \cite{GT12}) or  for \eqref{nls} and \eqref{nlsp}) equations in higher dimension. It should be noted that in the case of the beam equation studied in \cite{EGK16}, the cubic term is also not integrable and this does not prevent a KAM-type result from being obtained. But in this case, a finite number of symplectic transformations make it possible to get rid of the angles corresponding to the modes of the finite dimensional torus that is perturbed. In our case, we would need an infinite number of such transformations, which is not accessible because these transformations are not close to identity.  

Finally let us mention two recent results that open new directions in the world of Birkhoff normal forms. In \cite{BD} Berti-Delors have considered recently Birkhoff normal forms for a quasi linear PDE, namely the capillarity-gravity water waves equation, and thus faced unbounded nonlinearity. In this paper capillarity  plays the role of the external parameter. Also in \cite{BMP} Biasco-Masseti-Procesi,  considering a suitable Diophantine condition, prove exponential stability in Sobolev norm for parameter dependent NLS on the circle.

\vspace{2em}
\noindent{\bf Acknowledgments.} During the preparation of this work the three authors benefited from the support of the Centre Henri Lebesgue ANR-11-LABX- 0020-01 and B.G. was supported 
 by ANR -15-CE40-0001-02  ``BEKAM''  and ANR-16-CE40-0013 ``ISDEEC'' of the Agence Nationale de la Recherche.

\section{Statement of the results and sketch of the proof}

\subsection{Main results}

First we introduce the Hamiltonians associated with \eqref{nls} and \eqref{nlsp} written in Fourier variables:
\begin{align}
\label{Hnls} H_{\rm NLS}&= \sum_{a\in\Z}a^2 |u_a|^2+\frac1{2\pi}\int_\T g\Big(\sum_{a,b\in\Z} u_a \bar u_b  e^{i(a - b) x}\Big) \dd x, \quad\mbox{and}   \\
H_{\rm NLSP}&=\sum_{a\in\Z}a^2 |u_a|^2+\frac1{4\pi}\int_{\T} \Big(\sum_{a,b\in\Z} \hat V_{a- b }u_a \bar u_b  e^{i(a - b) x}\Big) \Big(\sum_{a,b\in\Z} \hat u_a \bar u_b  e^{i(a - b) x}\Big) \dd x, 
\end{align}
where the Fourier transform $\hat V_a = a^{-2}$, $a \neq 0$, $\hat V_0 = 0$ is associated with the Green function of the operator $- \Delta$ with zero average on $\T$.

\begin{theorem}[\eqref{nls} and \eqref{nlsp} cases]
\label{mainth} Let $H$ equals $H_{\rm NLS}$ or $H_{\rm NLSP}$.
For all  $r\geq 2$, there exists $s_0\equiv s_0(r)=O(r^2)$ such that for all $s\geq s_0$ the following holds:\\
 There exists $\eps_0\equiv\eps_0(r,s)$ such that for all $\varepsilon < \varepsilon_0$, there exist open sets $\mathcal{C}_{\varepsilon,r,s}$ and $\mathcal{O}_{\varepsilon, r,s}$ included in $B_s(0,4\varepsilon)$ the ball of radius $4\varepsilon$ centered at the origin in $\ell_s^1$, and an analytic canonical and bijective transformation $\tau: \mathcal{C}_{\varepsilon,r,s} \mapsto \mathcal{O}_{\varepsilon,r,s}$ satisfying 
 \begin{equation}\label{estimtau}
\Norm{\tau(z) - z}{s} \leq \varepsilon^{\frac{3}{2}}\quad \forall z\in\mathcal{C}_{\varepsilon, r,s}, 
\end{equation}
that puts  $H$ in normal form up to order $2r$:
$$
H \circ \tau  = Z + R
$$
where 
\begin{itemize}
\item 
$Z=Z(I)$ is a smooth function  of the actions and thus is an integrable Hamiltonian;
\item   the remainder $R$ is of order $\eps^{2r+2}$ on $\mathcal{C}_{\varepsilon,r,s}$, precisely
\begin{equation}\label{XR}\| X_R(z)\|_s\leq \eps^{2r + 1} \quad \text{for all }z\in \mathcal{C}_{\varepsilon,r,s}.\end{equation}
\end{itemize}
\end{theorem}
This theorem is proved in section \ref{RNF}.\\ In section \ref{dyn}, we  prove that for all $\varepsilon>0$, there exists a set of initial data included in $\mathcal{O}_{\varepsilon,r,s}$ on which \eqref{star} holds true: 

\begin{theorem}[\eqref{nls} and \eqref{nlsp} cases] \label{corodyn}Let $H$ and $\eps_0(r,s)$ as in Theorem \ref{mainth} and
 let $u_a(t)$ denotes the Fourier coefficients of the solution $u(t,x)$ of the Hamiltonian system associated with $H$. Then
 for all $\eps<\eps_0$ there exists an open set $\Vc_{\varepsilon,r,s} \subset \mathcal{O}_{\varepsilon, r,s}$ invariant by angle rotation in the sense of \eqref{JeSuisUnCylindre}, such that for all $(u_a(0))_{a\in\Z} \in \Vc_{\varepsilon,r,s}$ we have for all $t \leq \varepsilon^{-2r+1}$
 \begin{align}\label{estimua}\sup_{a \in \Z} \,\langle a  \rangle^{2s} \left| |u_a(t) |^2 - |u_a(0)|^2 \right|  \leq 3 \varepsilon^{\frac52},\\
\label{estimu}(u_a(t))_{a\in\Z} \in \mathcal{O}_{\varepsilon,r,s}\quad  \mbox{and in particular}\quad \Norm{u(t)}{s} \leq  4\varepsilon. 
\end{align}
Furthermore there exists a full dimensional torus $\mathcal T_{0}\in\ell^1_s$ such that for all $r_1+r_2= 2r+2$
 \begin{equation}\label{tore}
\mbox{dist}_{s}(u(t),\mathcal T_{0})\leq C\eps^{r_{1}} \quad 
\mbox{for } \vat \leq 1/\eps^{r_{2}}
\end{equation}
where $\mbox{dist}_{s}$ denotes the distance on $\ell^1_s$ associated with 
the norm $\Norm{\cdot}{s}$
\end{theorem}


The next step is to describe   the non resonant sets $\Vc_{\varepsilon,r,s}$ which, as we said, are open, invariant by rotation (see \eqref{JeSuisUnCylindre}) and included in the ball of $\ell^1_s$ centered at 0 and of radius $\eps$ but does not contain the origin. The following Theorems, proved in section \ref{NR},  show that in both cases these open sets contain {\em many} elements of the form \eqref{initdata} but is much larger in the \eqref{nlsp} case that in the \eqref{nls} case. 

The first result concerns the nonlinear Schr\"odinger case \eqref{nls}. 

\begin{theorem}[\eqref{nls} case]
\label{clouzot}
Let $(\Omega, \Ac,\Pb)$ be a probability space, and let us assume that $I: \Omega \mapsto (\R_+)^\Z$ are random variables  satisfying
\begin{itemize}
\item $(I_a)_{a\in \mathbb{Z}}$ are independent,
\item for each $a\in\Z$, $I_a^2$ is uniformly distributed in  $(0,  \langle a \rangle^{-4s-8} )$, 
\end{itemize}
and let $u_\varepsilon^0$ be the familly of random variables defined by \eqref{initdata}. \\
Let $r$, $s\geq s_0(r)$ and $\eps_0(r,s)$ as  in Theorem\eqref{mainth} for  \eqref{nls}. Then 
\begin{itemize}
\item for all $0\leq\varepsilon<\eps_0(r,s)$
\begin{equation}
\label{probaNLS}
\mathbb{P}\left( u_\varepsilon^0\in  \mathcal{V}_{\varepsilon,r,s} \right) \geq 1 - \eps^\frac{1}{3}.
\end{equation}
\item for all $0\leq\varepsilon<\eps_0(r,s)$ and for all sequence $(x_n)_{n\in \mathbb{N}}$ of random variables uniformly distributed in $(0,1)$ and independent of $I_a$, there is a probability larger than $1-\eps^{\frac{1}{6}}$ to realize $I$ such that there is a probability larger than $1-\eps^{\frac{1}{6}}$ to realize $(\varepsilon_n)_{n\in \mathbb{N}} := (\eps 2^{-(n+x_n)})_{n\in \mathbb{N}}$ such that $u_{\varepsilon_n}^0$ is non-resonant for all $n$  (i.e. $u_{\varepsilon_n}^0 \in \mathcal{V}_{\varepsilon_n,r,s}$). More formally, we have
\begin{equation}
\label{lastone}
\mathbb{P}\left( \mathbb{P} \left( \forall n \in\N, \  u_{\varepsilon_n}^0 \in \mathcal{V}_{\varepsilon_n,r,s} \ | I \right) \geq 1 -\eps^{\frac{1}{6}}  \right) \geq 1 - \eps^{\frac{1}{6}}, 
\end{equation}
where $\mathbb{P}(\, \cdot\,  | I) $ denote the probability conditionally to the distribution $I = (I_a)_{a\in \Z}$
\end{itemize}
\end{theorem}


The first part of the statement  corresponds to fixing an $\varepsilon$ and removing some resonant set of $I_a$ (depending on $\varepsilon$) the second part shows that for a given distribution of $I_a$, we can take a lot of arbitrarily small $\varepsilon$ fulfilling the assumptions of the Theorem. 
Moreover, as the set $\Vc_{\varepsilon,r,s}$ is invariant by angle rotation, then for a given $u_\varepsilon^0 \in \Vc_{\varepsilon,r,s}$ of the form \eqref{initdata}, all the rotated functions of the form \eqref{JeSuisUnCylindre} belong to $\Vc_{\varepsilon,r,s}$. 

The authors would like to mention that \eqref{lastone} corresponds to many numerical experiments confirming the absence of drift over long times when $\varepsilon \to 0$ for a generic initial distribution of $I_a$: in other words this statement correspond to what is generically numerically observed, confirming a sort of {\em generic behaviour} for solutions of \eqref{nls} for which no energy exchanges is observed between the frequencies over very long times. 

\medskip

The corresponding analysis for the Schr\"odinger-Poisson case leads to a better result: 
\begin{theorem}[\eqref{nlsp} case]
\label{gilles}
Let $(\Omega, \Ac,\Pb)$ be a probability space, and let us assume that $I: \Omega \mapsto (\R_+ )^\Z$ are random variables  satisfying
\begin{itemize}
\item $(I_a)_{a\in \mathbb{Z}}$ are independent,
\item for each $a\in\Z$, $I_a$ is uniformly distributed in  $(0,  \langle a \rangle^{-2s-4} )$, 
\end{itemize}
and let $u_\varepsilon^0$ be the familly of random variables defined by \eqref{initdata}. \\
Let $r$, $s\geq s_0(r)$ and $\eps_0(r,s)$ as  in Theorem\eqref{mainth} for  \eqref{nlsp}. Then for all $0\leq\varepsilon<\eps_0(r,s)$
\begin{equation}
\label{probaNLSP}
\quad \mathbb{P}\left(\forall \varepsilon'<\eps,   u_{\varepsilon'}^0\in  \mathcal{V}_{\varepsilon',r,s} \right) \geq 1 - \eps^{\frac{1}{3}}.
\end{equation}

\end{theorem}
This statement allows to take one distribution of the action $I_a$ fulfilling some generic non resonance condition, and then to take arbitrarily small $\varepsilon$ in the initial value \eqref{initdata} independently on the $I_a$. Thus the result for \eqref{nlsp} is much stronger from the point of view of phase space: stable initial distributions are much more likely for \eqref{nlsp} than for \eqref{nls}.

In the remainder of the paper, we will essentially focus on the \eqref{nls} case. The proof of the \eqref{nlsp} case will be outlined in appendix \ref{casNLSP}, where we stress the difference with  \eqref{nls}, which are mostly major simplifications. 

\subsection{Sketch of proof}\label{sketch}

In this section we explain the strategy of the proof and we describe the new mathematical objects needed. The starting point is to write (formally) the Hamiltonian \eqref{HamNLS} as 
$$
H = Z_2(I) + P_4 + \sum_{m\geq3} P_{2m}
$$
where $I$ denote the collection of $I_a = |u_a|^2$, $a \in \Z$, and where  
\begin{equation}
\label{Z2}
Z_2 = \sum_{a \in \Z} (a^2 + \varphi(0)) I_a
\end{equation}
 is the Hamiltonian associated with the linear part of the equation. The Hamiltonians $P_{2m}$ are polynomials of order $2m$ in the variables $u_a$,  and $P_4$ is explicitely given by  
 $$
 P_4 = \frac12 \varphi'(0) \sum_{a + b = c + d} \hat V_{a-c} u_a u_b \bar u_c \bar u_d
 $$
 where $\hat V_a = 1$ in the case of \eqref{nls} and $\hat V_{a} = a^{-2}$, $a \neq 0$, $\hat V_0 = 0$,  in the case of \eqref{nlsp}. 
 The first step is to perform a first resonant normal form transform $\tau_2$ with respect to $Z_2$. This step is classic and can be found for the first time in \cite{KP96}. 
After some iterations, the new Hamiltonian can be written 
$$
H \circ \tau_2 = Z_2(I) + Z_4(I) + Z_6(I) + R_6(u) + \sum_{m = 4}^{r} K_{2m} + R
$$
where $R$ is of order $2r +2$, $K_{2m}$ are polynomial of order $2m$, and $Z_4$ and $Z_6$ are polynomials of degre 4 and 6 containing only actions of the form $I_a$. Moreover, the polynomials $K_{2m}$ are resonant in the sense that they contain only monomials of the form $u_{a_1}\cdots u_{a_m} \bar u_{b_1} \cdots \bar u_{b_m}$ where the collection of indices satisfy the relation 
\begin{equation}
\label{midi}
a_1 + \cdots + a_m = b_1 + \cdots + b_m \quad \mbox{and} \quad   a_1^2 + \cdots + a_m^2 = b_1^2 + \cdots + b_m^2. 
\end{equation}
Indeed, these monomials correspond to the kernel of the operator $\chi \mapsto \{ Z_2, \chi\}$ when applied to polynomials, which is the engine of the construction of the normal form. Note that at this stage, no small divisor problem occur. 
Now  natural idea consists in using the term $Z_4$ to eliminate iteratively the terms of order $2m$ that do not depend on the actions. The general strategy is the following: 

\medskip 
{\em(i)}  Truncate all polynomials $K_{2m}$ and remove all the monomials containing at least three indices of size greater than $N$. The remainder term, as already noticed, see for instance \cite{Bou00,Gre07,BG06}, is thus of order $\varepsilon^{5} N^{-s}$. Moreover, taking into account the resonance condition \eqref{midi}, the remaining truncated monomials have {\em irreducible} part - meaning they do not contain actions - with indices bounded by $\mathcal{O}(N^2)$. Taking $N$ so that $N^{-s} = \mathcal{O}(\varepsilon^r)$ (so typically $N = \varepsilon^{-r/s}$) then ensures that this term will be small enough to control the dynamics over a time of order $\varepsilon^{-r}$ by using a bootstrap argument. 

\medskip
{\em (ii)} Construct iteratively normal form transformation to eliminate the remaining part of degree $2m$ that do not depend only of the action by using the integrable Hamiltonian $Z_4$ which is explicitly given. The engine underlying this construction is thus to solve iteratively homological equations of the form $\{Z_4,\chi\} = Z + K$ where $K$ is given and do not contains terms depending only on the actions. This step makes appear small denominators depending on $I$, so that such a construction makes naturally appear rational functions (see Section \ref{rational}) and not only polynomials. However, the division by small denominators depending on $I$ also yields poles in the normal form transforms. To avoid them, we use generic non resonance conditions on the distribution $I$, and the resolution of the homological equation thus brings loss of derivatives. As the small denominators are generated by irreducible monomials whose modes  are bounded by $N^2$, in the estimates this step results in a loss of order  $N^\alpha$ for some $\alpha>0$ versus a gain of $\varepsilon$ at each step of the normal form construction. 

\medskip
{\em (iii)} Try to trade off a few powers of $\varepsilon$  to control the normal form construction. This is done by a condition of the form $\varepsilon N^{\alpha} < 1$. The heart of the analysis is thus to make $\alpha$ independent of $s$  so that for $s$ large enough, such condition can be satisfied and will be compatible with $N = \varepsilon^{-r/s}$. If this is the case, the Hamiltonian thus depends only on the actions and a small remainder term of order $\varepsilon^r$ which allows to conclude. The remaining difficulty is to handle the algebra of rational functionals, and the control of the non resonant sets after each normal form steps. 

\medskip
With this roadmap in hand, the expression of $Z_4$ is fundamental as it drives the small denominators. Here appears a drastic difference between \eqref{nls} and \eqref{nlsp}. Indeed, the first normal term corresponding to the Hamiltonian $P_4$ is of the form 
\begin{equation}
\label{eqmagic}
Z_4 =  \varphi'(0) \Big( \sum_{a\neq b \in \Z} \hat V_{a-b} I_a I_b + \frac{1}{2} \hat V_{0} \sum_{a \in \Z}  I_a^2 \Big)
\end{equation}
The frequencies associated with this integrable Hamiltonian are of the form 
$$
\lambda_a (I) = \frac{\partial Z_4}{\partial I_a} = \varphi'(0) \Big( 2 \sum_{b \neq a \in \Z} \hat V_{a-b}  I_b + \hat V_0 I_a \Big). 
$$
We thus observe that for \eqref{nlsp} for which $\hat V_a = a^{-2}$, $a\neq 0$ and $\hat V_0 = 0$, if $u \in \ell_s^1$ with large $s$, these frequencies are essentially dominated by low modes $I_b$ at a scale $\langle a \rangle^{-2}$. Hence it is easy to prove that the small denominators associated with $Z_4$ which are linear combinations of the $\lambda_a(I)$ are generically non resonant, with a loss of derivative independent of $s$. 

Hence for \eqref{nlsp}, we can work through the previous programme, and the coefficients $\alpha$ in the condition $\varepsilon N^{\alpha} < 1$ will indeed be independent of $s$. 

\medskip 
For \eqref{nls} (for which $\hat V_a = 1$), the situation is much worse. Indeed the previous frequency degenerate to 
$$
\lambda_a (I) = \varphi'(0) \Big( 2 \sum_{b \neq a \in \Z}  I_b  +  I_a\Big)   = \varphi'(0) \Big( 2 \Norm{u}{L^2}^2 - I_a\Big). 
$$
We thus see that the small denominators associated with $Z_4$ are of the form 
\begin{equation}
\label{omegaI}
\omega(I) =  \varphi'(0) ( I_{a_1} + \cdots + I_{a_m} - I_{b_1} - \cdots I_{b_m}), 
\end{equation}
with $a\cap b=\emptyset$ and as $u$ is in $\ell_s^1$, the $I_a$ are of order $\varepsilon^2 \langle a \rangle^{-2s}$. Hence the natural non resonant condition (that is generic in $I$) takes the form 
\begin{equation}
\label{nlsbad}
| I_{a_1} + \cdots + I_{a_m} - I_{b_1} - \cdots I_{b_m}|\geq \gamma \varepsilon^{-2} \Big( \prod_{n = 1}^{m} \langle a_n \rangle \langle b_n \rangle \Big)^{-2} \langle \mu_{\min}\rangle^{2s},  
\end{equation}
where $\mu_{\min}$ denote the smallest index amongst the $a_n$ and $b_n$. Such a condition was used in \cite{Bou00}. We  see that to run through the previous programme, we have to {\em distribute} the derivative of order $2s$ associated with the lowest index of the irreducible parts of the monomials, coming at each step of the normal form construction. 

Unfortunately, such a distribution cannot be done straightforwardly. One of the reason is the presence of the remaining terms $Z(I)$ depending on the actions in the process. Indeed, take a monomial of the form  $ f(I) \prod_{n = 1}^m u_{a_n} \bar u_{b_n} $, where $f$ depends only on the actions associated with low modes, the remaining part being irreducible. These terms will enter into the normal form construction first as right-hand side of the homological equation, in which case they will be divided by $\omega(I)$ defined in \eqref{omegaI}, and then will contribute to the higher order terms by Poisson bracket with the other remaining terms. Now take some $Z$ previously constructed (for example $Z_6(I)$). New terms will enter into the next homological right hand side that are made of Poisson brackets between this term $Z$ and the constructed functional. Amongst the new term to solve, we will have terms of the form 
$ f_1(I) \prod_{n = 1}^m \omega(I)^{-1}u_{a_n} \bar u_{b_n} $ where $f_1$ depends again only on low modes. At this stage, it will be possible to distribute the derivative on the irreducible monomials, but by iterating, we see that at each resolution of the homological equation, we will have to divide by the {\em same} small denominator. After $p$ such iterations, we will end up with monomials of the form $ f_p(I) \prod_{n = 1}^m \omega(I)^{-p}u_{a_n} \bar u_{b_n} $ where $f_p(I)$ depends on low modes, and for $p > m$, we will not be able to control this terms independently of $s$. Hence the previous procedure cannot be applied. 

To remedy this difficulty, a natural idea (coming from KAM strategy) is to include the term $Z_6$ in the normal form construction, that is to solve at each step the homological equation with $Z_4 + Z_6$. 
Nevertheless, this trick brings good and bad news: \\
-The bad news is that the frequencies associated with the Hamiltonian $Z_4 + Z_6$ are not perturbations of the frequencies of $Z_4$. 
We can even show that for a given generic distribution of the $I_a$, there are $\varepsilon$ producing resonances for the Hamiltonian $Z_4 + Z_6$ while $Z_4$ is non resonant. \\
-The good news is that the structure of $Z_6$ ressembles the structures of the Hamiltonian of \eqref{nlsp} with a similar convolution potential coming from the first resonant normal form done with the Laplace operator. In other words, $Z_6$ is much less resonant than $Z_4$, and has frequencies that satisfy generic non resonance conditions with loss of derivatives independent of $s$. 

The strategy of proof is thus to apply the previous programme with $Z_4 + Z_6$ instead of $Z_4$ alone, after having taking care of the genericity condition on the initial data that have to link now the distribution of the $I_a$ and $\varepsilon$. This explain the major difference between the statement for \eqref{nls} and \eqref{nlsp}.  The main drawback is that by doing so, we break the natural homogeneity in $\varepsilon$ which yields some specific technical difficulties, in particular in the definition of a class of rational functions, which must be stable by Poisson bracket and solution of homological equations, while preserving the asymptotic in $\varepsilon$. 



\section{General setting}

\subsection{Hamiltonian formalism}

Let $\U_2=\{\pm 1\}$. We identify a pair $(\xi,\eta)\in \C^{\Z} \times \C^{\Z}$ with 
$(z_j)_{j \in \U_2 \times \Z} \in \C^{\U_2 \times \Z}$ via the formula
\begin{equation}
\label{Ezj}
j = (\delta,a) \in \U_2 \times \Z  \Longrightarrow 
\left\{
\begin{array}{rcll}
z_{j} &=& \xi_{a}& \mbox{if}\quad \delta = 1,\\[1ex]
z_j &=& \eta_a & \mbox{if}\quad \delta = - 1.
\end{array}
\right.
\end{equation}
We denote by $z = (\xi,\eta)$ such an element and we endow this set of sequences with the $\ell^1_s$ topology:
$$\ell^1_s = \ell_s^1(\Z,\C)^2 =   \{z\in \C^{\U_2 \times \Z}\mid \Norm{z}{s}<\infty\} 
$$
where\footnote{Here for $j=(\delta,a)$ we set $ \langle j \rangle=(1+a^2)^{1/2}= \langle a \rangle$.} 
$$
\Norm{z}{s} := \sum_{j \in \U_2 \times\mathbb{Z}} \langle j \rangle^s |z_j| = \sum_{a \in \mathbb{Z}} \langle a \rangle^s |\xi_a| + \sum_{a \in \mathbb{Z}} \langle a \rangle^s |\eta_a|.
$$
We associate with $z$ two complex functions on the torus $u$ and $v$ through the formulas 
\begin{equation}
\label{Exieta}
u (x)= \sum_{a\in \Z} \xi_{a} 
e^{i a x} \quad\mbox{and}\quad
 v (x)= \sum_{a\in \Z}
\eta_{a} 
e^{-i a x},\quad x\in\T.
\end{equation}
We say that $z$ is {\em real} when $z_{\overline{j}} = \overline{z_j}$ for any $j\in \U_2\times\Z$. In this case $v$ is the complex conjugate of $u$: $v(x)=\overline{u(x)},\ x\in\T$, and the definition of $\ell_s^1$ coincides with \eqref{ElleLaisseUn}.  

\begin{remark}
The sequences spaces $\ell^1_s$, which are in fact Besov spaces, are not perfectly adapted to  Fourier analysis: when $z\in\ell^1_s$ with $s\geq 0$ then the functions $u$ and $v$ belong to the Sobolev space $H^s(\T)$ while when $u$ and $v$ belong to $H^s(\T)$ then its sequence of Fourier coefficients $z$ belongs to $\ell^1_{s-\eta}$ only for $\eta>1/2$. This lost of regularity would not happen in  the Fourier  space $\ell^2_s$ nevertheless we prefer $\ell^1_s$ because it leads to simpler estimates of the flows (see for instance Proposition \ref{polyflow}). Anyway the results we obtain thus lead to control of Sobolev norms $H^{s-\frac12^{+}}$ over long times. \end{remark}


We endow $\ell^1_s$ with the symplectic structure 
\begin{equation}\label{Esymp}-i\sum_{a\in \Z}dz_{(+1,a)}\wedge dz_{(-1,a)}=-i\sum_{a\in\Z}d\xi_a\wedge d\eta_a.
\end{equation}

For a function $F$ of $\mathcal{C}^1(\ell^1_s,\C)$, we define its Hamiltonian vector field by $X_F=J\nabla F$ where  $J$ is the symplectic operator induced  by the symplectic form \eqref{Esymp},
$
\nabla F(z) = \left( \frac{\partial F}{\partial z_j}\right)_{j \in \U \times\Z}$, 
and by definition we set for $j = (\delta,a) \in \U_2 \times\Z $, 
$$
 \frac{\partial F}{\partial z_j} =
  \left\{\begin{array}{rll}
 \displaystyle  \frac{\partial F}{\partial \xi_a} & \mbox{if}\quad\delta = 1,\\[2ex]
 \displaystyle \frac{\partial F}{\partial \eta_a} & \mbox{if}\quad\delta = - 1.
 \end{array}
 \right.
$$
So  $X_F=J\nabla F$ reads in coordinates
$$
 (X_F)_j =
  \left\{\begin{array}{rll}
 \displaystyle  i\frac{\partial F}{\partial \eta_a} & \mbox{if}\quad\delta = 1,\\[2ex]
 \displaystyle -i\frac{\partial F}{\partial \xi_a} & \mbox{if}\quad\delta = - 1.
 \end{array}
 \right.
$$
For two functions $F$ and $G$, the Poisson Bracket is (formally) defined as
\begin{equation}\label{poisson}
\{F,G\} = \langle\nabla F, J \nabla G\rangle = i \sum_{a \in \Z} \frac{\partial F}{\partial \xi_a}\frac{\partial G}{\partial \eta_a} -   \frac{\partial F}{\partial \eta_a}\frac{\partial G}{\partial \xi_a}  
\end{equation}
where $\langle\cdot,\cdot\rangle$ denotes the natural bilinear pairing: $\langle z,\zeta\rangle=\sum_{j \in \U_2 \times\Z}z_j\zeta_j$.\\
We say that a Hamiltonian function $H$ is 
 {\em real } if $H(z)$ is real for all real $z$. 
\begin{definition}\label{def:2.1}
 For a given $s\geq0$ and a given open set $\Uc$  in $\ell_s^1$, we  denote by $\Hc_s(\Uc)$ the space of real Hamiltonians $P$ satisfying 
$$
P \in \mathcal{C}^{1}(\Uc,\C), \quad \mbox{and}\quad 
X_P \in \mathcal{C}^{1}(\Uc,\ell^1_s). 
$$
\end{definition}
We will use the shortcut  $F\in \Hc_s$ to indicate that there exists an open set $\Uc$  in $\ell_s^1$ such that $F\in\Hc_s(\Uc)$.\\
 Notice that for $F$ and $G$ in $\Hc_s(\Uc)$ the formula \eqref{poisson} is well defined. 
\subsection{Hamiltonian flows}
With a given Hamiltonian function $H \in \Hc_s$, we associate the Hamiltonian system
$$
\dot z =   X_H(z) =    J \nabla H(z), 
$$
which also reads in coordinates
\begin{equation}\label{Eham2}
\displaystyle \dot \xi_a = i\frac{\partial H}{\partial \eta_a}  \quad \mbox{and} \quad 
  \dot \eta_a= -i\frac{\partial H}{\partial \xi_a},  \quad a\in \Z.
\end{equation}
Concerning the Hamiltonian flows we have
\begin{proposition}\label{prop3.3} Let $s\geq 0$. Any Hamiltonian in $\Hc_s$ defines a local flow in $\ell^1_s$ which preserves the reality  condition, i.e. if the initial condition $z = (\xi,\bar \xi)$  is real, the flow $(\xi(t),\eta(t)) = \Phi_H^t(z)$ is also real, $\xi(t) = \overline{\eta(t)}$ for all $t$. 
\end{proposition} 
\proof
The existence of the local flow is a consequence of the Cauchy-Lipschitz theorem.\\ Furthermore let us denote by $f$ the $\mathcal{C}^1$ function defined by
$\ell_s^1(\Z,\C)\ni\xi\mapsto H(\xi,\bar \xi)-\overline{H(\xi,\bar \xi)}$. Since $H$ is real we have $f\equiv0$ and thus its differential at any point $\xi\in\ell_s^1(\Z,\C)$ and in any direction $\zeta\in\ell_s^1(\Z,\C)$ vanishes\footnote{Here by a slight abuse of notation $\langle\cdot,\cdot\rangle$ denotes the bilinear pairing in $\ell_s^1(\Z,\C)$.}: 
\begin{align*}0\equiv Df(\xi)\cdot\zeta&= \langle\nabla_\xi H(\xi,\bar\xi),\zeta\rangle+\langle\nabla_\eta H(\xi,\bar\xi),\bar\zeta\rangle\\
&-\overline{\langle\nabla_\xi H(\xi,\bar\xi),\zeta\rangle}-\overline{\langle\nabla_\eta H(\xi,\bar\xi),\bar\zeta\rangle}\\
&=\langle\nabla_\xi H(\xi,\bar\xi)-\overline{\nabla_\eta H(\xi,\bar\xi)},\zeta\rangle+\langle\nabla_\eta H(\xi,\bar\xi)-\overline{\nabla_\eta H(\xi,\bar\xi)},\bar\zeta\rangle.
\end{align*}
Therefore $\nabla_\xi H(\xi,\bar\xi)-\overline{\nabla_\eta H(\xi,\bar\xi)}$ for all $\xi\in\ell_s^1(\Z,\C)$ and the system \eqref{Eham2} preserves the reality condition.
\endproof

In this setting Equations \eqref{nls} and \eqref{nlsp} are equivalent to  Hamiltonian systems associated with the real Hamiltonian function
\begin{equation}
\label{H} 
H(\xi,\eta)=\sum_{a\in\Z}a^2 \xi_a\eta_a+P(z)
\end{equation}
where
\begin{equation}\label{P}
P (\xi,\eta) =\frac1{2\pi}\int_\T g\Big(\sum_{a,b \in \Z}   \xi_{a}  \eta_b 
e^{i (a-b) x}\Big)\dd x,\end{equation}
in the \eqref{nls} case, 
and 
\begin{equation}\label{PNLSP}
P (\xi,\eta) =\frac1{4\pi}\int_\T \Big(\sum_{a,b \in \Z} \hat  \xi_{a}  \eta_b 
e^{i (a-b) x}\Big) \Big(\sum_{a,b \in \Z} \hat V_{a-b} \xi_{a}  \eta_b 
e^{i (a-b) x}\Big)\dd x,\end{equation}
 in the \eqref{nlsp} case, where we recall that $\hat V_a = a^{-2}$ for $a \neq 0$ and $\hat V_0 = 0$. 
We first notice that in both cases, $P$
belongs to $\Hc_s$, in fact we have:
\begin{lemma} Let $s\geq0$ and let $z\mapsto f(z)\in\C$ and $z \mapsto h(z)$ be two analytic functions defined on a neighborhood $\Vc$ of the origin in $\C$ that takes real values when $z$ is real. Then the formulas
\begin{align}\label{QQ}
P (\xi,\eta) &=\frac1{2\pi}\int_\T f\Big(\sum_{a,b \in \Z}   \xi_{a}  \eta_b 
e^{i (a-b) x}\Big)\dd x,\\\label{PP}
Q(\xi,\eta)&=\int_\T h \Big(\sum_{a,b \in \Z} \hat V_{a-b} \xi_{a}  \eta_b 
e^{i (a-b) x}\Big) f\Big(\sum_{a,b \in \Z}  \xi_{a}  \eta_b 
e^{i (a-b) x}\Big) \dd x
\end{align}
define  Hamiltonian $P$ and $Q$ belonging to $\Hc_s(\Uc)$ where $\Uc$ is some neighborhood of the origin in $\ell_s^1$.
\end{lemma}
\proof First we verify that \eqref{QQ} and \eqref{PP} define regular maps on $\ell_s^1$.\\
By definition, $\ell_s^1=\ell_s^1(\Z,\C)^2$ and the Fourier transform $\xi \mapsto \sum_{a \in \Z} \xi_a e^{ia x}$ defines an isomorphism between $\ell_s^1(\Z,\C)$ and a subset of $L^2(\R,\C)$ that we still denote by $\ell^1_s(\Z,\C)$. Moreover, for $u,v \in \ell_s^1(\Z,\C)$, we have $\Norm{u v}{s} \leq \Norm{u}{s}\Norm{v}{s}$ and thus the mapping $(u,v)\mapsto uv$ is analytic from $\ell_s^1(\Z,\C)^2$ to $\ell_s^1(\Z,\C)$. 
Extending this argument, if $h: \C \to \C$ is analytic in a neighborhood of the origin, the application $\xi \mapsto u(x) \mapsto h(u(x))$ is analytic from a neighborhood of the origin in $\ell_s^1(\Z,\C)$ into $\ell_s^1(\Z,\C)$. \\
 Through the identification $(\xi,\eta) \mapsto (u(x),v(x))$, see \eqref{Exieta} the Hamiltonian $Q$ reads 
$$
Q(\xi,\eta) = \int_\T h \Big( V \star (u(x)v(x)) \Big) f\Big( u(x) v(x) \Big) \dd x
$$
Since the mapping $u \mapsto V \star u$ is analytic  on $\ell_s(\Z,\C)$, we conclude that $Q$ is an analytic function from a neighborhood $\Uc$ of the origin in $\ell_s^1(\Z,\C)$ into $\C$. Similar arguments  apply to $P$.\\
Next we verify that  $X_P$ and $X_Q$ are still regular\footnote{The analyticity of $P$ only insure that $X_P$ belongs to the dual of $\ell_s^1(\Z,\C)$.} function from $\Uc$ into $\ell_s^1(\Z,\C)$. 
We focus on $P$ but similar arguments apply to $Q$. 
We have 
$$\frac{\partial Q}{\partial \xi_a}= \int_\T \ k(x) e^{iax}dx$$
with $$k(x)=f'\Big(\sum_{a,b \in \Z}   \xi_{a}  \eta_b 
e^{i (a-b) x}\Big) \sum_{b\in \Z}
\eta_{b} 
e^{-i b x}=f'(u(x)v(x)) v(v).$$
Expanding $f'$ in entire series  we rewrite $\frac{\partial Q}{\partial \xi_a}$ in a convergent sum of terms of the form
$$c_{k}\int_\T (u(x))^{k_1}(v(x))^{k_2}e^{iax}dx,$$
 i.e. the convolution product of $k=k_1+k_2$ sequences in $\ell_s^1$. Then the conclusion follows from the fact that for any $s\geq0$
$$\ell_s^1 \star \ell_s^1\subset \ell_s^1.$$
%
%
\endproof

On the contrary the quadratic part of $H$, $\sum_{a\in \Z}a^2 \xi_a\eta_a$, corresponding to the linear part of \eqref{nls} does not belong to $\Hc_s$. Nevertheless it generates a continuous flow which maps $\ell^1_s$ into $\ell^1_s$ explicitly given for all time $t$ and for all indices $a$ by  $\xi_a(t)=e^{-ia^2 t}\xi_a(0)$, $\eta_a(t)=e^{ia^2 t}\eta_a(0)$. Furthermore this flow has the group property.  By standard arguments (see for instance \cite{Caz03}), this is enough to define the local flow of $\dot z =X_H(z)$ in $\ell_s^1$:
\begin{proposition} Let $s\geq0$.
Let $H$ be the NLS Hamiltonian  defined by \eqref{H} and $z_0\in\ell_s^1$ a sufficiently small initial datum. Then the Hamilton equation 
$$\dot z(t) =X_H(z(t)),\quad z(0)=z_0$$
admits a local solution $t\mapsto z(t)\in\ell_s^1$  which is real if $z_0$ is real.
\end{proposition}
The reality of the flow is proved as in the proof of Proposition \ref{prop3.3}.

\subsection{Polynomial Hamiltonians}

For $m\geq1$  we define  three nested subsets of  $(\U_2 \times \Z)^m$ satisfying zero momentum conditions of increasing order:
\begin{align}
\label{zeromomenta}
\Zc_m &= \{\jb = (\delta_\alpha, a_\alpha)_{\alpha=1}^m \mid  \sum_{\alpha = 1}^{2m} \delta_\alpha  = 0\}, \nonumber\\
\Mc_m &= \{\jb \in \mathcal{Z}_m\mid \sum_{\alpha = 1}^{2m} \delta_\alpha a_\alpha = 0\}, \nonumber\\
\Rc_m &= \{\jb \in \mathcal{M}_m\mid\sum_{\alpha = 1}^{2m} \delta_\alpha a_\alpha^2 = 0\}.\end{align}
We set $\mathcal{Z} = \mathop{\bigsqcup}\limits_{m \geq 0} \mathcal{Z}_m$, $\mathcal{M} = \mathop{\bigsqcup}\limits_{m \geq 0} \mathcal{M}_m$ and  $\mathcal{R} = \mathop{\bigsqcup}\limits_{m \geq 0} \mathcal{R}_m$.  \\
For $\jb\in\Zc$, $\irr(\jb)$ denotes the {\it irreductible} part of $\jb$, i.e.  a subsequence of maximal length $(j'_1,\ldots,j'_{2p})$ containing no actions in the sense that if $j'_\alpha \neq \bar j'_\beta$ for all $\alpha,\beta=1,\cdots, 2p$. \\
We set $$\Ic= \irr(\Rc) = \{\irr(\jb) \, | \, \jb \in \Rc\} .$$
We will use indices $\kbsf$ belonging to  $ \mathop{\bigsqcup}\limits_{p\in \mathbb{N}} \Ic^p$, i.e.  $\kbsf = (\kb_1,\ldots,\kb_p)$ for some $p \in \N$ with $\kb_\alpha \in \Ic$. We denote $\length\kbsf = p$.
We use the convention $\kbsf = \emptyset$ for $ p = 0$.  \\
For $\jb \in \Zc_m$, we set $z_{\jb} = z_{j_1}\cdots z_{j_{2m}}$. We also denote by $\overline \jb = (\bar j_1,\ldots,\bar j_{2m})$, and we notice that when $z$ is real, we have $\overline{z_\jb} = z_{\overline \jb}$. 
\begin{definition}\label{polydef}
We say that $P(z)$
 is a homogeneous polynomial of order $m$ if it can be written 
 \begin{equation}
 \label{poly}
 P(z)=P{[c]}(z) = \sum_{\jb \in \Mc_m} c_\jb z_\jb, \quad\mbox{with}\quad c=(c_\jb)_{\jb\in\Mc_m} \in \ell^\infty(\Mc_m),  
 \end{equation}
 and such that the coefficients $c_\jb$ satisfy $c_{\overline \jb} = \overline{c_{\jb}}$. 
 \end{definition}
 Note that the last condition ensures that $P$ is real, as the set of indices are invariant by the application $\jb \mapsto \overline \jb$.  
 Following \cite{FG10} we easily get
 \begin{proposition}
\label{polyflow}
Let  $s \geq 0$. A homogeneous polynomial, $P{[c]}$, of degree $m\geq2$ belongs to $ \Hc_s(\ell^1_s)$ and we have
\begin{equation}
\label{Echamp}
  \Norm{X_{P{[c]}}(z) }{s} \leq 2 m \Norm{ c}{\ell^\infty}\Norm{z}{s}^{2m-1}. 
\end{equation}
Furthermore  for two homogeneous polynomials, $P{[c]}$ and $P{[c']}$, of degree respectively $m$ and $n$, the Poisson bracket  is a homogeneous polynomial of degree $m+n-1$, $\{P{[c]},P{[c']}\}=P{[c'']} $  and we have the estimate
\begin{equation}
\Norm{ c''}{\ell^\infty} \leq 2 m n \Norm{ c}{\ell^\infty} \Norm{ c'}{\ell^\infty}
\label{Ebrack}
.\end{equation}
\end{proposition}


 For $j = (\delta,a)\in \U_2\times\Z$, we set $I_j=I_a = \xi_a \eta_a=z_j z_{\bar j}$ the action of index $a$.  $I = (I_a)_{a \in \Z}$ denote the set of all the actions. \\
 We note that for a real $z$ in $\ell_s^1$ we have $\Norm{z}{s}=\sum_a  \langle a \rangle^sI_a^{1/2}$. Therefore, an integrable Hamiltonian, i.e. a Hamiltonian function depending only on the actions has  a flow which   leaves invariant each  $\|\cdot\|_s$ norm. \\
We introduce $3$ integrable polynomials that will be used later (see Theorem \ref{BNF}): 
\begin{align}
 Z_2(I) &= \sum_{a \in \Z} (a^2+\varphi(0)) I_a,\nonumber\\
\label{Z4} Z_4(I) &= \varphi'(0) \left(\sum_{a \in \Z} I_a \right)^2  - \frac12 \varphi'(0)\sum_{a \in \Z} I_a^2 ,\\ 
\label{Z6}Z_6(I) &= -\frac{1}{2}\varphi'(0)^2\sum_{a\neq b\in\Z} \frac{1}{(a-b)^2}\ I_a^2I_b \\ \nonumber
&+\frac{\varphi^{''}(0)}{6}\left(6(\sum_{a\in\Z}I_a)^3-9(\sum_{a\in\Z} I_a^2)(\sum_{a\in\Z}I_a)+4\sum_{a\in\Z}I_a^3  \right). 
\end{align}
The first one is the quadratic part of the NLS Hamiltonian,
 the second one is the quartic part and the third one contains the effective terms of the sextic part (see Theorem \ref{BNF}). 
 
 We note that $Z_4$ and $Z_6$ are polynomials of degree $2$ and $3$ in the sense of Definition \ref{polydef}, and thus  define  Hamiltonians in $\Hc_s$ for all $s\geq0$.

\section{Non-resonance conditions}\label{NR}

In this section we discuss the control of small denominators corresponding to the the previous integrable Hamiltonians. We also give results allowing to control them, and show the probability estimates associated with non resonant sets. 

\subsection{Small denominators}
For $\jb = (j_1,\ldots,j_{2m}) \in \U_2\times \Z$, if $j_\alpha = (\delta_\alpha,a_\alpha)$ for $\alpha = 1,\ldots,2m$, we set 
$$
\Delta_{\jb} = \sum_{\alpha = 1}^{2m} \delta_\alpha a_\alpha^2, \quad 
\omega_{\jb}(I) = \sum_{\alpha = 1}^{2m} \delta_\alpha \frac{\partial Z_4}{\partial I_{a_\alpha}} (I)
$$
and
$$
\Omega_{\jb}(I) = \omega_{\jb}(I) + \sum_{\alpha = 1}^{2m} \delta_\alpha \frac{\partial Z_6}{\partial I_{a_\alpha}} (I). 
$$
With these notations, we have for $\jb \in \Zc$, owing to the fact that $\omega_{\jb} = \omega_{\irr(\jb)}$ and $\Omega_{\jb} = \Omega_{\irr(\jb)}$, 
\begin{equation} 
\label{ravel}
\{ Z_4,z_\jb\} = i \omega_{\irr(\jb)}(I) z_{\jb}\quad\mbox{and}\quad \{ Z_4 + Z_6,z_\jb\} = i \Omega_{\irr(\jb)}(I) z_{\jb}.
\end{equation}
Note also that 
\begin{equation} 
\label{ravelo}
\{ Z_2,z_\jb\} = i \Delta_\jb  z_{\jb}, 
\end{equation}
and that $|\Delta_{\jb}| \geq 1$ except when $\jb \in \Rc$ for which $\Delta_\jb = 0$.

For $\jb \in \Ic$, we have the expression 
\begin{equation}
\label{brassens}
\omega_{\jb}(I) = - \varphi'(0) \sum_{\alpha = 1}^{2m} \delta_\alpha  I_{a_\alpha}, 
\end{equation}
as the first term in \eqref{Z4} do not contribute using the relation $\sum_{\alpha = 1}^{2m}{\delta_{\alpha}} = 0$. 

We also introduce the following denominator: 
$$
\widetilde \Omega_{\jb}(I) = - \varphi'(0) \sum_{\alpha = 1}^{2m} \delta_\alpha I_{a_{\alpha}}     - \frac{1}{2}\varphi'(0)^2 \sum_{\alpha = 1}^{2m }  \delta_\alpha \sum_{\substack{b \in\Z \\ b \neq a_{\beta}, \beta = 1,\ldots,2m}} \frac{I_b^2 }{(a_\alpha-b)^2} .  
$$

The following lemma allows to control the evolution of the small denominators when moving the coordinates.

\begin{lemma}
Let $r$ and $s$ be given. The exists a constant $C$ such that for all 
 $\jb \in \Ic$ with $\sharp \jb \leq 2r$ and all $z \in \ell_s^1$, we have 
\begin{equation}
\label{pim}
|\widetilde \Omega_{\jb}(I) -\Omega_{\jb}(I)|\leq C \langle \mu_{\min}(\jb)\rangle^{-2s} \Norm{z}{s}^4. 
\end{equation}
Moreover, let $z,z' \in \ell_s^1$ be associated with the actions $I$ and $I'$, and let $h = \max(\Norm{z}{s}, \Norm{z'}{s})$. Then  we have 
\begin{equation}
\label{pam}
|\omega_{\jb}(I) -\omega_{\jb}(I')|\leq C  \Norm{z - z'}{s}  \langle \mu_{\min}(\jb)\rangle^{-2s} h . 
\end{equation}
and
\begin{equation}
\label{poum}
|\Omega_{\jb}(I) -\Omega_{\jb}(I')|\leq C \Norm{z - z'}{s} \Big(\langle \mu_{\min}(\jb)\rangle^{-2s}   h +   h^3\Big)
\end{equation}
\end{lemma}
\begin{proof}
Along the proof, $C$ will denote a constant depending on $r$, $s$ and derivatives of the function $\varphi$ at $0$. 
Let us denote $\jb = (j_1,\ldots,j_{2m}) \in \Zc_m$ with $m \leq r$ and $j_\alpha = (\delta_\alpha,a_\alpha) \in \mathbb{U}_2 \times \Z$ for $\alpha = 1,\ldots,2m$. 
We calculate that 
\begin{multline*}
\Omega_{\jb}(I)- \widetilde \Omega_{\jb}(I) =  - \frac{1}{2}\varphi'(0)^2 
\sum_{\substack{\alpha,\beta = 1 \\ \beta \neq \alpha}}^{2m} \delta_{\alpha} \frac{I_{a_{\beta}}^2 }{(a_\alpha- a_{\beta})^2} 
- \varphi'(0)^2  \sum_{\alpha = 1}^{2m} \sum_{b \neq a_\alpha} \delta_\alpha\frac{I_{a_\alpha} I_b}{(a_\alpha - b)^2}\\[2ex]
+ \varphi''(0) \sum_{\alpha = 1}^{2m} \delta_{a_\alpha}\left(-3 I_{a_\alpha}(\sum_{a\in\Z}I_a)   +2 I_{a_\alpha}^2  \right). 
\end{multline*}
We see that the first term can be controlled by 
$$
C \sum_{\alpha = 1 }^{2m} I_{a_{\alpha}}^2 \leq C \langle  \mu_{\min}(\jb)\rangle^{-4s} \sup_{\alpha = 1,\ldots,2m } \,  \langle j_\alpha \rangle^{4s} I_{a_{\alpha}}^2 \leq C \langle \mu_{\min}(\jb)\rangle^{-4s} \Norm{z}{s}^4, 
$$
and the second by 
$$
C \sum_{\alpha = 1 }^{2m} I_{a_{\alpha}} \sum_b I_b \leq C \Norm{z}{L^2}^2  \langle \mu_{\min}(\jb)\rangle^{-2s} \sup_{\alpha }\,  \langle j_\alpha \rangle^{2s}  I_{a_{\alpha}} \leq C \langle \mu_{\min}(\jb)\rangle^{-2s} \Norm{z}{s}^4. 
$$
The estimate for the third term is the same, which shows \eqref{pim}. 

Now to prove \eqref{pam}, we have using the expression \eqref{brassens}, 
$$
|\omega_{\jb}(I) -\omega_{\jb}(I')| \leq C \sum_{\alpha = 1}^{2m} |I_{a_\alpha} - I_{a_\alpha}'| \leq C \langle \mu_{\min}(\jb)\rangle^{-2s}  \max_{\alpha = 1}^{2m} \Big(\langle j \rangle^{2s}|I_{a_\alpha} - I_{a_\alpha}'| \Big). 
$$
We obtain \eqref{pam} by noticing that 
\begin{multline*}
\langle j \rangle^{2s}|I_{a_\alpha} - I_{a_\alpha}'| = \langle j \rangle^{s}\left|\sqrt{I_{a_\alpha}\phantom{'}} - \sqrt{I_{a_\alpha}'}\right| \langle j \rangle^{s} \left|\sqrt{I_{a_\alpha}\phantom{'}} + {\sqrt{I_{a_\alpha}'}} \right| \\ 
\leq  \Norm{|z| - |z'|}{s} ( \Norm{z}{s} + \Norm{z'}{s}) \leq 2 \Norm{z - z'}{s}  \max(\Norm{z}{s}, \Norm{z'}{s}). 
\end{multline*}
The proof of \eqref{poum} is then easily obtained by using the previous result, and explicit expressions of $\frac{\partial Z_6}{\partial I_{a_\alpha}}$ showing that this function is homogeneous of order $4$ and thus locally Lipschitz in $z$ with Lipschitz constant of order $h^3$ on balls of size $h$ in $\ell_s^1$ as can be seen by using estimates similar to the previous one. 
\end{proof}

\subsection{Non resonant sets}
As usual we have to control the small divisors, this will be the case for $z$ belonging the following  non resonant sets:
\begin{definition}\label{defnonres} Let $\eps,\gamma >0$, $r\geq1$ and $s\geq0$,
we say that  $z\in \ell_s^1$ belongs to the non resonant set $ \mathcal{U}_{\gamma,\varepsilon,r,s}$, if for all $\kb\in \Ic$ of length $\length\kb \leq 2r$ we have
\begin{equation}
\label{nonres1}
|\omega_\kb(I)|> \gamma  \varepsilon^2 \left( \prod_{\alpha=1}^{\length\kb } \langle k_\alpha \rangle^{-2} \right) \langle \mu_{\min}(\kb)\rangle^{-2s}
\end{equation}
and
\begin{equation}
\label{nonres2}
|\widetilde\Omega_\kb(I)|>  \gamma  \left( \prod_{\alpha=1}^{\length\kb} \langle k_\alpha \rangle^{-6}\right) \max( \varepsilon^2 \langle \mu_{\min}(\kb)\rangle^{-2s}, \varepsilon^4).
\end{equation}
\end{definition}

We also define the truncated non resonant set: 

\begin{definition}\label{defnonresN} Let $\eps,\gamma >0$, $N \geq 1$ $r\geq1$ and $s\geq0$, and let $\alpha_r = 24r$. 
We say that  $z\in \ell_s^1$ belongs to the non resonant set $ \mathcal{U}_{\gamma,\varepsilon,r,s}^N$, if for all $\kb\in \Ic$ of length $\length\kb \leq 7r$ such that
$\langle \mu_1(\kb)\rangle \leq N^2$,  we have \begin{equation}
\label{nonres1N}
|\omega_\kb(I)|>   \gamma \varepsilon^2  N^{-\alpha_r} \langle \mu_{\min}(\kb)\rangle^{-2s} 
\end{equation}
and
\begin{equation}
\label{nonres2N}
|\Omega_\kb(I)|>  \gamma  N^{-\alpha_r}\max(  \varepsilon^2 \langle \mu_{\min}(\kb)\rangle^{-2s},\varepsilon^4) . 
\end{equation}
\end{definition}
It turns out that for $N$ not too large depending on $\varepsilon$ and for $\gamma'<\gamma$ we have $\mathcal{U}_{\gamma,\varepsilon,r,s}\subset \mathcal{U}_{\varepsilon,\gamma',r,s}^N$. Precisely we have:
\begin{proposition}
\label{musset}
Let $r\geq1$ and $s\geq 1$ be given. There exists $c$ such that  for all $\eps,\gamma>0$,  for all $N \geq 1$ and all $\gamma' < \gamma$ satisfying $$\varepsilon^2  < cN^{- \alpha_r}( \gamma - \gamma')\quad \text{ with } \alpha_r = 24r, $$ we have that if  $z \in  \mathcal{U}_{\gamma,\varepsilon,r,s}$ and $\Norm{z}{s} \leq 4 \varepsilon$ then $z \in \mathcal{U}_{\varepsilon,\gamma',r,s}^N$. 

\end{proposition}

\begin{proof}
The hypothesis $z \in \mathcal{U}_{\gamma,\varepsilon,r,s}$ and $\langle \mu_1(\kb) \rangle \leq N^2$ shows that if $z$ satisfies \eqref{nonres1}, we have 
\begin{eqnarray*}
|\omega_\kb(I)|&>& \gamma  \varepsilon^2 \left( \prod_{\alpha=1}^{\length\kb } \langle k_\alpha \rangle^{-2} \right) \langle \mu_{\min}(\kb)\rangle^{-2s} \\
&\geq& \gamma  \varepsilon^2 N^{-4 \sharp \kb}\langle \mu_{\min}(\kb)\rangle^{-2s}
\end{eqnarray*}
which shows \eqref{nonres1N} for all $\gamma' \leq \gamma$, as $\sharp \kb \leq 2r$ and hence $4 \sharp \kb \leq 24 r = \alpha_r$. Similarly, we have 
$$
|\widetilde\Omega_\kb(I)|>  \gamma  N^{- 12 \sharp \kb} \max( \varepsilon^2 \langle \mu_{\min}(\kb)\rangle^{-2s}, \varepsilon^4), 
$$
which shows that $\widetilde \Omega_{\kb}(I)$ satisfies \eqref{nonres2N} with $\alpha_r = 24r$. 

To prove \eqref{nonres2N}, we use the fact that using \eqref{pim} we have 
$$
|\widetilde \Omega_{\jb}(I) -\Omega_{\jb}(I)|\leq C \langle \mu_{\min}(\jb)\rangle^{-2s} \varepsilon^4. 
$$
This shows that 
$$
|\Omega_{\jb}(I)| \geq ( \gamma N^{- \alpha_r} - C \varepsilon^2) \max(  \varepsilon^2 \langle \mu_{\min}(\jb)\rangle^{-2s},\varepsilon^4), 
$$
and we deduce the result by choosing $c = 1/C$. 
\end{proof}

We conclude this section with two stability results of the truncated resonant sets. The first one use the fact that the non resonance conditions depend only upon $I$:
\begin{proposition}
\label{georgette}
Let $r\geq1$ and $s\geq 1$ be given. There exists $c$ such that the following holds: for $\eps,\gamma>0$ and $N\geq1$, let $z \in  \mathcal{U}_{\gamma,\varepsilon,r,s}^N$ such that $\Norm{z}{s} \leq 4 \varepsilon$, then  for all 
$\gamma' > \gamma$ and for all $z' \in \ell_s^1$  such that  
\begin{equation}\label{II}\sup_{a\in\Z}| I'_a-I_a |\langle a\rangle^{2s} \leq  c \varepsilon^2 N^{- \alpha_r} (\gamma' - \gamma) \end{equation} we have $z' \in \mathcal{U}_{\varepsilon,\gamma',r,s}^N$. 

\end{proposition}
\begin{proof}
We introduce the Banach space $\ell^\infty_s=\{ (x_n)_{n\in\Z}\in\R^\Z \mid \sup  \langle n \rangle^{2s}|x_n|<\infty \}$ that we endow with the norm $|x|_s:=\sup \langle n \rangle^{2s}|x_n|$.\\
Let $\jb\in \Ic$, we have using  \eqref{brassens}, 
$$
|\omega_{\jb}(I) -\omega_{\jb}(I')| \leq C \sum_{\alpha = 1}^{2m} |I_{a_\alpha} - I_{a_\alpha}'| \leq C \langle \mu_{\min}(\jb)\rangle^{-2s}  |I-I'|_s. 
$$
Thus since $z \in  \mathcal{U}_{\gamma,\varepsilon,r,s}^N$ we get using \eqref{II}
$$|\omega_{\jb}(I') | \geq \big( \gamma  +  Cc(\gamma'-\gamma)\big)\eps^2N^{- \alpha_r} \langle \mu_{\min}(\jb)\rangle^{-2s}\geq \gamma'\eps^2N^{- \alpha_r} \langle \mu_{\min}(\jb)\rangle^{-2s}$$
by choosing $c\leq\frac1{C}$.\\
On the other hand using that $I\mapsto \Omega_\jb(I)$ is a homogeneous polynomial of order 3 on $\ell^\infty_s$. For $s>1/2$ such polynomial (with bounded coefficients) is a $\mathcal{C}^\infty$ function and for $I,I'$ in a ball of $\ell^\infty_s$ of size $O(\eps)$ centered at $0$ we have
$$| \Omega_{\jb}(I)- \Omega_{\jb}(I')| \leq C\eps^2 |I-I'|_s.$$
So  we get 
$$|\Omega_{\jb}(I)-\Omega_{\jb}(I')|\leq C|I-I'|_s\max(\langle \mu_{\min}(\jb)\rangle^{-2s},\eps^2)$$
Thus since $z \in  \mathcal{U}_{\gamma,\varepsilon,r,s}^N$ we get using \eqref{II}
\begin{align*}
|\Omega_{\jb}(I') | &\geq \big( \gamma  +  Cc(\gamma'-\gamma)\big)\eps^2N^{- \alpha_r} \max(\langle \mu_{\min}(\jb)\rangle^{-2s},\eps^2)\\
&\geq
\gamma'\eps^2N^{- \alpha_r} \max(\langle \mu_{\min}(\jb)\rangle^{-2s},\eps^2)
\end{align*}
by choosing $c\leq\frac1{C}$.
\end{proof}
The second stability result    shows that the truncated resonant sets are stable by perturbation in $\ell_s^1$ up to change of constants. 

\begin{proposition}
\label{georgesand}
Let $r\geq1$ and $s\geq 1$ be given. There exists $c$ such that the following holds: for $\eps,\gamma>0$ and $N\geq1$, let $z \in  \mathcal{U}_{\gamma,\varepsilon,r,s}^N$ such that $\Norm{z}{s} \leq 4 \varepsilon$, then  for all 
$\gamma' < \gamma$ and for all $z' \in \ell_s^1$  such that  
 $$\Norm{z - z'}{s} \leq  c \varepsilon N^{- \alpha_r} (\gamma - \gamma'),$$ we have $z' \in \mathcal{U}_{\varepsilon,\gamma',r,s}^N$. 

\end{proposition}
\begin{proof}
By using \eqref{pam} with $h = 4\varepsilon$, 
$$
|\omega_{\jb}(I) -\omega_{\jb}(I')|\leq C  \Norm{z - z'}{s}  \langle \mu_{\min}(\jb)\rangle^{-2s} \varepsilon
$$
We deduce that 
$$
|\omega_{\jb}(I)| \leq |\omega_{\jb}(I') | + C \Norm{z - z'}{s}  \langle \mu_{\min}(\jb)\rangle^{-2s} \varepsilon, 
$$
and hence 
$$
|\omega_{\jb}(I') | \geq \big( \gamma N^{- \alpha_r} -  C \varepsilon^{-1} \Norm{z - z'}{s}\big) \langle \mu_{\min}(\jb)\rangle^{-2s} \varepsilon^2,
$$
 and we deduce the first part of the result by taking $c \leq 1/C$. 
 Now using \eqref{poum}, we have 
 \begin{multline*}
 |\Omega_{\jb}(I) -\Omega_{\jb}(I')|\leq C \Norm{z - z'}{s} \Big(\langle \mu_{\min}(\jb)\rangle^{-2s}   \varepsilon +   \varepsilon^3\Big)\\
 \leq 2 C \varepsilon^{-1} \Norm{z - z'}{s}  \max (\varepsilon^2 \langle \mu_{\min}(\jb)\rangle^{-2s},   \varepsilon^4)
\end{multline*}
and we conclude as before by taking $c \geq 1/(2C)$. 
\end{proof}

\section{Probability estimates}\label{gene}

In this section, we prove two genericity results for the \eqref{nls} non-resonant sets (and give also some Lemmas that will be used in the \eqref{nlsp} case). We consider real $z$, we consider the actions $I_a \in \R_+$ as  random variables.  On the one hand, if $\varepsilon>0$ we prove that typically $\varepsilon z$ is non-resonant. On the other hand, typically, up to some exceptional values of $\varepsilon$, we show that $\varepsilon z$ is non-resonant. 

\medskip

In this section $r>0$ and $s>1$ are fixed numbers, $\lambda = \frac{\mathbb{1}_{\varepsilon>0}}{\varepsilon}  { \rm d}\varepsilon$ is the Haar measure of $\mathbb{R}_+^*$ and we consider $z=(\xi,\bar{\xi})$ as a function of the random variables $I = (I_a)_{a\in\Z}$,  such that 
\begin{itemize}
\item the actions $I_a$, ${a\in \mathbb{Z}}$, are independent variables,
\item  $I_a^2$ is uniformly distributed in  $(0,  \langle a \rangle^{-4s-8} )$.
\end{itemize}
We note that this last assumption implies that $z\in \ell_s^1$.

The first proposition describes the case where $\varepsilon>0$ is fixed.
\begin{proposition}
\label{counting_eps_fixed}
 There exists a constant $c>0$ such that for all $\gamma\in(0,1)$ we have
$$
\forall \varepsilon>0, \ \mathbb{P}\left( \varepsilon z\in  \mathcal{U}_{\gamma,\varepsilon,r,s} \right) \geq 1 - c\gamma.
$$
\end{proposition}
The second proposition describes the case where $z$ is chosen randomly and the asymptotic of $\varepsilon z$ is considered as $\varepsilon$ goes to $0$.
\begin{proposition}
\label{counting_u0_fixed}
 There exists a constant $c>0$ and $\nu_0>0$ such that for all $\gamma \in(0,1)$ and all $\nu\in(0,\nu_0)$ we have
$$
\mathbb{P}\left(  \lambda \left( \varepsilon z\in \ell^1_s \setminus \mathcal{U}_{\gamma,\varepsilon,r,s} \right) < \nu \right) \geq 1 - \frac{c}{\nu} \gamma.
$$
\end{proposition}
\begin{corollary}
\label{counting_sequence}
There exists a constant $c>0$ and $\nu_0>0$ such that for all $\gamma\in (0,1)$, $\nu\in(0,\nu_0)$, all $\epsilon_\circ > 0$,  all sequence $(x_n)_{n\in \mathbb{N}}$ of random variables uniformly distributed in $(0,1)$ and independent of $z$, there is a probability larger than $1-c\gamma/\nu$ to realize $z$ such that that there is a probability larger than $1-\nu$ to realize $(\varepsilon_n)_{n\in \mathbb{N}} := (\epsilon_\circ  2^{-(n+x_n)})_{n\in \mathbb{N}}$ such that $\varepsilon_n z$ is non-resonant (i.e. $\varepsilon_n z \in \mathcal{U}_{\varepsilon_n,\gamma,r,s}$). More formally, we have
$$
\mathbb{P}\left( \mathbb{P} \left(\forall\, n, \,  \varepsilon_n z \in \mathcal{U}_{\varepsilon_n,\gamma,r,s} \ | \ z\right) \geq 1  - \nu   \right) \geq 1 - c \frac{\gamma}{\nu}.
$$
\end{corollary}
\begin{remark} The variables $x_n$ are not necessarily independent. For example, we could choose $x_n=x_0$ for all $n$.
\end{remark}

In order to prove these propositions and the corollary, we introduce some notations and elementary stochastic lemmas.
\begin{definition}
If a random variable $I$ has a density with respect to the Lebesgue measure, we denote $f_I$ its density, i.e.
\[  \forall g\in C^0_b(\mathbb{R}), \quad \mathbb{E}\left[ g(I) \right] = \int_{\mathbb{R}} g(x) f_I(x) {\rm dx}. \]
\end{definition}
\begin{lemma} 
\label{coord_change}
If $I$ is a random variable with density with values in $A\subset \mathbb{R}$ and $\phi:A\to \mathbb{R}$ satisfies $\phi'>0$ then $\phi$ has a density and
$$
 f_{\phi(I)} = \frac{f_I }{\phi'}\circ \phi^{-1}  .
 $$
\end{lemma}
\begin{proof} If $g\in C^0_b(\mathbb{R})$ then we have
$$
 \mathbb{E}\left[ g\circ \phi(I) \right] = \int_{I} g\circ \phi(x) f_I(x) {\rm dx} = \int_{{\rm Im }\phi} g(y) \frac{f_I }{\phi'}\circ \phi^{-1}(y) {\rm dy} .
 $$
\end{proof}
\begin{corollary}
\label{corail} If $I^2$ is uniformly distributed on $(0,1)$ and $I\geq 0$ almost surely then $I$ has a density given by 
$$
\forall x\in \mathbb{R}, \ f_{I} (x) = 2x \mathbb{1}_{0<x<1}. 
$$
Moreover, if $I$ has a density and $\varepsilon >0$ then $\varepsilon I$ has a density and for all $x\in \mathbb{R}$
$$
\forall x\in \mathbb{R}, \ f_{\varepsilon I}(x) = \frac1{\varepsilon}f_I(\frac{x}{\varepsilon}).
$$
In particular, we have  for all $a \in \Z$ and $x \in \R$, \begin{equation}
\label{gloubi}
\ f_{I_a^2}(x) = \langle a \rangle^{4s + 8} \mathbb{1}_{0<x< \langle a \rangle^{- 4s - 8} } \quad\mbox{and}\quad f_{I_a}(x) = 2 x \langle a \rangle^{ 4s + 8} \mathbb{1}_{0<x< \langle a \rangle^{- 2s - 4} }. 
\end{equation}
\end{corollary}
\begin{lemma} 
\label{whaou}
Let $I,J$ be some real independent random variables. If $I$ has a density, then for all $\gamma>0$
\[ \mathbb{P}(|I+J|<\gamma)\leq 2\gamma \|f_I\|_{L^{\infty}}. \]
\end{lemma}
\begin{proof} By Tonelli theorem, we have
$$  \mathbb{P}(|I+J|<\gamma) = \mathbb{E}\left[ \mathbb{1}_{|I+J|<\gamma} \right]=\mathbb{E}\left[ \int_{J-\gamma}^{J+\gamma} f_I(x) {\rm dx} \right] \leq 2\gamma \|f_I\|_{L^{\infty}}.  
$$
\end{proof}

\begin{lemma} 
\label{whaou_haar}
If $a,b, \lambda \in \mathbb{R}_+^*$ are such that $0<\gamma < |b|$ then for all $\sigma\in \mathbb{R}^*$
$$
\lambda( |a\varepsilon^\sigma+ b|<\gamma ) \leq \frac{2\gamma}{|\sigma|(|b|-\gamma)}.
$$
\end{lemma}
\begin{proof} Applying a natural change of coordinate, we get
\begin{multline*}
\lambda( |a\varepsilon^\sigma+ b|<\gamma ) = \int_{\mathbb{R}_+^*}  \mathbb{1}_ {|a\varepsilon^\sigma+ b|<\gamma}\frac{ {\rm d}\varepsilon}{\varepsilon} = \frac1{|\sigma|}\int_{\mathbb{R}_+^*}  \mathbb{1}_ {|a\varepsilon+ b|<\gamma}\frac{ {\rm d}\varepsilon}{\varepsilon} \\
= \frac1{|\sigma|}\int_{ \varepsilon\in \left(\frac{-b-\gamma}{a};\frac{-b+\gamma}{a} \right)\cap \mathbb{R}_+^* }  \frac{ {\rm d}\varepsilon}{\varepsilon} \leq   \frac{2\gamma}{|\sigma|(|b|-\gamma)}.
\end{multline*}
\end{proof}

A first application of these lemmas is the genericity of the non-resonance assumption \eqref{nonres1}.
\begin{lemma}
\label{lemma_nonres1} There exists a constant $c>0$ such that for all $\gamma\in (0,1)$ we have
$$
\mathbb{P}\left( \forall \kb \in \Ic, \ \length k \leq 2r \Rightarrow |\omega_{\kb}(I)| \geq \gamma \left( \prod_{\alpha=1}^{\length\kb } \langle k_\alpha \rangle^{-2} \right) \langle \mu_{\min}(\kb)\rangle^{-2s} \right) \geq 1 - c\gamma.
$$
\end{lemma}
\begin{proof} We are going to bound the probability of the complementary event by $c\gamma$. For each $\kb \in \irr$ of length smaller than or equal to $2r$, we have to estimate $\mathbb{P}\left( |\omega_{\kb}(I)| < \gamma_{\kb}\right)$, where $\gamma_{\kb}>0$ will be judiciously chosen. \\
We recall that by definition, if $\kb = (k_1,\dots, k_{\length \kb})$ and $k_\alpha = (\delta_\alpha,a_{\alpha})$ then we have
$$
\omega_{\kb}(I) = - \varphi'(0) \sum_{\alpha = 1}^{\length \kb} \delta_\alpha  I_{a_\alpha}.
$$
So paying attention to the multiplicity, this sum writes
$$
\omega_{\kb}(I) = \sum_{\beta = 1}^{m} p_\beta  I_{a_{\alpha_\beta}},
$$
where $m\leq \length \kb$ is an integer, $p\in \mathbb{R}^m$ satisfies $|p_\beta|\geq |\varphi'(0)|$ and $(a_{\alpha_\beta})_{\beta}$ is a subsequence of $a$.  
Thus, using Lemma \ref{whaou}, we have
\begin{multline*}
\mathbb{P}\left( |\omega_{\kb}(I)| < \gamma_{\kb}\right) \leq 2\gamma_{\kb} \min_{\beta =1,\dots,m} \frac{\Norm{ f_{I_{a_{\alpha_\beta}}}}{L^\infty}}{|p_\beta|}  \\
\leq   \frac{2\gamma_{\kb}}{|\varphi'(0)|} \min_{\alpha =1,\dots,\length \kb} \| f_{I_{a_\alpha}}\|_{L^\infty} \leq \frac{4\gamma_{\kb}}{|\varphi'(0)|} \langle \mu_{\min}(\kb)\rangle^{2s-4}, 
\end{multline*}
by using \eqref{gloubi}. 
Consequently, if we take
$$
\gamma_{\kb} = \gamma \left( \prod_{\alpha=1}^{\length\kb } \langle k_\alpha \rangle^{-2} \right) \langle \mu_{\min}(\kb)\rangle^{-2s}
$$
we get
$$
\mathbb{P}\left( |\omega_{\kb}(I)| < \gamma_{\kb}\right) \leq \frac{4\gamma}{|\varphi'(0)|} \prod_{\alpha=1}^{\length\kb-2 } \langle \mu_\alpha(\kb) \rangle^{-2}.
$$
Using the fact that $\kb \in \Rc$ and the zero momenta conditions \eqref{zeromomenta}, we see that $\mu_{\length \kb}(\kb)$ and $\mu_{\length \kb -1}(\kb)$ can be expressed as functions of $\mu_{\length \kb -2}(\kb),\dots,\mu_1(\kb)$, so this last product is summable on $\{ \kb \in \irr \hspace{5pt}| \hspace{5pt} \length \kb \leq 2r \}$. Consequently, there exists a constant $c>0$ such that
$$
\mathbb{P}\left( \exists \kb \in \Ic, \ \length k \leq 2r\textrm{ and } |\omega_{\kb}(I)| < \gamma_{\kb}\right) \leq \sum_{\substack{\kb \in \irr\\ \length \kb \leq 2r}} \mathbb{P}\left( |\omega_{\kb}(I)| < \gamma_{\kb}\right) \leq c \gamma.
$$
\end{proof}

A second application is the proof of Corollary \ref{counting_sequence} of Proposition \ref{counting_u0_fixed}.\\
\noindent  \bf Proof of Corollary \ref{counting_sequence}\rm.

We denote $\mathscr{E}_\lambda$ the event defined by
$$
\mathscr{E}_\lambda = \{  \lambda \left( \varepsilon z\in \ell^1_s \setminus \mathcal{U}_{\gamma,\varepsilon,r,s} \right) < \nu \}.
$$
Applying Proposition \eqref{counting_u0_fixed}, there exists a constant $c>0$ such that for all $\gamma\in (0,1)$ we have $\mathbb{P}(\mathscr{E}_\lambda)\geq 1- c\gamma/\nu$. Thus, we will conclude this proof showing that
$$
\mathbb{P}\left( \forall\, n ,\,  \varepsilon_n z \in \mathcal{U}_{\varepsilon_n,\gamma,r,s}  \ | \ \mathscr{E}_{\lambda}\right) \geq 1 - \nu.
$$
To show this, we just have to prove that
$$
\sum_{n\in \mathbb{N}} \mathbb{P}\left(  \varepsilon_n z \in  \ell^1_s \setminus \mathcal{U}_{\varepsilon_n,\gamma,r,s}   \ | \ \mathscr{E}_{\lambda} \right) < \nu. 
$$
By a natural change of variable (see Lemma \ref{coord_change}), $\varepsilon_n$ has a density given by
$$
f_{\varepsilon_n} = \frac1{\log 2}\left\{ \begin{array}{lll} \epsilon_\circ^{-1}\varepsilon^{-1} &\textrm{ if } \epsilon_\circ 2^{-n-1}<\varepsilon< \epsilon_\circ 2^{-n},\\
					0 & \textrm{ else.}
					\end{array}
			  \right.
$$
Consequently, since $(\varepsilon_n)$ is independent of $z$, applying Chasles formula, we get
\begin{multline*}
\sum_{n\in \mathbb{N}} \mathbb{P}\left(  \varepsilon_n z \in  \ell^1_s \setminus \mathcal{U}_{\varepsilon_n,\gamma,r,s}   \ | \ \mathscr{E}_{\lambda} \right) =\sum_{n\in \mathbb{N}} \mathbb{E}\left[  \int_{\mathbb{R}} \mathbb{1}_{\varepsilon z \in  \ell^1_s \setminus \mathcal{U}_{\gamma,\varepsilon,r,s}} f_{\varepsilon_n}  (\varepsilon)  {\rm d}\varepsilon \ |\ \mathscr{E}_{\lambda} \right] \\
= \frac1{\log 2} \mathbb{E}\left[ \sum_{n\in \mathbb{N}}  \int_{2^{-n-1}}^{2^{-n}} \mathbb{1}_{\varepsilon z \in  \ell^1_s \setminus \mathcal{U}_{\gamma,\varepsilon,r,s}}   \frac{{\rm d}\varepsilon}{\varepsilon} \ |\ \mathscr{E}_{\lambda} \right]
\leq  \frac1{\log 2} \mathbb{E}\left[\nu\ | \ \mathscr{E}_{\lambda} \right] = \frac{\nu}{\log 2}, 
\end{multline*}
and we easily obtain the result after a scaling in $\nu$. 
\hfill $\square$

\medskip

To take into account the terms induced by $Z_6$ in the proof Proposition \ref{counting_eps_fixed} and Proposition \ref{counting_u0_fixed}, we are going to need an useful algebraic lemma.
\begin{lemma}
\label{proba_need_algebra} If $\kb =(k_1,\dots,k_{2m}) \in \Ic$ with $k_\alpha = (\delta_\alpha,a_\alpha)$,
there exists $a^* \in \rrbracket- 3m,  3m  \llbracket   \setminus \{  a_1,\dots,a_{2m}  \}$ such that
\begin{equation}
\label{patate}
\left| \sum_{\alpha = 1}^{2m} \frac{\delta_\alpha}{(a^*-a_{\alpha})^2} \right| \geq (6m)^{-4m} \prod_{\alpha=1}^{2m} \langle a_{\alpha} \rangle^{-2}.
\end{equation}
\end{lemma}
\begin{proof} First, we observe that there exists $P\in \mathbb{Z}[X]$ of degree smaller than or equal to $4m-2$ such that
\begin{equation}
\label{temp_frac}
\sum_{\alpha = 1}^{2m} \frac{\delta_\alpha}{(X-a_{\alpha})^2} = P(X) \prod_{\alpha = 1}^{2m} \frac1{(X-a_{\alpha})^2}.
\end{equation}
Since $\kb$ is irreductible, we deduce of the uniqueness of partial fraction decomposition that $P\neq 0$. Hence, $P$ vanishes in, at most, $4m-2$ points. But there are, at least, $4m-1$ points into $\rrbracket - 3m,  3m \llbracket \setminus \{  a_1,\dots,a_{2m}  \}$. So we can find $a^*$ in this set such that $P(a^*)\neq 0$.

\medskip

Then, since $P\in \mathbb{Z}[X]$ and $a^*\in \mathbb{Z}$, we deduce that $|P(a^*)|\geq 1$. Thus, to prove \eqref{patate}, we just have to bound each factor of the denominator in \eqref{temp_frac} by 
$$
|a_{\alpha}-a^*|\leq 6 m \langle a_{\alpha} \rangle.
$$
To get this estimate we just have to observe that if $|a_\alpha|<3m$ then
$$
|a_{\alpha}-a^*|\leq |a_{\alpha}|+|a^*| \leq 6m \leq 6 m \langle a_{\alpha} \rangle,
$$
whereas, if $|a_\alpha|\geq 3m$ then $|a^*|\leq |a_\alpha|$ and so
$$
|a_{\alpha}-a^*|\leq |a_{\alpha}|+|a^*| \leq 2|a_\alpha| \leq 6m  \langle a_{\alpha} \rangle.
$$
\end{proof}

\noindent \bf Proof of Proposition \ref{counting_u0_fixed}.  \rm 
Let $\nu \in (0,\frac12)$ and $\gamma\in (0,1)$. We introduce three events
$$
  \mathscr{E}_4 = \left\{  \forall \kb \in \Ic, \ \length k \leq 2r \Rightarrow |\omega_{\kb}(I)| \geq \frac{\gamma}{\nu}  \left( \prod_{\alpha=1}^{\length\kb } \langle k_\alpha \rangle^{-2} \right)  \langle \mu_{\min}(\kb)\rangle^{-2s} \right\} ,
  $$
  $$
  \mathscr{E}_6 = \left\{  \forall \kb \in \Ic, \ \length k \leq 2r \Rightarrow |\widetilde{\Omega}_\kb(I) - \omega_{\kb}(I)| \geq \frac{\gamma}{\nu} \left( \prod_{\alpha=1}^{\length\kb } \langle k_\alpha \rangle^{-4} \right)   \right\}  ,
  $$
  and
  $$
 \mathscr{E}_{\lambda} = \left\{  \lambda \left(  \varepsilon z\in \ell^1_s \setminus \mathcal{U}_{\gamma,\varepsilon,r,s} \right) < c_{\lambda}\nu \right\},
  $$
  where $c_{\lambda}$ is a positive constant that will be determine later.

  We have proven in Lemma \ref{lemma_nonres1} that there exists a constant $c_4>0$ such that 
  $$
  \mathbb{P}\left(  \mathscr{E}_4  \right) \geq 1 - c_4\frac{\gamma}{\nu}.
  $$
  We are going to prove, on the one hand, that there exists a constant $c_6>0$ (independent of $\gamma$ and $\nu$) such that
  \begin{equation}
 \label{promize1}
  \mathbb{P}\left(  \mathscr{E}_6  \right) \geq 1 - c_6\frac{\gamma}{\nu}.
  \end{equation}
On the other hand, we will prove that
  \begin{equation}
   \label{promize2}
   \mathscr{E}_4  \cap \mathscr{E}_6 \subset  \mathscr{E}_{\lambda}.
  \end{equation}
  Assuming \eqref{promize1} and \eqref{promize2}, and up to a natural rescaling with respect to $\nu$, Proposition \ref{counting_u0_fixed} becomes a straightforward estimate:
\begin{multline*}
\mathbb{P}(\mathscr{E}_{\lambda}) \geq \mathbb{P}(\mathscr{E}_4  \cap \mathscr{E}_6) = 1 - \mathbb{P}(\mathscr{E}_4^c  \cup \mathscr{E}_6^c) \geq 1 - (\mathbb{P}(\mathscr{E}_4^c)+  \mathbb{P}(\mathscr{E}_6^c) ) \\= \mathbb{P}(\mathscr{E}_4) + \mathbb{P}(\mathscr{E}_6) - 1 \geq 1 - \frac{c_4+c_6}{\nu} \gamma.
\end{multline*}

\medskip

First, we focus on the proof of \eqref{promize1}, which is similar to the proof of Lemma \ref{lemma_nonres1}. We are going to bound the probability of the complementary event by $c_6 \frac{\gamma}{\nu}$. For each $\kb \in \irr$ of length smaller than or equal to $2r$, we have to estimate $\mathbb{P}\left( |\widetilde{\Omega}_\kb(I) - \omega_{\kb}(I)| < \gamma_{\kb}\right)$, where $\gamma_{\kb}>0$ will be judiciously chosen. \\
We recall that by definition, if $\kb = (k_1,\dots, k_{\length k})$ and $k_\alpha = (\delta_\alpha,a_{\alpha})$ then we have
\begin{equation}
\label{nocturne}
\widetilde{\Omega}_\kb(I) - \omega_{\kb}(I) = - \frac{1}{2}\varphi'(0)^2 \sum_{\substack{b \in\Z \\ b \notin\{ a_1,\dots,a_{2m}\}}} I_b^2  \sum_{\alpha = 1}^{2m }  \frac{\delta_\alpha }{(a_\alpha-b)^2}  .
\end{equation}
Thus, using Lemma \ref{whaou} and Corollary \eqref{corail} we have
$$
\mathbb{P}\left( |\widetilde{\Omega}_\kb(I) - \omega_{\kb}(I)| <  \gamma_{\kb}\right) \leq 2 \gamma_{\kb} \inf_{b\in \mathbb{Z}} \left| \frac{1}{2}\varphi'(0)^2 \langle b \rangle^{-4s-8} \sum_{\alpha = 1}^{2m }  \frac{\delta_\alpha }{(a_\alpha-b)^2}  \right|^{-1}.
$$
Applying Lemma \ref{proba_need_algebra} to estimate this infimum, we get
$$
\mathbb{P}\left( |\widetilde{\Omega}_\kb(I) - \omega_{\kb}(I)| <  \gamma_{\kb}\right) \leq 4 \gamma_{\kb} (\varphi'(0))^{-2} (6r)^{4r} (3r)^{4s+8} \prod_{\alpha =1}^{\length \kb} \langle a_\alpha \rangle^2.
$$
Consequently, choosing $\gamma_{\kb} = \frac{\gamma}{\nu} \prod_{\alpha =1}^{\length \kb} \langle k_\alpha \rangle^{-4}$,
we get $\mathbb{P}(\mathscr{E}_6) \geq 1 - c_6 \frac{\gamma}{\nu}$ with
$$
c_6= 4 \varphi'(0)^{-2} (6r)^{4r} (3r)^{4s+8} \sum_{\substack{\kb \in \irr \\ \length \kb \leq 2r}}   \prod_{\alpha =1}^{\length \kb} \langle k_\alpha \rangle^{-2}.
$$

Now, we focus on the proof of \eqref{promize2}. So, we consider a realization of the actions $I = (I_a)_{a\in \Z}$ where the lower bounds characterising  $\mathscr{E}_4$ and $\mathscr{E}_6$  are satisfied, {\em i.e.} for all $\kb$ irreductible of length smaller than or equal to $2r$ we have
\begin{equation}
\label{carac_E4}
|\omega_{\kb}(I)| \geq \frac{\gamma}{\nu}  \left( \prod_{\alpha=1}^{\length\kb } \langle k_\alpha \rangle^{-2} \right)  \langle \mu_{\min}(\kb)\rangle^{-2s} ,
\end{equation}
and
\begin{equation}
\label{carac_E6}
|\widetilde{\Omega}_\kb(I) - \omega_{\kb}(I)| \geq \frac{\gamma}{\nu} \left( \prod_{\alpha=1}^{\length\kb } \langle k_\alpha \rangle^{-4} \right)  .
\end{equation}

 We have to estimate $\lambda \left(  \varepsilon z\in \ell^1_s \setminus \mathcal{U}_{\gamma,\varepsilon,r,s} \right)$ for such a realization $I$. Thus, we decompose naturally the set we are estimating:
$$
\left\{ \varepsilon >0 \ | \ \varepsilon z\in \ell^1_s \setminus \mathcal{U}_{\gamma,\varepsilon,r,s} \right\} = \Sigma_1 \cup \Sigma_2 \cup \Sigma_3,
$$
with
\begin{multline*}
\Sigma_1 =\left\{ \varepsilon >0 \ | \ \exists \kb \in \Ic, \ \length \kb \leq 2r \right. \\\left. \textrm{ and } |\omega_{\kb}(\varepsilon^2 I)| < \gamma \varepsilon^2  \left( \prod_{\alpha=1}^{\length\kb } \langle k_\alpha \rangle^{-2} \right)  \langle \mu_{\min}(\kb)\rangle^{-2s}   \right\},
\end{multline*}
\begin{multline*} 
\Sigma_2  =\left\{ \varepsilon >0 \ | \ \exists \kb \in \Ic, \ \length \kb \leq 2r \right. \\ \left. \textrm{ and } |\widetilde{\Omega}_{\kb}(\varepsilon^2 I)| < \gamma \varepsilon^2  \left( \prod_{\alpha=1}^{\length\kb } \langle k_\alpha \rangle^{-6} \right)  \langle \mu_{\min}(\kb)\rangle^{-2s}   \right\},
\end{multline*}
and
$$
\Sigma_3 =\left\{ \varepsilon >0 \ | \ \exists \kb \in \Ic, \ \length \kb \leq 2r \textrm{ and } |\widetilde{\Omega}_{\kb}(\varepsilon^2 I)| < \gamma \varepsilon^4  \left( \prod_{\alpha=1}^{\length\kb } \langle k_\alpha \rangle^{-6} \right)   \right\}.
$$
In fact, since $\omega_{\kb}$ is linear, $I$ satisfies \eqref{carac_E4} and $\nu<1$, $\Sigma_1$ is the empty set. Thus, we have
$$
\lambda \left(  \varepsilon z\in \ell^1_s \setminus \mathcal{U}_{\gamma,\varepsilon,r,s} \right) \leq \lambda(\Sigma_2) + \lambda(\Sigma_3).
$$
So, we have to estimate $\lambda(\Sigma_2)$ and $\lambda(\Sigma_3)$.
Observing that $\widetilde{\Omega}_{\kb} - \omega_{\kb}$ is quadratic, we have 
$$
\widetilde{\Omega}_\kb(\varepsilon^2 I) = \varepsilon^2 \omega_{\kb}(I) + \varepsilon^4 \left( \widetilde{\Omega}_\kb(I) - \omega_{\kb}(I) \right).
$$
Thus, we are going to estimate $\lambda(\Sigma_2)$ and $\lambda(\Sigma_3)$ with Lemma \ref{whaou_haar}. To apply, this Lemma, we observe that from \eqref{carac_E4} and \eqref{carac_E6} we have
$$
|\omega_{\kb}(I)| -  \gamma   \left( \prod_{\alpha=1}^{\length\kb } \langle k_\alpha \rangle^{-6} \right)  \langle \mu_{\min}(\kb)\rangle^{-2s}   \geq \gamma \left(\frac1{\nu}-1\right)  \left( \prod_{\alpha=1}^{\length\kb } \langle k_\alpha \rangle^{-2} \right)  \langle \mu_{\min}(\kb)\rangle^{-2s} ,
$$
and
$$
 |\widetilde{\Omega}_\kb(I) - \omega_{\kb}(I)| - \gamma   \left( \prod_{\alpha=1}^{\length\kb } \langle k_\alpha \rangle^{-6} \right)  \geq  \gamma \left(\frac1{\nu}-1\right)  \left( \prod_{\alpha=1}^{\length\kb } \langle k_\alpha \rangle^{-4} \right).
$$

Consequently, applying Lemma \ref{whaou_haar} with $\gamma_{\kb,s} :=\gamma  \left( \prod_{\alpha=1}^{\length\kb } \langle k_\alpha \rangle^{-6} \right)  \langle \mu_{\min}(\kb)\rangle^{-2s}$,  we get
\begin{multline*}
\lambda(\Sigma_2) \leq \sum_{\substack{\kb \in \irr \\ \length \kb \leq 2r}} \lambda\left(|\omega_{\kb}(I) + \varepsilon^2 \left( \widetilde{\Omega}_\kb(I) - \omega_{\kb}(I) \right)| < \gamma_{\kb,s}\right)  \leq \sum_{\substack{\kb \in \irr \\ \length \kb \leq 2r}}   \frac{\gamma_{\kb,s}}{|\omega_{\kb}(I) | - \gamma_{\kb,s}}\\ \leq \left(  \sum_{\substack{\kb \in \irr \\ \length \kb \leq 2r}}  \prod_{\alpha=1}^{\length\kb } \langle k_\alpha \rangle^{-4} \right) \frac{\nu}{1-\nu},
\end{multline*}
as the first previous estimate can be recast as $|\omega_{\kb}(I) | - \gamma_{\kb,s} \geq  \gamma_{\kb,s}( \frac{1}{\nu} - 1) \prod_{\alpha=1}^{\length\kb } \langle k_\alpha \rangle^{4}. $ Similarly, we obtain 
\begin{multline*}
\lambda(\Sigma_3) \leq \sum_{\substack{\kb \in \irr \\ \length \kb \leq 2r}} \lambda\left(|\varepsilon^{-2} \omega_{\kb}(I) +  \left( \widetilde{\Omega}_\kb(I) - \omega_{\kb}(I) \right)| < \gamma  \left( \prod_{\alpha=1}^{\length\kb } \langle k_\alpha \rangle^{-6} \right)  \right) \\ \leq \left(  \sum_{\substack{\kb \in \irr \\ \length \kb \leq 2r}}  \prod_{\alpha=1}^{\length\kb } \langle k_\alpha \rangle^{-2} \right) \frac{\nu}{1-\nu}.
\end{multline*}
Hence, since these sum are clearly convergent, we have proven that $\mathscr{E}_4  \cap \mathscr{E}_6 \subset  \mathscr{E}_{\lambda}$ for a convenient choice of $c_\lambda$.
\hfill $\square$

\noindent \bf Proof of Proposition \ref{counting_eps_fixed}. \rm Let $\varepsilon>0$ be a fixed positive real number. By definition of $\mathcal{U}_{\gamma,\varepsilon,r,s}$, we decompose $\left\{\varepsilon z\in  \mathcal{U}_{\gamma,\varepsilon,r,s} \right\}$ into
$$
\left\{\varepsilon z\in  \mathcal{U}_{\gamma,\varepsilon,r,s} \right\} = \mathscr{E}_4 \cap \mathscr{E}_{46}
$$
where
$$
\mathscr{E}_4 = \left\{    \forall \kb \in \irr, \ \length k \leq 2r \Rightarrow |\omega_{\kb}(\varepsilon^2 I)| \geq \gamma  \varepsilon^2 \left( \prod_{\alpha=1}^{\length\kb } \langle k_\alpha \rangle^{-2} \right)  \langle \mu_{\min}(\kb)\rangle^{-2s} \right\} 
$$
and
\begin{multline*}
\mathscr{E}_{46} = \left\{    \forall \kb \in \irr, \ \length k \leq 2r  \right. \\ \left.\Rightarrow   |\widetilde{\Omega}_{\kb}(\varepsilon^2 I)| \geq \gamma \left( \prod_{\alpha=1}^{\length\kb } \langle k_\alpha \rangle^{-6} \right)   \max \left(  \varepsilon^2  \langle \mu_{\min}(\kb)\rangle^{-2s} ,\varepsilon^4 \right)  \right\}.
\end{multline*}
Since $\omega_{\kb}$ is linear, $\mathscr{E}_4$ does not depend of $\varepsilon$. Consequently, applying Lemma \ref{lemma_nonres1}, we get a constant $c_4>0$ such that $\mathbb{P}(\mathscr{E}_4)\geq 1-c_4 \gamma$.
So assuming that there exists a constant $c_{46}>0$ such that $\mathbb{P}(\mathscr{E}_{46})\geq 1-c_{46} \gamma$, we could conclude the proof of Proposition \ref{counting_eps_fixed} by the following estimate
\begin{multline*}
\mathbb{P}\left( \varepsilon z\in  \mathcal{U}_{\gamma,\varepsilon,r,s} \right) = 1 - \mathbb{P}(\mathscr{E}_4^c\cup \mathscr{E}_{46}^c) \geq \mathbb{P}(\mathscr{E}_4)+\mathbb{P}(\mathscr{E}_{46}) - 1 \geq 1 - (c_4+c_{46}) \gamma.
\end{multline*}
Thus, we just have to focus on the proof of the existence of $c_{46}$. So, we recall that by definition, if $\kb = (k_1,\dots, k_{\length \kb})$ and $k_\alpha = (\delta_\alpha,a_{\alpha})$ then we have with $\sharp \kb = 2m$, 
$$
\widetilde{\Omega}_\kb(\varepsilon^2 I)  =  - \varepsilon^2 \varphi'(0) \sum_{\alpha = 1}^{2m} \delta_\alpha I_{a_{\alpha}}  - \frac{\varepsilon^4}{2}\varphi'(0)^2 \sum_{\substack{b \in\Z \\ b \notin\{ a_1,\dots,a_{2m}\}}} I_b^2  \sum_{\alpha = 1}^{2m }  \frac{\delta_\alpha }{(a_\alpha-b)^2}  .
$$ 
By construction, it is a sum of independent random variable, thus applying Lemma \ref{whaou}, we have the estimate on the complement
$$
\mathbb{P}(\mathscr{E}_{46}^c) \leq \sum_{\substack{\kb \in \irr \\ \length \kb \leq 2r}}  2 \gamma_{\kb,\varepsilon} \min \left(  \min_{\alpha =1,\dots,2m} \Norm{F_\alpha}{L^\infty} ,\inf_{b\in \mathbb{Z}} \Norm{G_b}{L^\infty}  \right),
$$
where $F_\alpha$ is the probability density function of the part depending on $I_\alpha$ in $\widetilde{\Omega}_\kb(\varepsilon^2 I) $ and $G_b$ the probability density function of the part depending on the variable $I_b^2$, and where
$$
\gamma_{\kb,\varepsilon} = \gamma \left( \prod_{\alpha=1}^{\length\kb } \langle k_\alpha \rangle^{-6} \right)   \max \left(  \varepsilon^2  \langle \mu_{\min}(\kb)\rangle^{-2s} ,\varepsilon^4 \right) .
$$
By using \eqref{gloubi}, we have that 
$$
\Norm{F_\alpha}{L^\infty} \leq  \frac{\Norm{f_{I_\alpha}}{L^\infty}}{|\varepsilon^2 \varphi'(0) {\rm Card} \{ \beta \hspace{5pt}| \hspace{5pt} a_{\beta}=a_{\alpha} \}|} \leq \frac{\langle a_\alpha \rangle^{2s + 4}}{|\varepsilon^2 \varphi'(0) | }  
$$
and thus 
$$\min_{\alpha =1,\dots,2m} \Norm{F_\alpha}{L^\infty} \leq  \frac{\langle \mu_{\min}(\kb)\rangle^{2s + 4}}{|\varepsilon^2 \varphi'(0) | } .$$
Similarly  using Lemma \ref{proba_need_algebra}, we obtain 
\begin{multline*}
\inf_{b \in \Z}
\Norm{G_b }{L^\infty} \leq \min_{b \in \rrbracket- 3m,  3m  \llbracket   \setminus \{  a_1,\dots,a_{2m}  \}}\frac{\Norm{f_{I_b^2}}{L^\infty}}{\left|\frac{\varepsilon^4}{2}\varphi'(0)^2  \displaystyle \sum_{\alpha = 1}^{2m }  \frac{\delta_\alpha }{(a_\alpha-b)^2}  \right| }\\ 
\leq \frac{4 \langle 3m \rangle^{2s + 4} (6m)^{4m} \displaystyle }{ \varepsilon^4\varphi'(0)^2  }\prod_{\alpha=1}^{2m} \langle a_{\alpha} \rangle^{2} \leq \frac{4 \langle 3r \rangle^{2s + 4} (6r)^{4r} }{ \varepsilon^4\varphi'(0)^2  }\prod_{\alpha=1}^{\length \kb} \langle k_{\alpha} \rangle^{2}. 
\end{multline*}
Hence there exists a constant $c_{r,s}$ depending on $r,s$ such that 
\begin{multline*}
\mathbb{P}(\mathscr{E}_{46}^c) \leq \sum_{\substack{\kb \in \irr \\ \length \kb \leq 2r}}  2 \gamma_{\kb,\varepsilon}  \min \left( \frac{\langle \mu_{\min}(\kb)\rangle^{2s + 4}}{|\varepsilon^2 \varphi'(0) | },\frac{c_{r,s} }{ \varepsilon^4\varphi'(0)^2  }\prod_{\alpha=1}^{\length \kb} \langle k_{\alpha} \rangle^{2}\right) \\
\leq \gamma c  \sum_{\substack{\kb \in \irr \\ \length \kb \leq 2r}}  \left( \prod_{\alpha=1}^{\length\kb } \langle k_\alpha \rangle^{-2} \right) \leq c_{46} \gamma, 
\end{multline*}
for some constants $c$ and $c_{46}$, and this shows the result. 
\hfill $\square$

\section{A class of rational Hamiltonians}\label{rational}

The set of Hamiltonian functions constructed in the normal form process arise from solving homological equation associated with small denominators $\omega_\kb(I)$ and $\Omega_{\kb}(I)$ (see \eqref{ravel}). 
 Then a natural class of Hamiltonians should be
\begin{equation}
\label{aznavour}
F(z) = \sum_{\jb \in \mathcal{R}}  f_\jb(I) z_\jb 
\end{equation}
for some functions $f_{\jb}$ which are inverse of products of small denominators $\omega_{\kb}(I)$ and $\Omega_{\kb}(I)$ associated with multi-indices depending on the construction process. Note that we sum over $\jb\in\Rc$ since the non resonant part will be killed beforehand by a standard resonant normal form procedure involving polynomial Hamiltonians (see section \ref{killbill0}).     Each  term of  \eqref{aznavour}  will be controlled by the non resonance conditions \eqref{nonres1} and \eqref{nonres2}, provided we can compensate the loss of derivative arising in the small denominator by terms in the numerator $z_\jb$. For a given $\jb$, several terms can appear in $f_{\jb}$ that are associated with different small denominators. To take into account the specificity of each term arising in the normal form process described in section \ref{RNF} we will introduce four sub-classes of rational Hamiltonians. \\
Notice that in the case of \eqref{nlsp}, the situation is simpler and only two sub-classes are needed (see Appendix \ref{casNLSP}).

\subsection{Construction of the class}
First we introduce the following set of indices, encoding the structure of the possible terms $f_{\jb}(I)$ arising in \eqref{aznavour}. 

For $r\in \mathbb{N}$, let $\Hscr_{r}$ be a set of multi-indices valued functions 
\begin{equation}
\label{pogorelich} \displaystyle (\pig,\kbsf,\hbsf,n):\Z^*\to \mathcal{R}  \times  \mathop{\bigsqcup}\limits_{p\in \mathbb{N}} {\Ic}^p  \times \mathop{\bigsqcup}\limits_{q\in \mathbb{N}} {\Ic}^q \times \mathbb{Z}^*.  
\end{equation}
For a given $\Gamma = (\pig,\kbsf,\hbsf,n)\in \Hscr_{r}$ and $\ell \in \Z^*$, we associated $\pi_\ell \in \mathcal{R}_{m_\ell}$ for some $m_\ell \in \N$, $\kbsf_\ell = (\kb_{\ell,1},\ldots,\kb_{\ell,p_\ell})$ and $\hbsf_\ell = (\hb_{\ell,1},\ldots,\hb_{\ell,q_\ell})$ for some $p_\ell$ and $q_\ell$ in $\N$.
For some given set of coefficients $c = (c_\ell)_{\ell \in \Z^*} \in \C^{\Z^*}$ we will define the Hamiltonian function
\begin{equation}
\label{Q}
 Q_\Gamma[c](z)= \sum_{ \ell \in \mathbb{Z^*} } c_\ell (-i)^{p_\ell + q_\ell} \frac{z_{\pig_\ell}}{\displaystyle \prod_{\alpha =1}^{n_{\ell}} \omega_{\kb_{\ell,\alpha}}  \prod_{\alpha =n_{\ell}+1}^{p_{\ell}} \Omega_{\kb_{\ell,\alpha}} \prod_{\alpha =1}^{q_{\ell}} \Omega_{
 \hb_{\ell,\alpha}}} .
 \end{equation}
 Note that such Hamiltonian can be recast under the form \eqref{aznavour} by setting 
\begin{equation}
\label{fj}
f_\jb(I) =   \sum_{ \ell \in \pig^{-1}(\jb)  } c_\ell (-i)^{p_\ell + q_\ell} \frac{1}{\displaystyle \prod_{\alpha =1}^{n_{\ell}} \omega_{\kb_{\ell,\alpha}}  \prod_{\alpha =n_{\ell}+1}^{p_{\ell}} \Omega_{\kb_{\ell,\alpha}} \prod_{\alpha =1}^{q_{\ell}} \Omega_{
 \hb_{\ell,\alpha}}}. 
\end{equation}

Roughly speaking, the structure of the class can be explained as follows: Each time a homological equation for $Z_4$ or $Z_4 + Z_6$ is solved, the term $f_{\jb}(I)$ is divided by $\omega_\jb$ or $\Omega_{\jb}$ so the functionals are naturally under the previous form. The fact that we decompose into two parts the contribution of $\Omega_{\jb}$ in the denominator is explained by the  $2r$ order condition {\bf (ii)} just below.

To ensure that the Hamiltonians $Q_{\Gamma}[c]$ are well defined and that their vectorfield can be controlled in $\ell_s^1$, we impose several restrictions on the set $\Hscr_{r}$:  $\Gamma = (\pig,\kbsf,\hbsf,n)\in \Hscr_{r}$ if  the following conditions are satisfied 

\medskip

\noindent {\bf (i)  \ul{Reality}}\label{i}. 
The functional is real, {\em i.e.}  $Q_{\Gamma}[c](z) \in \R$ for real $z$. This condition is satisfied by imposing for all $\ell \in \Z^*$, 
$$
\pig_{-\ell} = \overline{\pig_{\ell}},\quad
\overline{\kbsf_{\ell}} = \kbsf_{-\ell},\quad  n_{-\ell} = n_{\ell}, \quad  \mbox{and}\quad\overline{\hbsf_{\ell}} = \hbsf_{-\ell}
$$
after noticing that $\omega_{\overline{\jb}} = - \omega_{\jb}$ and $\Omega_{\overline{\jb}} = - \Omega_{\jb}$. Note that this implies that $p_\ell,q_\ell$ and $m_\ell$ are even functions of $\ell$.

\noindent {\bf (ii)  \ul{ $2r$ order}}. 
The  link between the terms of the class and the order $2r$ is given by the relation
\begin{equation}
\label{HP_order}
\forall \ell \in \Z^*,\quad  \ r=m_{\ell} - p_{\ell} -2q_{\ell}.
\end{equation}
This definition corresponds to the fact that while the numerator $z_{\pi_\ell}$ is of order $\varepsilon^{2m_\ell}$ and the homogeneity of the non resonance condition of $\omega_{\kb_{\ell,\alpha}}$ is $\varepsilon^2$, the homogeneity of the small denominator $\Omega_{\kb}$ can be of order $\varepsilon^2$ or $\varepsilon^4$ depending on the non resonance condition we use in \eqref{nonres2N}. The previous notation then specifies that $\Omega_{\kb_{\ell,\alpha}}$ will be controlled by the non resonance condition homogeneous to $\varepsilon^2$ while the others, $\Omega_{\hb_{\ell,\alpha}}$, will be controlled by the non resonance condition homogeneous to $\varepsilon^4$. 

\medskip
\noindent {\bf (iii)  \ul{Consistency}}. We assume that 
$ \forall \ell \in \Z^*, \ 0\leq n_{\ell} \leq p_{\ell}$. 

\medskip
\noindent {\bf (iv) \ul{Finite numerator and denominator degrees}}. We assume that 
\begin{equation}
\label{je_sers_a_rien} 
\sup_{\ell \in \N} m_\ell <\infty,
\end{equation}
and
\begin{equation}
\label{aouaismerde}
\forall\, \ell\in \Z^*, \quad \forall\,  \alpha, \quad \sharp \kb_{\ell,\alpha} \leq m_\ell \quad  \mbox{and}\quad\sharp \hb_{\ell,\alpha} \leq m_\ell. 
\end{equation}

\medskip
\noindent {\bf (v) \ul{Finite multiplicity}}. We assume that 
\label{HP_findeg}
$\sup_{j\in \mathcal{R}} \mathrm{Card}\,  \pig^{-1}(j) <\infty$. 
This condition ensures that the number of terms defining $f_j(I)$ in \eqref{fj} is finite.

\medskip
\noindent {\bf (vi) \ul{Distribution of the derivatives}}. 
There exists a positive constant $C>0$ such that for all $\ell\in \Z^*$, there exists $\iota : \llbracket 1,2p_\ell \rrbracket \to \llbracket 3,2m_\ell \rrbracket$, an injective function satisfying
\begin{equation}\label{crux}
 \left\{ \begin{array}{lll} \displaystyle \max_{1\leq \alpha \leq p_{\ell}}  \frac{\langle\mu_{\min}( \kb_{\ell,\alpha}) \rangle}{\langle\mu_{\iota_{2\alpha-1}}( \pig_\ell) \rangle}\leq C \quad\mbox{and}\\  \displaystyle \max_{1\leq \alpha \leq p_{\ell}} \frac{\langle\mu_{\min}( \kb_{\ell,\alpha}) \rangle}{\langle\mu_{\iota_{2\alpha}}( \pig_\ell) \rangle} \leq C. \end{array}\right. 
\end{equation}
This condition ensures that terms of the form $\langle\mu_{\min}( \kb_{\ell,\alpha}) \rangle^{-2s}$ arising in the denominators when using \eqref{nonres1} or \eqref{nonres2}  can be compensated by modes in the numerators $z_{\pig_\ell}$ smaller than the third largest.  In other words, the first and second largest indices in $\pig_\ell$, $\mu_1(\pig_\ell)$ and $\mu_2(\pig_\ell)$ will be free and not required to control the small denominators $\omega_{\kb_{\ell,\alpha}}$ and $\Omega_{\kb_{\ell,\beta}}$. This will ensure a global control of the vector field associated with $Q_\Gamma[c]$ after truncation independently of $s$. 

\medskip
\noindent {\bf (vii) \ul{Global control of the structure}}.
The following condition ensures that the structure has a kind of memory of the zero-momentum condition. 
\begin{equation}
\label{HP_old_momo}
\sup_{\ell \in \Z^*} \max_{1\leq \alpha \leq q_{\ell}} \frac{\langle\mu_{\min}( \hb_{\ell,\alpha}) \rangle}{\langle\mu_{2}( \pig_\ell) \rangle}<\infty.
\end{equation}

For $\Gamma = (\pig,\kbsf,\hbsf,n)\in \Hscr_{r}$ and 
 $c:\Z^* \to \mathbb{C}$, we define the weight of $c$ relatively to $\Gamma$ by
\begin{equation}
\label{weight} 
\Nc_{\Gamma}(c) = \sup_{\substack{\ell \in \Z^*\\ c_\ell \neq 0}} \max_{\substack{1\leq \alpha \leq p_{\ell}\\ 1\leq \beta \leq q_{\ell}} }  \big(\langle \mu_1(\irr(\pig_{\ell})\rangle,  \langle\mu_1(\kb_{\ell,\alpha}\rangle,\langle\mu_1(\hb_{\ell,\beta}\rangle\big).
\end{equation}

\medskip 
Then we introduce the space
\[ \ell^\infty_{\Gamma}(\mathbb{Z}^*) = \{ c\in \ell^\infty(\mathbb{Z}^*)\ | \  \quad \forall\, \ell \in \Z^*, \quad c_{-\ell} = \overline{c_\ell}\quad \mbox{and} \quad   \Nc_\Gamma(c) <\infty \} .\]
 Note that the Hamiltonian defined by \eqref{Q}  is clearly an analytic function on an  open subset of $\ell^1_s$ avoiding the zeros of the denominators. Further we note that $\Nc_\Gamma(c)$ is the maximal size of indices we have  at the denominator and thus the control that we will have on this denominator when $z$ belongs to the non resonant set (see Definition \ref{defnonres}) will only  depend on $\Nc_\Gamma(c)$. 


In order to stick as closely as possible to the rational Hamiltonians we are going to build in the next sections, we  introduce four subclasses of $\Hscr_r$ denoted by $\Hscr_{r,\omega},\Hscr_{r,\Omega},\Hscr_{r,\omega}^*$ and $\Hscr_{r,\Omega}^*$ respectively. This technical refinement, not really indispensable, will allow us to control $m_\ell$ (see Remark \ref{controlml}) which in turn will allow us to obtain better constants in our Theorems\footnote{Without tracking the form of our rational normal forms we will obtain in the right hand side of \eqref{probaNLS} $1-\eps^\nu$ for some constant  $\nu$ depending on $r$ and $s$,  instead of $1-\eps^{1/3}$. }. We first give the four definitions and then comment on them.
\begin{itemize}
\item $\Gamma = (\pig,\kbsf,\hbsf,n)\in \Hscr_{r}$ belongs to $\Hscr_{r,\omega}$ if 
$$
\forall \ell \in \mathbb{Z}^*, \ n_\ell = p_\ell \textrm{ and } q_\ell = 0 \textrm{ and } n_\ell \leq 2r - 6.
$$
\item $\Gamma = (\pig,\kbsf,\hbsf,n)\in \Hscr_{r}$ belongs to $\Hscr_{r,\omega}^*$ if
$$
\forall \ell \in \mathbb{Z}^*, \ n_\ell = p_\ell \textrm{ and } q_\ell = 0 \textrm{ and } n_\ell \leq 2(r+1) - 5.
$$
\item $\Gamma = (\pig,\kbsf,\hbsf,n)\in \Hscr_{r}$ belongs to $\Hscr_{r,\Omega}$ if 
\begin{equation}
\label{grizzli}
n_\ell = \alpha_1 + \alpha_2  \textrm{ and }  p_\ell = n_\ell + \alpha_3 \textrm{ and } q_\ell = \alpha_4+\alpha_5
\end{equation}
where $\alpha\in (\N)^5$ satisfies
\begin{equation}
\label{ilpleut}
 \alpha_1 \leq 2r-6 \textrm{ and } \alpha_2+\alpha_3+\alpha_4\leq \alpha_5 \textrm{ and } \alpha_5 \leq r-4.
\end{equation}
\item $\Gamma = (\pig,\kbsf,\hbsf,n)\in \Hscr_{r}$ belongs to $\Hscr_{r,\Omega}^*$ if 
$$
n_\ell = \alpha_1 + \alpha_2  \textrm{ and }  p_\ell = n_\ell + \alpha_3 \textrm{ and } q_\ell = \alpha_4+\alpha_5+1
$$
where $\alpha\in (\N)^5$ satisfies
\begin{equation}
\label{beaucoup}
 \alpha_1 \leq 2(r+2)-6 \textrm{ and } \alpha_2+\alpha_3+\alpha_4\leq \alpha_5 \textrm{ and } \alpha_5 \leq (r+2)-4.
\end{equation}
\end{itemize}
Some comments to clarify the meaning of these definitions: 
\begin{itemize}
\item  $\Hscr_{r,\omega}$ and  $\Hscr_{r,\Omega}$ will be used to describe the Hamiltonians arising in our normal forms. $\Hscr_{r,\omega}^*$ and $ \Hscr_{r,\Omega}^*$ will be used to describe the Hamiltonians obtained after solving a Homological equation (see Lemmas \ref{lemmehomo} and \ref{rapido}), and thus,  that govern our canonical changes of variables.
\item  $\Hscr_{r,\omega} \subset \Hscr_{r,\Omega}$  by taking $\alpha_i = 0$ for $i = 2,\ldots,5$. Nevertheless we prefer to introduce the class $\Hscr_{r,\omega}$ since it plays a special role in our construction. Actually in  our second step of normal form (see section \ref{killbill1}) we only use the class $\Hscr_{r,\omega}$ while in the third step of normal form (see section \ref{killbill2}) we only use the class $\Hscr_{r,\Omega}$.
\item
the $\alpha_i$ give some informations about the history that generated the term $\Gamma$: $\alpha_1$ counts the number of homological equations we solved with $Z_4$ in the second normal form process (section \ref{killbill1}); $\alpha_2$ increases when in a Poisson bracket, some $\omega_{\kb}(I)$ is involved (see \eqref{dnn}) in the third normal form process (section \ref{killbill2}); $\alpha_5$ control the number of homological equations we solved with $Z_4+Z_6$ in the third normal form process (section \ref{killbill2}); 
$\alpha_3$ increases when in a Poisson bracket, some $\Omega_{\kb}(I)$ is involved and we apply the derivative on the part of $\Omega_{\kb}(I)$ that comes from $Z_4$; $\alpha_4$ increases when in a Poisson bracket, some $\Omega_{\kb}(I)$ is involved and we apply the derivative on the part of $\Omega_{\kb}(I)$ that comes from $Z_6$.
\item the precise numerology is dictated by the experience of calculating the first terms and by the need for the overall structure to be stable by Poisson bracket (see  Lemma \ref{cap} which underlies the whole construction).
\end{itemize}
We eventually define the set of functionals associated with a structure in $\Hscr_{r}$, 
$$
\Fscr_{r} := \{F = Q_{\Gamma}[c],\, \ \, \Gamma \in \Hscr_{r}, \quad c \in  \ell^\infty_{\Gamma}(\mathbb{Z^*})   \}. 
$$
Then, we define naturally its subsets $\Fscr_{r,\omega},\Fscr_{r,\Omega},\Fscr_{r,\omega}^*$ and $\Fscr_{r,\Omega}^*$.
\begin{remark}
\label{loureed}
Note that all polynomials of the form \eqref{poly} can be written under the form $Q_\Gamma[c]$ for some $\Gamma = (\pig,\kbsf,\hbsf,n)$ with $\kbsf = \hbsf = \emptyset$ and the convention $n = 0$. More precisely, if $P$ is a polynomial of order $2m$, then it can be written under the previous form, with $\Gamma \in \Hscr_{m,\omega}\subset \Hscr_{m,\Omega}$. 
\end{remark}
\begin{remark}\label{controlml}
The uniform bound on the numerator in condition \eqref{je_sers_a_rien} can be specified on the subclasses. More precisely, using \eqref{HP_order} we deduce that if  $\Gamma = (\pig,\kbsf,\hbsf,n)$ belongs to $\Hscr_{r,\Omega}$ or $\Hscr_{r-2,\Omega}^*$ and $r\geq 4$ then 
\begin{equation}
\label{johncale}
m_\ell \leq 7r-22.
\end{equation}
Similarly, if $\Gamma$ belongs to $\Hscr_{r,\omega}$ or $\Hscr_{r-1,\omega}^*$ and $r\geq 3$ then 
\begin{equation}
\label{nico}
m_\ell \leq  3r-6.
\end{equation}

\end{remark}

\subsection{Structural lemmas}
In this section we verify that our class allows to define flows and that this class is stable by resolution of homological equations and by Poisson bracket.
\subsubsection{Control of the vector fields}

First, we have to verify that  the vector field associated with Hamiltonian belonging to the class  defined above are under control in $\ell^1_s$ in such way it defines a regular flow. In other words we would like to prove that such Hamiltonian are regular in the sense of Definition \ref{def:2.1}. Actually, we will control the vector field of Hamiltonian of the form $Q_{\Gamma}[c]$ for which $\Nc_\Gamma(c) \leq N^2$ for a given $N$, a property that is stable by Poisson bracket and solution of homological equation according to Lemmas \ref{lemmehomo} and \ref{cap}. 

\begin{lemma}\label{XQG} Let $r \geq 2$, $\alpha_r = 24r$, and $s$ be given. For all $\Gamma\in \Hscr_{r,\omega}, \Hscr_{r,\omega}^*, \Hscr_{r,\Omega}$ or $\Hscr_{r,\Omega}^*$ there exists a constant $C>0$ such that for all $\varepsilon,\gamma < 1$, all $c\in \ell^\infty_\Gamma(\mathbb{Z^*})$ and all $N \geq 1$ such that $\Nc_\Gamma(c) \leq N^2$, then  $Q_\Gamma[c] \in \mathcal{C}^\infty(\mathcal{U}_{\gamma,\varepsilon,r,s}^N)$ is a regular Hamiltonian in the sense of Definition \ref{def:2.1}
 and for all $z \in B_s(0,4 \varepsilon)  \cap \mathcal{U}_{\gamma,\varepsilon,r,s}^N$
\begin{equation}
\label{prokofiev}
\Norm{X_{Q_\Gamma[c]}(z)}{s} \leq C \varepsilon^{2r - 1}\Norm{c}{\ell^\infty}  \left(\frac{ N^{\alpha_r}}{\gamma}\right)^{\beta_r} .  
\end{equation}
with 
$$
\beta_r = \left\{
\begin{array}{ll}
 2r - 5 &\mbox{for} \quad \Gamma \in \Hscr_{r,\omega}, \\[2ex]
 2r -2 &\mbox{for} \quad \Gamma \in \Hscr_{r,\omega}^*, \\
\end{array}
\right. 
\quad\mbox{and}\quad 
\beta_r = \left\{
\begin{array}{ll}
 4r - 13 &\mbox{for} \quad \Gamma \in \Hscr_{r,\Omega},\\[2ex]
 4r - 5 &\mbox{for} \quad \Gamma \in \Hscr_{r,\Omega}^*. 
\end{array}
\right. 
$$
\end{lemma}
\begin{proof} 
Let $\rho = \sup_{\ell \in \mathbb{N}} m_\ell$. We have seen in \eqref{johncale} and \eqref{nico} that this quantity is bounded by $7r$. 
The functional $Q_{\Gamma}[c]$ can be written under the form 
$$
Q_{\Gamma}[c](z) =  \sum_{m = 1}^\rho \sum_{\jb \in \mathcal{R}_m} f_{\jb,m}(I) z_\jb, 
$$
where the coefficients $f_{\jb,m}(I)$, which depend on $\Gamma$, are given by \eqref{fj}. \\
Let $j_0 \in \U_2 \times \Z$ be fixed, the component of the vector field $(X_{Q_\Gamma[c]})_{j_0}(z)$ is given by 
\begin{equation}
\label{finger}
(X_{Q_\Gamma[c]})_{j_0} (z)= \frac{\partial}{\partial z_{\overline{j_0}}} Q_\Gamma[c] (z) =  \sum_{m = 1}^\rho \sum_{\jb \in \mathcal{R}_m} 
f_{\jb,m}(I)\frac{\partial}{\partial z_{\overline{j_0}}} (z_{\jb})  + z_\jb \frac{\partial}{\partial z_{\overline{j_0}}} (f_{\jb,m}(I))
\end{equation}
Let us examine the contributions coming from the first type of terms in the right-hand side. 

Let $\ell \in \pig^{-1}(\jb)$ with $m_\ell = m$ be given, and $p_\ell$, $q_\ell$ and $n_\ell$ the integers associated with one term in the decomposition \eqref{fj}. 
To control the denominators, as $z \in \mathcal{U}_{\gamma,\varepsilon,r,s}^N$ we will use the estimates \eqref{nonres1N} and \eqref{nonres2N}. More precisely, as $\sharp \kb_{\ell,\alpha} \leq m_\alpha \leq 7r$ (see \eqref{aouaismerde}), we have  
$$
|\omega_{\kb_{\ell,\alpha}}(I)|> \gamma  \varepsilon^2 N^{-\alpha_r} \langle \mu_{\min}(\kb_{\ell,\alpha})\rangle^{-2s}
$$
by definition of the weight and using the fact that $\langle \mu_1(\kb_{\ell,\alpha}) \rangle \leq \Nc_\Gamma(c) \leq N^2$. Similarly, we will use 
$$
|\Omega_{\kb_{\ell,\alpha}}(I)|> \gamma  \varepsilon^2 N^{-\alpha_r} \langle\mu_{\min}(\kb_{\ell,\alpha})\rangle^{-2s}
$$
and 
$$
|\Omega_{\hb_{\ell,\alpha}}(I)|> \gamma  \varepsilon^4 N^{-\alpha_r} . 
$$
After using these bounds, we can conclude that for $z \in \mathcal{U}_{\gamma,\varepsilon,r,s}^N$, there exists $C$ depending only on $r$ and $s$ such that 
\begin{multline*}
|f_{\jb,m}(I)| \leq C \Norm{c}{\ell^\infty}\sum_{\substack{\ell \in \pig^{-1}(\jb) \\m_\ell = m}}\frac{N^{\alpha_r(p_\ell+   q_\ell) }}{\gamma^{p_\ell + q_\ell} \varepsilon^{2p_\ell + 4q_\ell} } \prod_{\alpha = 1}^{p_\ell} \langle\mu_{\min}(\kb_{\ell,\alpha})\rangle^{2s} \\
\leq C \Norm{c}{\ell^\infty} \left( \frac{N^{\alpha_r}}{\gamma}\right)^{b_r} \varepsilon^{2r - 2 m}\prod_{\alpha = 1}^{p_\ell} \langle\mu_{\min}(\kb_{\ell,\alpha})\rangle^{2s}
\end{multline*}
where we verify that $b_r = 2r - 6$ for $\Gamma \in \Hscr_{r,\omega}$, 
and $b_r = 2r - 3$ for  $\Hscr_{r,\omega}^*$, $b_r = 2r - 14$ for $\Gamma \in \Hscr_{r,\Omega}$ and $b_r = 4r - 6$ for $\Gamma \in \Hscr_{r,\Omega}^*$. 
Indeed, we used that in all those cases we always have $p_\ell + 2 q_\ell = m_\ell - r$  by using \eqref{HP_order}. In the other hand if $\Gamma \in \Hscr_{r,\omega}$ or $\Hscr_{r,\omega}^*$, we have $q_\ell = 0$ and $p_\ell \leq 2r - 6$ or $2r - 3$ for $\Hscr_{r,\omega}$ and $\Hscr_{r,\omega}^*$ respectively. Hence the value of $b_r$ in these both cases.
Now if $\Gamma \in \Hscr_{r,\Omega}$, we have with the notations \eqref{ilpleut}-\eqref{beaucoup}, 
$p_\ell + q_\ell \leq \alpha_1 + \alpha_2 + \alpha_3 + \alpha_4 + \alpha_5 \leq \alpha_1 + 2 \alpha_5$, infering the value of $b_r$. The case or $\Gamma \in\Hscr_{r,\Omega}^*$ is treated similarly. \\
Up to a combinatorial factor, we can assume that $j_1 = \overline{j_0}$, and hence $\partial_{z_{\overline{j_0}}} (z_{\jb}) = z_{j_2} \cdots z_{j_{2m}}$, and moreover we can also assume that $j_2$ is the largest index amongst $(j_2,\ldots,j_{2m})$. Hence $j_2= \mu_1(\jb)$ or $j_2 = \mu_2(\jb)$ depending if $j_1 = \mu_1(\jb)$ or not.  Furthermore using our Hypothesis (vi) on the repartition of the derivatives (see \eqref{crux}) we have
\begin{equation}\label{mumu}\prod_{\alpha = 1}^{p_\ell} \langle\mu_{\min}(\kb_{\ell,\alpha})\rangle^{2s}\leq \prod_{\alpha = 3}^{m} \langle\mu_{\alpha}(\jb)\rangle^{s}.\end{equation}
With these choices and this estimate, we  get
\begin{multline*}
\sum_{j_0 \in \U_2 \times \Z}\langle j_0 \rangle^{s} \left|\sum_{\jb \in \mathcal{R}_m} 
f_{\jb,m}(I)\frac{\partial}{\partial z_{\overline{j_0}}} (z_{\jb})\right| 
\\
\leq 
C  \Norm{c}{\ell^\infty} \left( \frac{N^{\alpha_r}}{\gamma}\right)^{b_r} \varepsilon^{2r - 2 m} \sum_{\jb = (j_1,\ldots j_{2m}) \in \Rc_m}  \prod_{\alpha = 3}^{m} \langle\mu_{\alpha}(\jb)\rangle^{s} \langle j_1\rangle^{s} |z_{j_2} \cdots z_{j_{2m}}|
\\
\leq 
C  \Norm{c}{\ell^\infty} \left( \frac{N^{\alpha_r}}{\gamma}\right)^{b_r} \varepsilon^{2r - 2 m}  \sum_{\jb = (j_1,\ldots j_{2m}) \in \Rc_m}  \langle j_1\rangle^{s}  |z_{j_2}| |v_{j_3}| \cdots |v_{j_{2m}}|
\end{multline*}
where $v_k = \langle k \rangle z_k$ is in $\ell^1$ and of norm smaller than $\varepsilon$ by assumption. 
Since $\jb\in\Rc_m$ it satisfies the zero-momentum condition and thus $\langle j_1\rangle \leq (2m-1) \langle j_2\rangle$ . Hence the last sum is bounded by 
$$
\sum_{(j_2,\ldots j_{2m})} |v_{j_2}| |v_{j_3}| \cdots |v_{j_{2m}}| \leq C \varepsilon^{2m - 1}. 
$$
By summing with respect to $m$, we get that the first contribution of the right-hand side of \eqref{finger} for the estimate of 
$\Norm{X_{Q_\Gamma[c]} (z) }{s} = \sum_{j} \langle j \rangle^s |(X_{Q_\Gamma[c]})_j(z)|$ satisfies the bound 
\eqref{prokofiev}.

Now we study the second contribution in the equation \eqref{finger}. To this aim, let us write
$$
f_{\jb,m}(I) =   \sum_{ \substack{\ell \in \pig^{-1}(\jb) \\m_\ell = m} } c_\ell f_{\jb,m}^\ell(I)
$$
where $f_{\jb,m}^\ell(I)$ correspond to the decomposition \eqref{fj}. 
In view of the structure of $f_{\jb,m}^\ell(I)$, we have 
\begin{multline}
\label{ataaable}
z_\jb \frac{\partial}{\partial z_{\overline{j_0}}} (f_{\jb,m}^\ell(I)) = \\
- z_\jb z_{j_0} f_{\jb,m}^\ell(I)
\left( \sum_{\alpha = 1}^{n_\ell} 
\frac{\partial_{I_j} \omega_{\kb_{\ell,\alpha}}}{\omega_{\kb_{\ell,\alpha}}}
+ \sum_{\alpha = n_\ell + 1}^{p_\ell} 
\frac{\partial_{I_j} \Omega_{\kb_{\ell,\alpha}}}{\Omega_{\kb_{\ell,\alpha}}}
+ \sum_{\beta = n_1}^{q_\ell} 
\frac{\partial_{I_j} \Omega_{\hb_{\ell,\beta}}}{\Omega_{\hb_{\ell,\beta}}}
\right).  
\end{multline}
Let us assume that $j_1 = \mu_1(\jb)$ and $j_2 = \mu_2(\jb)$. 
We have with the previous notation and using again \eqref{mumu}
$$
\langle j_0 \rangle^s|z_\jb z_{j_0} f_{\jb,m}^\ell(I)| \leq 
C \Norm{c}{\ell^\infty} \left( \frac{N^{\alpha_r}}{\gamma}\right)^{b_r} \varepsilon^{2r - 2 m}
 |z_{j_1}| |z_{j_2}| |v_j| |v_{j_3}| \cdots |v_{j_{2m}}|. 
$$
Now as $\partial_{I_j} \omega_{\kb_{\ell,\alpha}} = \pm 1$, we  have by using \eqref{crux} and the fact that $n_\ell \leq m-r$, 
$$
\left|\frac{\partial_{I_j} \omega_{\kb_{\ell,\alpha}}}{\omega_{\kb_{\ell,\alpha}}} \right| \leq \frac{C}{\gamma \varepsilon^{2}} N^{\alpha_r}
\langle \mu_{\min}(\kb_{\ell,\alpha}) \rangle^{2s} \leq \frac{C}{\gamma \varepsilon^2} N^{\alpha_r}
\langle \mu_{3}(\jb) \rangle^{2s}, 
$$
and the contribution corresponding to this term in the expression
$$
\sum_{j_0} \langle j_0 \rangle^s \left|\sum_{\jb \in \mathcal{R}_m} 
 z_\jb \frac{\partial}{\partial z_{\overline{j_0}}} (f_{\jb,m}(I))\right| \leq 
 C\sum_{j_0}\sum_{\substack{\ell \in \pig^{-1}(\jb)\\ m_\ell = m}}  \langle j_0 \rangle^s \left| z_\jb \frac{\partial}{\partial z_{\overline{j_0}}} (f_{\jb,m}^\ell(I))\right| 
$$
is thus bounded by 
\begin{multline*}
C \Norm{c}{\ell^\infty}\left( \frac{N^{\alpha_r}}{\gamma}\right)^{b_r + 1} \varepsilon^{2r - 2 m - 2}   \sum_{j_0, \jb}
\langle \mu_{3}(\jb) \rangle^{2s} |z_{j_1}| |z_{j_2}| |v_{j_0}| |v_{j_3}| \cdots |v_{j_{2m}}| \\
\leq 
C  \Norm{c}{\ell^\infty}\left( \frac{N^{\alpha_r}}{\gamma}\right)^{b_r + 1} \varepsilon^{2r  - 1}, 
\end{multline*}
as $j_1$ and $j_2$ are larger than the third largest index in $\jb$. 
By summing with respect to $m$, the global contribution of these terms satisfies the estimate \eqref{prokofiev}. 

We obtain similar estimates for the terms in \eqref{ataaable} associated with the part of $\Omega_{\kb_{\ell,\alpha}}$ and $\Omega_{\hb_{\ell,\alpha}}$ coming from $Z_4$. 
It remains to estimates the part coming from $Z_6$ in \eqref{ataaable}. Typically a term of the form $\frac{\partial_{I_j} \Omega_{\kb_{\ell,\alpha}}}{\Omega_{\kb_{\ell,\alpha}}}$ will yield a contribution of the form 
$$
\sum_{p} \alpha_p \frac{I_p }{\Omega_{\kb_{\ell,\beta}}}
$$
where $\alpha_p$ are uniformly bounded in $p$. 
The global contribution of these term, by estimating $\Omega_{\kb_{\ell,\alpha}}$ and $\Omega_{\hb_{\ell,\beta}}$ by $\gamma \varepsilon^4 N^{\alpha_r}$ will be 
\begin{multline*}
C \Norm{c}{\ell^\infty}\left( \frac{N^{\alpha_r}}{\gamma}\right)^{b_r + 1} \varepsilon^{2r - 2 m - 4} 
\sum_{j_0,p,\jb}
 |z_{j_1}| |z_{j_2}| |v_{j_0}| |v_{j_3}| \cdots |v_{j_{2m}}||z_p|^2 \\ 
 \leq
 C \Norm{c}{\ell^\infty}\left( \frac{N^{\alpha_r}}{\gamma}\right)^{b_r + 1} \varepsilon^{2r - 1}. 
 \end{multline*}
This shows the result with $\beta_r = b_r + 1$. 
 \end{proof}

\subsubsection{Homological equations}
In this section we will see that our class is particularly well adapted to the solution homological equations, the central step in the construction of normal forms. Actually, this class was constructed precisely to be invariant by Poisson bracket and by solution of the homological equation with $Z_4$ or $Z_4 + Z_6$.\\
We define the set $\Ascr_{r}$ as the subset of elements $\Gamma=(\pig,\kbsf,\hbsf,n)$ of $\Hscr_{r}$ for which $Q_\Gamma[c]$ depends only on the actions. This means that for all $\ell \in \Z^*$, 
$\irr(\pig_\ell) = \emptyset$.\\ Then we define  $\Rscr_{r}$ the subset of elements $\Gamma=(\pig,\kbsf,\hbsf,n)$ of $\Hscr_{r}$ such that for all $\ell \in \Z^*$, 
$\irr(\pig_\ell) \neq \emptyset$. 
Note that for all $\Gamma\in \Hscr_{r}$ there exists $A \in \Ascr_{r}$ and $R \in \Rscr_{r}$ such that for all $c \in \ell^{\infty}_\Gamma$, $c \in \ell^{\infty}_A \cap\ell^{\infty}_R$, and 
\begin{equation}
\label{decomposition}
Q_\Gamma[c] = Q_{A}[c] + Q_{R}[c]. 
\end{equation}
We also naturally define the corresponding subsets of $\Fscr_{r}$
\begin{equation}
\label{FA}
\Fscr_{r}^A := \{F = Q_{\Gamma}[c],\, \ \, \Gamma \in \Ascr_{r}, \quad c \in  \ell^\infty_{\Gamma}(\mathbb{Z}^*)   \}. 
\end{equation}
the functionals of order $r$ depending only on the actions, and 
\begin{equation}
\label{FR}
\Fscr_{r}^R := \{F = Q_{\Gamma}[c],\, \ \, \Gamma \in \Rscr_{r}, \quad c \in  \ell^\infty_{\Gamma}(\mathbb{Z}^*)   \}. 
\end{equation}
Naturally, we define $\Rscr_{r,\omega},\Rscr_{r,\Omega},\Rscr_{r,\omega}^*$, $\Rscr_{r,\Omega}^*$ as the restrictions  of $\Rscr_r$ to  $\Hscr_{r,\omega}$, $\Hscr_{r,\Omega}$, $\Hscr_{r,\omega}^*$, $\Hscr_{r,\Omega}^*$.

With this formalism, the resolution of the homological equation is  trivial, after noticing that $Z_4$ and $Z_6$ commute with terms depending only of the actions and by using the relations \eqref{ravel}. 
\begin{lemma} 
\label{lemmehomo} Let $\Gamma = (\pig,\kbsf,\hbsf,n)\in \Rscr_{r,\Omega}$. 
Defining $\Gamma' = (\pig,\kbsf,\hbsf',n)\in \Hscr_{r-2,\Omega}^*$ with 
\[ \forall \ell\in \mathbb{Z}^*, \quad \hbsf'_{\ell} = (\hbsf_\ell,\irr(\pig_{\ell})),  \]
Then for all $c\in \ell_\Gamma^{\infty}(\Z^*)$, $Q_\Gamma'[c]$ is solution of the homological equation
\begin{equation}
\label{eqhomo} \left\{Z_4+Z_6, Q_{\Gamma'}[c]  \right\} = Q_\Gamma[c],
\end{equation}
and we have 
\[  \Nc_{\Gamma'}(c)=\Nc_{\Gamma}(c).\]
\end{lemma}
We will also need to solve a homological equation associated with $Z_4$:
\begin{lemma} 
\label{rapido}
 Let $\Gamma = (\pig,\emptyset,\emptyset,0)\in \Rscr_{3,\omega}$. 
Defining $\Gamma' = (\pig,\kbsf',\emptyset,n')\in \Hscr_{2,\omega}^*$ with 
\[ \forall \ell\in \mathbb{Z}^*, \quad \kbsf'_{\ell} = (\pig_{\ell}) \textrm{ and } n_\ell'=1,  \]
Then for all $c\in \ell_\Gamma^{\infty}(\Z^*)$, $Q_\Gamma'[c]$ is solution of the homological equation
\[ \left\{Z_4, Q_{\Gamma'}[c]  \right\} = Q_\Gamma[c],\]
and we have 
\[  \Nc_{\Gamma'}(c)=\Nc_{\Gamma}(c).\]
\end{lemma}

\subsubsection{Stability by Poisson bracket}

Now comes the main technical result of this paper: the stability of our classes by Poisson bracket.
\begin{lemma} 
\label{cap} Let $W\in \{\omega,\Omega\}$, let $\Gamma\in \Hscr_{r,W}^*$ and let $\Gamma'\in \Hscr_{r',W}$. \\
There exists $\Gamma''\in \Hscr_{r'',W}$, where
\[ r'' = r+r'-1 \]
 and there exists a bilinear continuous application
\[g: \ell^\infty_{\Gamma}(\mathbb{Z}^*) \times \ell^\infty_{\Gamma'}(\mathbb{Z}^*) \to \ell^\infty_{\Gamma''}(\mathbb{Z}^*) \]
such that for all $c\in \ell^\infty_{\Gamma}(\mathbb{Z}^*)$, $c'\in \ell^\infty_{\Gamma'}(\mathbb{Z}^*) $
\[ \{ Q_\Gamma[c],Q_{\Gamma'}[c'] \} = Q_{\Gamma''}[g(c,c')]\]
and
\begin{equation}
\label{caribou}\Nc_{\Gamma''}(g(c,c')) \leq \max( \Nc_{\Gamma}(c), \Nc_{\Gamma'}(c')). 
\end{equation}
\end{lemma}
\begin{proof}
We postpone the proof to appendix \ref{appendixA}
\end{proof}

\section{Rational normal form}\label{RNF}
In this section we prove Theorem \ref{mainth} for \eqref{nls}. As announced in section \ref{sketch} this is achieved in three steps: First we kill the non resonant monomials in the Hamiltonian $P$ by using $Z_2$ as normal form (Section \ref{killbill0}), then we kill the remaining non integrable terms ($K_6$) of order 6 by including the resonant part of order 4, namely $Z_4$ (which is integrable),  in the normal form (Section \ref{killbill1}), finally we kill all the non integrable terms up to order $r$ by including the integrable part of order 6,  namely  $Z_6$, in the normal form (Section \ref{killbill2}).

\subsection{Resonant normal form}\label{killbill0}
In this section we apply a Birkhoff normal form procedure to kill iteratively the non resonant monomials up to order $r$ of the Hamiltonian $P$. 
\begin{theorem}\label{BNF}
For all $r \geq 4$ and $s\geq0$, 
there exists $\tau_2$ a $\mathcal{C}^1$ symplectomorphism in a neighborhood of the origin in $\ell_1^s$  close to the identity:
$$
\Norm{\tau_2 (z) - z }{s} \leq C \Norm{z}{s}^3 
$$
which puts $H$ in normal form up to order 6:
\begin{equation}
\label{bach}
H\circ \tau_2 (z) = Z_2(I) + Z_4(I) + Z_6(I) + K_6(z) + \sum_{m = 4}^r K_{2m}(z) + R(z)
\end{equation} where for all $m = 4,\ldots,r$, $K_{2m}$ is a homogeneous resonant polynomial of order $m$
\begin{equation}
 \label{polynomes2}
 K_{2m}(z) =P{[c^{(m)}]}(z) =\sum_{\jb \in \Rc_m} c^{(m)}_\jb z_\jb, \quad\mbox{with}\quad \quad  c^{(m)}\in\ell^\infty( \mathcal{R}_m),
 \end{equation}
and where $K_6(z)$ contains only irreducible monomials
\begin{equation}
 \label{polynomes3}
 K_{6}(z)  =\sum_{\jb \in \Rc_3\cap\Ic} c_\jb z_\jb, \quad\mbox{with}\quad \quad  c\in\ell^\infty( \mathcal{R}_3).
 \end{equation}
Moreover, $R$ is smooth in a neighborhood of the origin and satisfies
\begin{equation}
 \label{polynomes4}
\Norm{X_R(z)}{s} \leq C \Norm{z}{s}^{2r + 1}, 
\end{equation}
for $z$ small enough in $\ell_s^1$. 

\end{theorem}
\proof The proof is standard (it first appears in \cite{KP96}) except for the calculation of the resonant terms of order six. For convenience of the reader we give the details.\\
We have $H=\sum_{a\in\Z}a^2\xi_a\eta_a+P$ where $P$ is given by \eqref{P} and we write 
$$P=\sum_{m=1}^r P_{2m}+R_{2r+2}$$
where
\begin{align}\label{P2m}
P_{2m}&=\frac{\varphi^{(m-1)}(0)}{m!}\frac{1}{2\pi}\int_\T(\sum_{a \in \Z} \xi_a e^{iax})(\sum_{b \in \Z} \eta_b e^{-ib x})dx\\
\nonumber &= \frac{\varphi^{(m-1)}(0)}{m!}\sum_{a_1+\cdots+a_m=b_1+\cdots b_m}\xi_{a_1}\cdots\xi_{a_m}\eta_{b_1}\cdots\eta_{b_m}\\
\nonumber &=\frac{m!\ \varphi^{(m-1)}(0)}{(2m)!}\sum_{\jb\in \Mc_m} \  z_\jb 
\end{align}
and $R_{2r+2}$ is a remainder of order $2r+2$ {\em i.e.} $R_{2r+2}\in \Hc_s(\ell^1_s)$ and $\Norm{X_{R_{2r+2}}(z)}{s} \leq C \Norm{z}{s}^{2r + 1}$. We note that the integrable Hamiltonian given by \eqref{Z2} reads $$Z_2  =\sum_{a\in\Z}a^2\xi_a\eta_a+P_2.$$
First we kill the non resonant monomials of order 4 by a change of variables $\Psi_4$. We search for $\Psi_4=\Phi^1_{\chi_4}$, the time one flow of $\chi_4 $ of a polynomial Hamiltonian  homogeneous of order 4: 
\begin{equation*} 
\chi_4 = \sum_{j\in \Mc_2} a_{\jb}\  z_\jb.
\end{equation*}
For any  $F\in\Hc_s$,   the Taylor expansion of $F\circ \Phi^t_\chi$ between $t=0$ and $t=1$ gives
$$F\circ \Phi^1_\chi = F+ \{F ,\chi \}+\frac 1 2 \int_0^1(1-t)\{\{F,\chi\},\chi\}\circ \Phi^t_\chi \text{d}t.$$
Applying this formula to $H =Z_2+P_4+R_6$ we get
\begin{equation*} 
H\circ \Psi_4 =Z_2+(P-P_2)+ \{Z_2 ,\chi \}+\{(P-P_2),\chi\}+\frac 1 2 \int_0^1(1-t)\{\{H,\chi\},\chi\}\circ \Phi^t_\chi \text{d}t.
\end{equation*}
In this formule the homogeneous part of order 4 is $P_4+ \{Z_2 ,\chi_4 \}$.
 Then we set
 \begin{equation*} 
\chi_4 := \frac{1}{12} {\varphi'(0)}\sum_{\jb\in\Mc_2\setminus\Rc_2} \frac {1} {i\Delta_\jb} z_\jb,\quad Z_4 =  \frac{1}{12} {\varphi'(0)} \sum_{\jb\in\Rc_2} z_\jb.
\end{equation*}
We note that at this stage there are no small divisors problem since $|\Delta_\jb|\geq1$ except when $\jb\in\Rc$ in which case $\Delta_\jb=0$. So
$\chi_4$ and $Z_4$ are well defined homogeneous polynomials of order 4 and, using \eqref{ravelo} they solve   the homological equation
\begin{equation} \label{homo}
Z_{4}=P_4+ \{Z_2 ,\chi_4 \}.
\end{equation}
Further 
$H\circ \Psi_4 = Z_2 +Z_{4}+R_{6}$ with \begin{equation}\label{R}
R_{6}=(P-P_2-P_4)+\{P-P_2 ,\chi_4 \}+ \frac 1 2 \int_0^1(1-t)\{\{H,\chi\},\chi\}\circ \Phi^t_\chi \text{d}t
\end{equation}
is a smooth Hamiltonian beginning at order 6 {\em i.e.} $R_6=0(z^6)$.\\ 
We can iterate this procedure to kill successively the non resonant monomials of order $6,\cdots,2r$. Then we get the existence of a symplectomorphism $\Psi$ close to the identity and defined on a neighborhood of the origin in $\ell_1^s$ such that
\begin{equation} \label{Bir}H\circ \Psi = Z_2 +Z_{4}+Z'_{6}+ \sum_{m = 4}^r K_{2m}(z) + R(z),\end{equation}
where $K_{2m}$ are resonant polynomials of the form \eqref{polynomes2}, $R$ is a smooth remainder satisfying \eqref{polynomes4} on a neighborhood of the origin and $Z'_6$ is a resonant monomial of order 6. It remains to compute $Z_4$ and $Z'_6$.\\
Concerning $Z_4$ we have 
$$Z_4 = \frac12  {\varphi'(0)} \sum_{\substack{a_1+a_2=b_1+b_2\\a^2_1+a^2_2=b^2_1+b^2_2}} \xi_{a_1}\xi_{a_2}\eta_{b_1}\eta_{b_2}$$
but 
$$a_1+a_2=b_1+b_2 \quad \text{and}\quad a^2_1+a^2_2=b^2_1+b^2_2$$
 leads to
$$\{a_1,a_2\}=\{b_1,b_2\}.$$
Therefore we get as anounced in \eqref{Z4}
\begin{align*}Z_4 =Z_4(I)=&  {\frac12\varphi'(0)} \sum_{a,b\in\Z} I_aI_b(2-\delta_{ab})\\
=&  {\varphi'(0)}  (\sum_{a\in\Z}I_a)^2-\frac{1}{2}\varphi'(0)\sum_{a\in\Z}I_a^2.
\end{align*}
After the first two Birkhoff procedures we get\footnote{Recall that the Poisson bracket of a Polynomial of order $m$ with a polynomial of order $m$ is a polynomial of order $m+n-2$. }   $Z'_6=Z_{6,1}+Z_{6,2}$ where $Z_{6,1}$ is the resonant part of $\{P_4 ,\chi_4 \}+\frac12\{\{Z_2,\chi_4\},\chi_4\}$ and $Z_{6,2}$ is the resonant part of $P_6$.\\ Let us start with the latter, following \eqref{P2m} we have
$$Z_{6,2}=\frac{\varphi^{''}(0)}{6}\sum_{\substack{a_1+a_2+a_3=b_1+b_2+b_3\\a^2_1+a^2_2+a_3^2=b^2_1+b^2_2+b_3^2}} \xi_{a_1}\xi_{a_2}\xi_{a_3}\eta_{b_1}\eta_{b_2}\eta_{b_3}.$$
If $\{a_1,a_2,a_3\}\cap\{b_1,b_2,b_3\}\neq\emptyset$ then, assuming for instance $a_3=b_3$, we get $((1,a_1),(1,a_2),(-1,b_1),(-1,b_2))\in\Rc_2$ which leads as before to $\{a_1,a_2\}=\{b_1,b_2\}.$ So either $((1,a_1),(1,a_2),(1,a_3),(-1,b_1),(-1,b_2),(1,b_3))\in \Ic$ or $\{a_1,a_2,a_3\}=\{b_1,b_2,b_3\}$, i.e.
\begin{align}\nonumber Z_{6,2}=&K'_6(z)+\frac{\varphi^{''}(0)}{6}\sum_{\{a_1,a_2,a_3\}=\{b_1,b_2,b_3\}} \xi_{a_1}\xi_{a_2}\xi_{a_3}\eta_{b_1}\eta_{b_2}\eta_{b_3}\\
\nonumber=&K_6(z)+ \frac{\varphi^{''}(0)}{6}\left(\sum_{\substack{a \neq b, \  a \neq c \\ b \neq c}} \  6I_aI_bI_c+\sum_{a\neq b} \  9I^2_aI_b + \sum_a I_a^3\right)\\
\label{Z62}=&K'_6(z)+ \frac{\varphi^{''}(0)}{6}\left(6(\sum_{a\in\Z}I_a)^3-9(\sum_{a\in\Z} I_a^2)(\sum_{a\in\Z}I_a)+4\sum_{a\in\Z}I_a^3  \right)\end{align}
where $K'_6$ is of the form \eqref{polynomes3}.\\
It remains to compute $Z_{6,1}$. First we notice that using the homological equation \eqref{homo} we get
$$\{P_4 ,\chi_4 \}+\frac12\{\{Z_2,\chi_4\},\chi_4\}=\{P_4 ,\chi_4 \}+\frac12\{Z_4-P_4 ,\chi_4 \}=\{Z_4 ,\chi_4 \}+\frac12\{Q_4 ,\chi_4 \}$$
where $Q_4$ denotes the non resonant part of $P_4$: 
$$Q_4=P_4-Z_4=Z_4 =  \frac{1}{12} {\varphi'(0)} \sum_{\jb\in\Mc_2\setminus\Rc_2} z_\jb.$$
  We easily verify that the Poisson bracket of a resonant monomial with a non resonant monomial cannot be resonant. Therefore
$Z_{6,1}$ is the resonant part of 
\begin{align}\label{Z61-1}&\frac12\{Q_4 ,\chi_4 \}=\frac{\varphi'(0)^2}{288}\left\{\sum_{\jb\in\Mc_2\setminus\Rc_2} z_\jb,\sum_{\kb\in\Mc_2\setminus\Rc_2} \frac {1} {i\Delta_\kb} z_\kb \right\}\\
\label{Z61-2}&=\frac{\varphi'(0)^2}{8}\left\{ \sum_{\substack{a_1+a_2=b_1+b_2\\a^2_1+a^2_2\neq b^2_1+b^2_2}} \xi_{a_1}\xi_{a_2}\eta_{b_1}\eta_{b_2}, \sum_{\substack{a_1+a_2=b_1+b_2\\a^2_1+a^2_2\neq b^2_1+b^2_2}} \frac {\xi_{a_1}\xi_{a_2}\eta_{b_1}\eta_{b_2}} {i(a^2_1+a^2_2-b^2_1-b^2_2)}  \right\}.\end{align} 
Then we proceed as for $Z_{6,2}$ to conclude that
$$Z_{6,1}= K''_6+Z'_6(I)$$
where $K''_6$ is of the form\footnote{In fact a long but straightforward computation leads to $K''_6=0$ which means that, up to order 6, the Birkhoff normal form of the cubic NLS depends only on the actions. A sort of reminiscence of the complete integrability. Nevertheless this result is not needed in this paper and the calculation is long...  } \eqref{polynomes3} and $Z_6(I)$ is the part of $\frac12\{Q_4 ,\chi_4 \}$ depending only on the actions.
So we can write
$$Z'_{6}(I)=\sum_{a\in\Z} \alpha_a\ I_a^3+\sum_{a\neq b\in\Z} \beta_{ab}\ I_a^2I_b+\sum_{\substack{a \neq b, \  a \neq c \\ b \neq c}} \gamma_{abc}\ I_aI_bI_c$$ 
where the values of  $\alpha_a$, $\beta_{ab}$, $\gamma_{abc}$ are compute in Lemma \ref{argh} below.\\
Thus we get
$$Z_{6,1}= K''_6-\frac{1}{2}\varphi'(0)^2\sum_{a\neq b\in\Z} \frac{1}{(a-b)^2}\ I_a^2I_b$$
and using \eqref{Z62}
$$Z'_6(z)=Z_{6,1}+Z_{6,2}=K_6(z)+Z_6(I)$$
where $K_6=K'_6+K''_6$ is of the form \eqref{polynomes3} and $Z_6$ is given by \eqref{Z6} as expected.
\endproof
\begin{lemma}\label{argh}{\ }
The coefficients of the term $Z_6'(I)$ satisfy
\begin{itemize}
\item[(i)]$\alpha_a=0$ for all $a\in\Z$,
\item[(i)]$\gamma_{abc}=0$ for all $a\neq b$, $a\neq b$ and $b \neq c\in\Z$,
\item[(i)] $\beta_{ab}=\frac{-\varphi'(0)^2}{2(a-b)^2}$ for all $a\neq b\in\Z$.
\end{itemize}
\end{lemma}
\proof We use formulas \eqref{Z61-1} and \eqref{Z61-2} to identify the terms of $\frac12\{Q_4 ,\chi_4 \}$ depending only on actions.\\
(i) If $I_a^3=\frac{\partial z_\jb}{\partial \xi_b}\frac{\partial z_{\kb}}{\partial \eta_b}$ with $z_\jb$ a monomial from $Q_4$ and $z_\kb$ a monomial from $\chi_4$  then necessarily  $z_\jb=\xi_b\xi_a\eta_a^2$ and $z_\kb=\xi_a^2\eta_a\eta_b$. But since  $\jb,\kb\in \Mc_2$ we get $a=b$ and thus $\jb=\kb$ is resonant which is not possible.

\medskip

\noindent(ii) Assume $I_aI_bI_c=\frac{\partial z_\jb}{\partial \xi_d}\frac{\partial z_{\kb}}{\partial \eta_d}$ with $\jb,\kb\in \Mc_2\setminus \Rc_2$. We consider two different cases:\\ 
$\bullet$ $z_\jb=\xi_a\xi_d\eta_a\eta_b$ and $z_\kb=\xi_c\xi_b\eta_d\eta_c$. Since $\jb\in\Mc_2$ we get $b=d$ which is incompatible with $\jb\notin \Rc_2$. All similar cases obtained by permutation  of $a,b,c$ lead to the same incompatibility.\\
$\bullet$ $z_\jb=\xi_a\xi_d\eta_c\eta_b$ and $z_\kb=\xi_c\xi_b\eta_d\eta_a$. Since $\jb\in\Mc_2$ we get $d=c+b-a$ and then we calculate $\Delta_k=-2(a-b)(a-c)$. By permutation we get up to an irrelevant constant $c$
$$c\gamma_{abc}=\frac1{(a-b)(a-c)}+\frac1{(b-a)(b-c)}+\frac1{(c-a)(c-b)}=0.$$

\medskip

\noindent(iii) Assume $I_a^2I_b=\frac{\partial z_\jb}{\partial \xi_c}\frac{\partial z_{\kb}}{\partial \eta_c}$ with $\jb,\kb\in \Mc_2\setminus \Rc_2$. We consider different cases:\\ 
$\bullet$   $z_\jb=\xi_a\eta_a^2\xi_c$ and $z_\kb=\xi_a\xi_b\eta_b\eta_c$. Since $\jb\in\Mc_2$ we get $a=c$ which is incompatible with $\jb\notin \Rc_2$.\\
$\bullet$  $z_\jb=\xi_b\eta_a\eta_b\xi_c$ and $z_\kb=\xi_a^2\eta_a\eta_c$. We get again using the zero momentum condition that $a=c$ which is incompatible with $\jb\notin\Rc_2$.\\
$\bullet$ $z_\jb=\xi_b\eta_a^2\xi_c$ and $z_\kb=\xi_a^2\eta_b\eta_c$. The zero momentum leads to $c=2a-b$ and we get $\Delta_\kb= -\Delta_\jb= 2a^2-b^2-(2a-b)^2=-2(a-b)^2$. So $\jb,\kb\in \Mc_2\setminus \Rc_2. $\\
It remains to calculate the number of occurrences of this configuration in $D(z)$: we can exchange $a$ and $b$ in $z_\jb$ and in $z_\kb$ and we can exchange $\jb$ and $\kb$. So 8 occurrences in \eqref{Z61-2} and thus
$$\beta_{ab}= \frac{\varphi'(0)^2}8 8\frac1{-2(a-b)^2}=\frac{-\varphi'(0)^2}{2(a-b)^2}.$$

\endproof

\subsection{Elimination of the quintic term by the cubic}\label{killbill1}

{\em It's mercy, compassion, and forgiveness I lack. Not rationality.} Beatrix Kiddo in ``Kill Bill: Volume 1" (Q. Tarentino, 2003). 

In this section we will truncate the new Hamiltonian $H \circ \tau_2$ and eliminate the resonant term $K_6$ with the help of $Z_4$. Moreover, we will show that the new Hamiltonian admits a development with terms of the form $Q_\Gamma[c]$ with $\Gamma$ in the class $\Hscr_{r,\omega}$. 
\begin{proposition} 
Let $r\geq4$ and $s\geq0$ be given. 
There exist $N_0$,  $\Gamma_{2m} \in \Hscr_{m,\omega}$ for $4 \leq m \leq r$, and a constant $C$ , such that for all $0<\varepsilon\leq1,\gamma>0$, $N\geq1$  with
\begin{equation}
\label{CFL}
C \varepsilon^{\frac{3}{2}} \left(\frac{N^{\alpha_r}}{\gamma}\right)^3 < 1,
\end{equation}
 there exist
\begin{itemize}
\item $c_8,\dots,c_{2r} \in \ell^\infty_{\Gamma_8} \times \dots \times \ell^\infty_{\Gamma_{2r}}$
\item  $\tau_4 : B_s(0,2\varepsilon)  \cap \mathcal{U}_{\gamma/2,\varepsilon,r,s}^N \to \ell^1_s$ a $\mathcal{C}^1$ injective  symplectomorphism, 
\item $R^{\rm high},R^{\rm ord} \in \mathcal{C}^1(B_s(0,2\varepsilon)  \cap \mathcal{U}_{\gamma/2,\varepsilon,r,s}^N)$ 
\end{itemize}
such that
\begin{equation}
\label{norm_form_ind}
H\circ \tau_2 \circ \tau_4 = Z_2+Z_4+Z_6 + \sum_{m=4}^r Q_{\Gamma_{2m}}[c_{2m}] + R^{\rm high} + R^{\rm ord},
\end{equation}
and we have the following bounds
\begin{itemize}
\item for all $m=4\dots r$, $\Nc_{\Gamma_{2m}}(c_{2m}) \leq N^2$
\item for all $m=4\dots r$, $\|c_{2m}\|_{\ell^{\infty}}\leq C_r$
\item For all $z \in B_s(0,2\varepsilon)  \cap \mathcal{U}_{\gamma/2,\varepsilon,r,s}^N$, we have 
$$
\|X_{R^{\rm high}}(z)\|_{\ell_s^1} \leq  C \varepsilon^{5} N^{-s}  \quad \mbox{and}\quad \|X_{R^{\rm ord}}(z)\|_{\ell_s^1} \leq C \varepsilon^{2r + 1}  \left(  \frac{N^{\alpha_{r+1}}}\gamma \right)^{2r-3}$$
\item $\tau_4$ takes values in $z \in B_s(0,3\varepsilon)  \cap \mathcal{U}_{\gamma/4,\varepsilon,r,s}^N$ and satisfies the estimates
\begin{equation}
\label{esto4}
\Norm{\tau_4(z) - z}{s} \leq C \varepsilon^\frac{3}{2} \gamma N^{-\alpha_r}. 
\end{equation}
\end{itemize}
\end{proposition}

\begin{proof}

The proof is divided into two steps. First we introduce a cut-off in frequency allowing to work only with rational functionals whose irreducible monomials have their largest index bounded by $N^2$. Then we will define the change of variable $\tau_4$ and express the Hamiltonian in the new variable. 

\medskip
\noindent {\bf First step: Truncation}. 
For all $4\leq m\leq r$, we decompose $K_{2m}(z)$ of \eqref{polynomes2} into $ K_{2m}^N + R_{2m}^N$, with 
$$
K_{2m}^N(z) = \sum_{\substack{\jb \in \Rc_m \\\langle \mu_3(\jb) \rangle\leq \nu_m N}} b_\jb z_\jb,
\quad\mbox{and}\quad
R_{2m}^N(z) = \sum_{\substack{\jb \in \Rc_m \\\langle \mu_3(\jb) \rangle > \nu_m N}} b_\jb z_\jb,
$$
where $\nu_m$ will be constant that will be specified later. 
Let $j_0 \in \mathbb{U}_2 \times \Z$ be a given index.  Up to a combinatorial factor, we have 
$$
|X_{R_{2m}}(z)_{j_0} | \leq  C_m \sum_{\substack{\jb = (j_1,\ldots,j_{2m}) \in \Rc_m \\ j_1 = \overline{j_0}}} |z_{j_2} \cdots z_{j_{2m}}|. 
$$
Let us assume that $j_2$ is the highest index in the monomial in the right-hand side, and $j_3$ the second highest. We thus have at least $\langle j_3 \rangle > \nu_m N$. 
By using the zero momentum condition, we have 
$$
\Norm{X_{R_{2m}}(z)}{s} \leq CN^{-s} \sum_{j_2,\ldots,j_{2m}} \langle j_2 \rangle^{s} |z_{j_2}|\langle j_3 \rangle^{s}| | z_{j_3}\cdots z_{j_{2m}}|
$$
where the constant $C$ depends on $m$ and $s$. 
It is thus easy to verify that when $\Norm{z}{s} \leq \varepsilon$, we have 
$$
\Norm{X_{R_{2m}}(z)}{s} \leq C \varepsilon^{2m - 1} N^{-s}. 
$$
If we define 
$R^{\rm high} = \sum_{m = 3}^{r-1}R_{2m}$, we easily verify that it satisfies the hypothesis of the Proposition. 

Now let us consider $K_{2m}^N$. By Remark \ref{loureed}, there exists $\Lambda_{2m} \in \Hscr_{m,\omega}$ and $b_{2m}: \Z^* \to \C$ with $\Norm{b}{\ell^\infty_{\Lambda_{2m}}} \leq C_m$  and such that $$K^N_{2m} = Q_{\Lambda_{2m}}[b_{2m}].$$ 
Let us prove that, up to a choice of $\nu_m$,
$\Nc_{\Lambda_{2m}}(b_{2m}) \leq N^2$,

For a given monomial $\jb = (j_1,\ldots,j_{2m}) = \pig_\ell$ up to a combinatorial factor, let us assume that $j_1$ and $j_2$ correspond to the first and second largest indexes. We thus have $\langle j_p \rangle \leq \nu_m N$ for  $ p = 3,\ldots,2m$. 
Let us denote by $j_p = (\delta_p,a_p) \in \mathbb{U}_2 \times \Z$. We have by definition of $\mathcal{R}_m$, 
\begin{equation}
\label{seminaire}
|\delta_1 a_1 + \delta_2 a_2 | \leq (2m - 2) \nu_m N
\quad \mbox{and}\quad |\delta_1 a_1^2 + \delta_2 a_2^2 | \leq (2m - 2) \nu_m^2 N^2. 
\end{equation}
If $\delta_1$ and $\delta_2$ are of the same sign, this implies that $|a_1|$ and $|a_2|$ are smaller than $(2m - 2)^{1/2} \nu_m N\leq N^2$ for $\nu_m$ small enough. If $\delta_1$ and $\delta_2$ are opposite signs, then two cases can occur. 

\begin{itemize}
\item $a_1 = a_2$. In this case, $j_1 = \bar j_2$ and the product $z_{j_1} z_{j_2} = I_{a_1}$ is an action. In this situation, $\irr(\jb) \subset (j_3,\cdots,j_{2m})$ and hence $\langle \mu_1(\irr(\jb)) \rangle \leq \langle j_3 \rangle \leq N \leq N^2$. 
\item $a_1 = -a_2$. In this case, the first equation in \eqref{seminaire} yields 
if $a_1 - a_2 \neq 0$, then we have 
$$
|2a_1| \leq  (2m - 2) \nu_m N 
$$
showing that $\langle j_1 \rangle \leq N^2$ for $\nu_m$ small enough. As necessarily, $j_1 = \mu_1(\irr(\jb))$ we conclude that $\Nc_{\Lambda_{2m}}(b_{2m}) \leq N^2$. 
\item In any other situation, we have 
$$
|a_1 + a_2 ||a_1 - a_2| \leq (2m - 2) \nu_m^2 N^2
$$
with $|a_1 + a_2|$ and $|a_1 - a_2| \geq 1$. This shows that $|a_1 \pm a_2 |\leq (2m - 2) \nu_m^2 N^2$ and hence $\langle j_1 \rangle$ and $\langle j_2 \rangle$ smaller than $N^2$, for a good choice of $\nu_m$. We conclude as in the previous case that $\Nc_{\Lambda_{2m}}(b_{2m}) \leq N^2$. 
\end{itemize}

\medskip
\noindent {\bf Second step: Construction of $\tau_4$}. 
As we have seen, $K_6^N$ can be written under the form $Q_{[\Lambda_{6}]}[b_{6}]$ for some $\Lambda_{6} \in \Hscr_{3,\omega}$ and weight $\Nc_{[\Lambda_{6}]}(b_{6}) \leq N^2$. Furthermore by Theorem \ref{BNF}, $K_6$ contains only irreducible monomials so $\Lambda_{6} \in \Rscr_{3,\omega}$. By using Lemma \ref{rapido}, there exists $\Lambda'_6 \in \Hscr_{2,\omega}^*$ such that $\chi := Q_{\Lambda'_6}[b_6]$ is solution of the homological equation 
$$
\left\{Z_4, \chi  \right\} = \left\{Z_4, Q_{\Lambda'_6}[b_3]  \right\}  = Q_{\Lambda_6}[b_3] = K_6^N. 
$$
Moreover,  $\Nc_{\Lambda'_6}(b_6)=\Nc_{\Lambda_6}(b_6) \leq N^2$. Hence by using \eqref{prokofiev}, we immediately obtain the estimate 
\begin{equation}
\label{swimrun}
\Norm{X_{\chi}(z)}{s} \leq C  \varepsilon^{3}  \left(\frac{N^{\alpha_r}}{\gamma}\right)^{2}  
\end{equation}
for $z \in B_s(0,2\varepsilon)  \cap \mathcal{U}_{\gamma/2,\varepsilon,r,s}$. 

We then define $\tau_4 = \Phi_{\chi}^1$ the flow at time $1$ associated with the Hamiltonian $\chi$. Now let $z \in  B_s(0,\varepsilon)  \cap \mathcal{U}_{\gamma,\varepsilon,r,s}^N$.  We have to prove that the flow at time $1$ of the Hamiltonian $\chi_4$ remain in the set $B_s(0,3\varepsilon)  \cap \mathcal{U}_{\gamma/4,\varepsilon,r,s}^N$. To prove this result, we use a bootstrap argument. Let us assume that this is the case. \\
By using \eqref{swimrun}, we easily obtain that 
$$
\Norm{\Phi_\chi^1(z) - z}{s} \leq C  \varepsilon^{3}  \left(\frac{N^{\alpha_r}}{\gamma}\right)^{2}\leq \eps^{3/2 }\left(\frac{N^{\alpha_r}}{\gamma}\right)^{-1},
$$
for some constant $C$ that we choose to be the one of assumption \eqref{CFL} .
So we have in particular
$C \varepsilon^{3}  \left(\frac{N^{\alpha_r}}{\gamma}\right)^{2} \leq \varepsilon^\frac{3}{2}$, and hence $\Norm{\tau_4(z)}{s} \leq 3 \varepsilon$ provided $\varepsilon < 1$. Moreover, using Proposition \ref{georgesand}, with $z' = \tau_4(z)$ 
and $\gamma' = \gamma/2$, we have
$$
\Norm{z - \tau_4(z)}{s} \leq \frac12 c \varepsilon N^{- \alpha_r} \gamma
$$
where  $c$ is the constant of Proposition \ref{georgesand}. As a consequence
we have $\tau_4(z) \in \mathcal{U}_{\varepsilon,\gamma/2,r,s}^N$ which concludes the bootstarp argument. Estimate \eqref{esto4} then easily follows. 
 Note that $\tau_4$ is injective by definition of the flow. 

Now we apply $\tau_4$ to \eqref{bach}, taking into account, $\sum_{m=3}K_{2m}^r=\sum_{m=3}^rK^N_{2m}+R^{\rm high}$, we get 
 \begin{equation}
\label{Hcirc}
H\circ \tau_2 \circ \tau_4(z) = \big( Z_2  + Z_4(I) + Z_6 + K_6^N + \sum_{m = 4}^r K_{2m}^N  + R^{\rm high} + R\big) \circ \tau_4. 
\end{equation}
First we notice that $Z_2 \circ \tau_4 = Z_2$. Indeed, $\chi_4 = \sum_{\jb \in \Rc_3} f_j(I) z_\jb$, and $\jb$ is a resonant monomial in $\mathcal{R}_3$ thus we have $\{ Z_2(I), z_\jb\} = \Delta_{\jb} z_\jb = 0$ as well as $\{Z_2(I), f_\jb(I)\} = 0$. Hence $\{ Z_2, \chi_4 \} = 0$. \\
On the other hand we have, using the notation $\mathrm{ad}_\chi{G} = \{G,\chi\}$,
$$Z_4(I) \circ \tau_4 = Z_4 + \{ Z_4,\chi\} + \sum_{\alpha = 2}^{M} \frac{1}{\alpha!} \mathrm{ad}_{\chi}^\alpha Z_4 + \int_0^1  \frac{(t - s)^{M+1}}{(M)!}\mathrm{ad}_{\chi}^{M} Z_4 \circ \Phi_{\chi}^s ( z) \dd s,  
$$
 and a similar formula for all the terms of \eqref{Hcirc}, in particular
 \begin{equation}
\label{dev}
K_{2m}^N \circ \tau_4 = K_{2m}^N  + \sum_{\alpha = 1}^{M} \frac{1}{\alpha!} \mathrm{ad}_{\chi}^\alpha K_{2m}^N  + \int_0^1  \frac{(t - s)^{M+1}}{(M)!}\mathrm{ad}_{\chi}^{M} K_{2m}^N  \circ \Phi_{\chi}^s ( z) \dd s. 
\end{equation}
Note that by definition of $\chi$, the term $\{ Z_4,\chi\} $ and $K_{6}^N$ cancel. The first three terms in the asymptotics are thus $Z_2 + Z_4 + Z_6$. 
Now let us look at the other terms generated. 
As  $\chi := Q_{\Lambda'_6}[b_3]$ with $\Lambda'_6 \in \Hscr_{2,\omega}^*$, and as $Z_4(I) \in \Hscr_{2,\omega}$, Lemma \ref{cap} shows that $ \mathrm{ad}_{\chi}^\alpha Z_4  \in \Fscr_{2 + \alpha, \omega}$. 
Similarly, we have 
$\mathrm{ad}_{\chi}^\alpha K_{2m}^N  \in \Fscr_{m + \alpha,\omega}$ 
By collecting the elements of same degree, we obtain the claimed decomposition \eqref{norm_form_ind} where $R^{\rm ord}$ is a sum of terms of order greater than $2r + 2$ and where by a slight abuse of notation we still denote by $R^{\rm high}$ its composition by $\tau_4$ (which is closed to the identity). \\
The estimates on the remainder are then consequences of the previous estimates on $\tau_4$, upon using  the condition \eqref{CFL}. 
\end{proof}

\begin{remark}\label{TheDoors}
We have for $z \in B_s(0,2\varepsilon)  \cap \mathcal{U}_{\gamma/2,\varepsilon,r,s}^N$
\begin{multline*}
\Norm{z - \tau_2\circ \tau_4(z)}{s} \leq \Norm{z - \tau_2(z)}{s} + \Norm{\tau_2(z) - \tau_2 \circ \tau_4(z)}{s}
\\
\leq C \varepsilon^3 + C \Norm{z - \tau_4(z)}{s}\leq C  \varepsilon^{3}  \left(\frac{N^{\alpha_r}}{\gamma}\right)^{2}
\end{multline*}
as $\tau_2$ is $\mathcal{C}^1$ in a neighborhood of the origin and up to some change of constant $C$. Hence by the same argument as in the proof, we have that 
the application $\tau_2 \circ \tau_4$ maps $B_s(0,2\varepsilon)  \cap \mathcal{U}_{\gamma/2,\varepsilon,r,s}^N$ into $B_s(0,4\varepsilon)  \cap \mathcal{U}_{\gamma/4,\varepsilon,r,s}^N$. 
\end{remark}

\subsection{Quintic normal form}\label{killbill2}

{\em You and I have unfinished business.} Beatrix Kiddo in "Kill Bill: Volume 2" (Q. Tarentino, 2004).

Recall that $\Fscr_{r,\Omega}^A$ is the set of rational functions that depend only on the actions and can be written $ Q_\Gamma$ with $\Gamma\in\Ascr_{r,\Omega}$. By solving iteratively homological equations with the normal form term $Z_4 + Z_6$, we obtain the following proposition: 
\begin{proposition} \label{prop64}
Let $r\geq4$ be given. For all $s\geq0$, 
there exist   $A_{2m} \in \Ascr_{m,\Omega}$, for $4 \leq m \leq r$, and a constant $C>0$, such that for all $0<\varepsilon<1$, $\gamma>0$, $N \geq 1$ satisfying 
\begin{equation}
\label{CFL2}
C \varepsilon^\frac{3}{2} \left(\frac{N^{\alpha_r}}{\gamma}\right)^{3} < 1,
\end{equation}
 there exist
\begin{itemize}
\item $e_8,\dots,e_{2r} \in \ell^\infty_{A_8} \times \dots \times \ell^\infty_{A_{2r}}$
\item  $\tau_6 : B_s(0,\frac{3}{2}\varepsilon)  \cap \mathcal{U}_{\gamma,\varepsilon,r,s}^N \to \ell^1_s$ a $\mathcal{C}^1$ injective symplectomorphism, 
\item $R^{\rm high},R^{\rm ord} \in \mathcal{C}^1(B_s(0,\frac{3}{2}\varepsilon)  \cap \mathcal{U}_{\gamma,\varepsilon,r,s}^N)$ 
\end{itemize}
such that
\begin{equation}
\label{norm_form_ind2}
H\circ \tau_2 \circ \tau_4 \circ \tau_6 = Z_2+Z_4+Z_6 + \sum_{m=4}^r Z_{2m} + R^{\rm high} + R^{\rm ord},
\end{equation}
where $Z_{2m} = Q_{A_{2m}}[e_{2m}] \in \Fscr_{m,\Omega}^A$ depends only on the actions. Furthermore 
 we have the following bounds
\begin{itemize}
\item for all $m=4\dots r$, $\Nc_{\Gamma_{2m}}(e_{2m}) \leq N^2$
\item for all $m=4\dots r$, $\|c_{2m}\|_{\ell^{\infty}}\leq C_r$
\item For all $z \in B_s(0,\frac{3}{2}\varepsilon)  \cap \mathcal{U}_{\gamma,\varepsilon,r,s}^N$, we have 
\begin{equation}
\label{XRR}
\|X_{R^{\rm high}}(z)\|_{\ell_s^1} \leq  C \varepsilon^{5} N^{-s}  \quad \mbox{and}\quad \|X_{R^{\rm ord}}(z)\|_{\ell_s^1} \leq C \varepsilon^{2r + 1} \left(  \frac{N^{\alpha_{r+1}}}\gamma \right)^{4r-9} \end{equation}
\item $\tau_6$ takes values in $z \in B_s(0,2\varepsilon)  \cap \mathcal{U}_{\gamma/2,\varepsilon,r,s}^N$ and satisfies the estimates
\begin{equation}
\label{esto6}
\Norm{\tau_6(z) - z}{s} \leq C \varepsilon^{\frac{3}{2}} \gamma N^{-\alpha_r}. 
\end{equation}
\end{itemize}
\end{proposition}

\begin{proof}
We construct $\tau_6$ by induction. 
Note that in the Hamiltonian \eqref{norm_form_ind}, the terms are in $\Fscr_{m,\omega} \subset \Fscr_{m,\Omega}$. Starting with this Hamiltonian, we define $A_{8}$ and $R_8$ according to the decomposition (see \eqref{decomposition}) 
$$
Q_{\Gamma_8}[c_8] = Q_{A_8}[c_8] + Q_{R_8}[c_8], 
$$
where $R_8 \in \Rscr_{4,\Omega}$. Then Lemma \ref{lemmehomo} gives us $\Lambda_8 \in \Rscr_{2,\Omega}^* $ such that
$$
\{ Z_4 + Z_6, Q_{\Lambda_8}[c_8]\} = R_8[c_8], 
$$
We define $e_8 = c_8$, and we easily verity that $A_8$ and $e_8$ satisfy the hypothesis of the proposition. 
Setting $\chi_8 = Q_{\Lambda_8}[c_8]$, and using \eqref{prokofiev}, 
the application $\Phi_{\chi_8}^1$ satisfies an estimate under the form 
$$
\Norm{\Phi_{\chi_8}^1 (z) - z}{\ell_s^1} \leq C \varepsilon^{2p - 1} \Big( \frac{N^{\alpha_r}}{\gamma}\Big)^{4p - 5}\quad \mbox{with}\quad p = 8. 
$$
Thus using \eqref{CFL2} we conclude
$$\Norm{\Phi_{\chi_8}^1 (z) - z}{\ell_s^1} \leq C \varepsilon^\frac{3}{2} \gamma N^{-\alpha_r}.$$ 
As in the previous Proposition, we verify by using Proposition \ref{georgesand} that if $z \in B_s(0,\frac{3}{2}\varepsilon)  \cap \mathcal{U}_{\gamma,\varepsilon,r,s}^N$, then  $\Phi_{\chi_8}^s(z) \in B_s(0,2\varepsilon)  \cap \mathcal{U}_{\gamma' ,\varepsilon,r,s}^N$ for all $s \in (0,1)$ where we take $\gamma' = \frac{\gamma}{2}( 2 - \frac{1}{r})$. 

By using formulas similar to \eqref{dev} and shrinking $\gamma'$ up to $\gamma/2$, we see that 
$$
H\circ \tau_2 \circ \tau_4 \circ \Phi_{\chi_8}^1  = Z_2+Z_4+Z_6 + Z_8 + \sum_{m=5}^r  Q_{\Gamma_{2m}'}[c_m']+ R^{\rm high} + R^{\rm ord},
$$
where $\Gamma_{2m}'$, $c_{2m}'$, $ R^{\rm high}$ and $ R^{\rm ord}$ satisfy the condition of the Theorem. \\
By induction, for a given $p$, let us assume that the Hamiltonian is put on normal form up to order $2p$,
$$
\sum_{m = 2}^{p-1} Z_{2m} + \sum_{m = p}^{r-1} K_{2m} + R^{\rm high} + R^{\rm ord}, 
$$
with remainder terms $ R^{\rm high}$, $ R^{\rm ord}$ satisfying \eqref{XRR}, $Z_{2m} \in \Fscr_{m,\Omega}^A$ and $K_{2m}\in \Fscr_{m,\Omega}^R$. Let us decompose $K_{2p} = Z_{2p} + R_{2p}$ where 
$Z_{2m} \in \Fscr_{p,\Omega}^A$ and $R_{2p} \in \Fscr_{p,\Omega}^R$. Then to eliminate $R_{2p}$ we construct $\chi_{2p} \in \Hscr_{p-2,\Omega}^*$ by solving the homological equation of Lemma \eqref{lemmehomo}. 
We have $\chi_{2p} = Q_{\Lambda_{2p}}[c_{2p}]$ with $\Lambda_{2p} \in \Rscr_{p,\Omega }$ and by Lemma \ref{XQG} and under the assumption \eqref{CFL}
$$
\Norm{\Phi_{\chi_{2p}}^1 (z) - z}{\ell_s^1} \leq C \varepsilon^{2p - 1} \Big( \frac{N^{\alpha_r}}{\gamma}\Big)^{4 p - 5} \leq C \varepsilon^{\frac{3}{2}} \gamma N^{-\alpha_r}. 
$$
We then easily verify that the transformation $\tau_{6} = \Phi_{\chi_8}^1  \circ \cdots \circ \Phi_{\chi_{2r}}^1 $ satisfies the conditions of the Theorem. 
\end{proof}
\subsection{Proof of the rational normal form Theorem}
\begin{proof}[Proof of Theorem \ref{mainth}]
To prove Theorem \ref{mainth} it suffices to apply Proposition \ref{prop64} at order $r' = 6r$ and to choose
\begin{equation}\label{para}\alpha_{r'}=24 r',\quad N_\eps=\eps^{-\frac{2r-2}s},\quad \gamma_\eps=\eps^{\frac{1}{3} + \frac{1}{12}}, \quad s\geq s_0(r) = 48 \times 24  \times 6 (2r-2). \end{equation}
With this choice of $N = N_\varepsilon$, we have 
$$
\varepsilon^5 N_\varepsilon^{-s} = \varepsilon^{2r + 3}  \quad \mbox{and}\quad \varepsilon^{2r'+ 1} \left(  \frac{N^{\alpha_{r'+1}}}\gamma \right)^{4r'-9} \leq \varepsilon^{2r+3} 
$$ so that the estimate \eqref{XR} is satisfied for $R = R^{\rm high} + R^{\rm ord}$ in \eqref{XRR} for $\varepsilon$ small enough (we reach the order $2r+3$ instead of $2r+1$ to normalize the constant to $1$). 

Now condition  \eqref{CFL2} can be written 
$$
\varepsilon^{\frac{1}{2} - \frac{1}{3} - \frac{1}{12}  -  24 r' \frac{2r-2}s}  \leq C^{- \frac{1}{3}}.
$$
Choosing  $\eps<\eps_0(r,s) = C^{- 16/3}$ and using the definition of $s_0$, this condition is satisfied:
$$
\varepsilon^{\frac{1}{2} - \frac{1}{3} - \frac{1}{12}  -  24 r' \frac{2r-2}s} \leq \varepsilon^{\frac3{48}} \leq C^{- \frac{1}{3}}.
$$
With these choices, Theorem \ref{mainth} holds true with
\begin{equation}\label{Vc}
\mathcal{C}_{\eps,r,s}=\mathcal{U}_{\gamma_\eps,\varepsilon,r',s}^{N^\eps}\cap B_s(0,\textstyle\frac{3}{2}\varepsilon),\qquad \tau=\tau_2 \circ \tau_4 \circ \tau_6, \quad\mbox{and}\quad \mathcal{O}_{\eps,r,s} = \tau(\mathcal{C}_{\varepsilon,r,s}), 
\end{equation}
as $\tau$ is injective on $\mathcal{C}_{\eps,r,s}$. Moreover, by Remark \eqref{TheDoors} and the previous estimates, we have 
$\mathcal{O}_{\varepsilon,r,s} \subset \mathcal{U}_{\gamma_\eps/4,\varepsilon,r',s}^{N^\eps}\cap B_s(0,4\eps)$, and  
\begin{equation}
\label{gdansk}
\Norm{\tau(z) - z}{s} \leq C \varepsilon^{\frac{3}{2}} \gamma N^{-\alpha_r'} \leq \varepsilon^{\frac32}. 
\end{equation}
\end{proof}

\section{Dynamical consequences and probability estimates}\label{dyn}
We are now in position to prove Corollary \ref{corodyn} and Theorem \eqref{clouzot}. First, we define the sets
$$
\Vc_{\varepsilon,r,s} = \mathcal{U}_{4\gamma_\eps,\varepsilon,r',s}^{N^\eps}\cap B_s(0,\textstyle\frac{1}{2}\varepsilon) \subset \mathcal{O}_{\eps,r,s}
\quad \mbox{and} \quad 
\Wc_{\varepsilon,r,s} = \tau^{-1}(\Vc_{\varepsilon,r,s}) \subset \mathcal{C}_{\varepsilon,r,s} ,
$$
 where as previously $r'=6r$.
With this definition, Estimate \eqref{probaNLS} is a consequence of Proposition \ref{counting_eps_fixed} and Proposition \eqref{musset}. Note that the condition required in this proposition is ensured (with $\gamma' = \gamma/2 = 4\gamma_\varepsilon$) under the condition \eqref{CFL2}. Note moreover that we use the term $\varepsilon^{\frac{1}{12}}$ in the definition of $\gamma$ to fix the constant to $1$ in the final probability estimate and obtain $\varepsilon^{\frac13}$. 

Similarly, \eqref{lastone} is obtained from Corollary \eqref{counting_sequence} with $\nu = \varepsilon^{\frac16}$. This proves Theorem \eqref{clouzot}. 

To prove the dynamical consequences, we note that the open set $\Wc_{\eps,r,s}$ contains the initial value in the new variable. 
Let $z \in \Vc_\varepsilon$ and $z' = \tau^{-1}(z) \in \Wc_{\varepsilon,r,s}$. 
By using \eqref{gdansk} we have $\Norm{z' - z}{s} = \Norm{z' - \tau(z')}{s} \leq C \varepsilon^{\frac{3}{2}} \gamma N^{-\alpha_r}$. 
Hence by using Proposition \ref{georgesand}, we deduce that 
$$
\Wc_{\varepsilon,r,s} \subset  \mathcal{U}_{2 \gamma_\eps,\varepsilon,r',s}^{N^\eps}\cap B_s(0,\varepsilon). 
$$
The goal of the analysis is thus to prove that the dynamics starting in $\mathcal{U}_{2 \gamma_\eps,\varepsilon,r',s}^{N^\eps}\cap B_s(0,\varepsilon)$ remains in the set $\mathcal{C}_{\eps,r,s}=\mathcal{U}_{\gamma_\eps,\varepsilon,r',s}^{N^\eps}\cap B_s(0,\textstyle\frac{3}{2}\varepsilon)$ for a time $T \leq \varepsilon^{-2r + 1}$. To prove this, 
we first recall a {\it small} lemma proved in \cite{FG10}:
\begin{lemma}
\label{petit}
let $f:\R \to \R_+$ a continuous function, and $y:\R \to \R_+$ a differentiable function satisfying the inequality
$$
\forall\, t \in \R, \quad  \frac{\dd}{\dd t} y(t) \leq 2 f(t) \sqrt{y(t)}.
$$
Then we have the estimate
$$
\forall\, t \in \R, \quad \sqrt{y(t)} \leq \sqrt{y(0)} + \int_0^t f(s) \, \dd s.
$$
\end{lemma}

\begin{proof}[Proof of Corollary \ref{corodyn}]
We use a bootstrap argument. Let us fix $r\geq 2$, $s\geq s_0(r)$ and $\eps<\eps_0(r,s)$ as in Theorem \ref{mainth}. Let $U(0)=(u_a(0))_{a\in\Z} \in \Vc_{\varepsilon,r,s}$ and $V(0) = \tau^{-1}(U(0))  = (v_a(0))_{a\in\Z} \in \Wc_{\varepsilon,r,s}$. 
By definition, we have 
$$
V(0) \in \Wc_{\varepsilon,r,s} \subset  \mathcal{U}_{2 \gamma_\eps,\varepsilon,r',s}^{N^\eps}\cap B_s(0,\varepsilon) 
$$
and let 
$$
T=\sup\{t>0\mid  V(t')  \in \mathcal{C}_{\varepsilon,r,s} \text{ for all }0\leq t'< t\}.
$$
Note that for $t < T$, we have $U(t) = \tau(V(t)) \in \mathcal{O}_{\varepsilon,r,s}$ which coincides with the solution governed by 
the Hamiltonian $H_{\rm NLS}$ by uniqueness of the solution. We are going to prove that if $t\leq\min(T,\eps^{-2r +1 })$ then  $V(t)  \in \mathcal{U}_{\gamma_\eps,\varepsilon,r',s}^{N^\eps}\cap B_s(0,\textstyle\frac{3}{2}\varepsilon) = \mathcal{C}_{\varepsilon,r,s}$ 
and then conclude to $T\geq\eps^{-2r}$ by a continuity argument. To prove this we have, in view of \eqref{Vc}, to control the small divisors \eqref{nonres1N} and \eqref{nonres2N} and the norm $\Norm{V(t)}{s}$.\\

Let $J_a(t) = |v_a(t)|^2$ denote the actions of $V(t)$. 
For $t<T$ we can use Theorem \ref{mainth} to conclude that 
\begin{align}
\label{I'}\big|\frac \dd{\dd t}J_a(t)\big|&\leq ( J_a(t))^{1/2} \left(\left|\frac{\partial R}{\partial \xi_a}(V(t))  \right| + \left|\frac{\partial R}{\partial \xi_a}(V(t))  \right| \right)\\
\label{I'2}&\leq
(J_a(t))^{1/2} \Norm{X_R(V(t))}{s} \langle a  \rangle^{-s}\leq C\eps^{2r+2}\langle a  \rangle^{-2s}.
\end{align}
Therefore  for $t<T$, we have 
\begin{equation}\label{Itilde}
|J_a(t)-J_a(0)  |\leq C T\eps^{2r+2}\langle a  \rangle^{-2s} \quad \forall a\in\Z. 
\end{equation}
Together with Proposition \ref{georgette}, 
this equation shows that for $T \leq \varepsilon^{-2r + 1}$ and under the condition \eqref{CFL2} fulfilled by $N_\varepsilon$ and $\gamma_\varepsilon$, we have $V(t) \in \mathcal{U}_{\gamma_\eps,\varepsilon,r',s}^{N^\eps}$. 

To control the norm of $V(t)$, we note that 
since  $\Norm{V(t)}{s}=\sum_{a \in \Z} \langle a \rangle^s |v_a(t)| =\sum_{a \in \Z} \langle a \rangle^s |J_a(t)|^{1/2} $ we get using \eqref{I'} and Lemma \ref{petit}  
\begin{equation*}
\Norm{V(t)}{s}\leq\Norm{V(0)}{s} +\int_0^t \|X_R(V(s))\|_sds \leq \Norm{V(0)}{s} +  t\eps^{2r +1}. 
\end{equation*}
Using \eqref{estimtau} we get for $t\leq\min(T,\eps^{-2r+1})$ and $\eps$ small enough
\begin{equation}\label{NU}
\Norm{ V(t)}{s}\leq\Norm{ V(0)}{s} + t\eps^{2r + 1}\leq \eps+ \eps^2\leq \textstyle\frac32\eps, 
\end{equation}
hence $V(t) \in \mathcal{C}_{\varepsilon,r,s}$ for $T \leq {\varepsilon^{-2r +1}}$. This shows in particular that $U(t) \in \mathcal{O}_{\varepsilon,r,s}$ on this time horizon and conclude our bootstrap argument.

Finally it remains to prove \eqref{estimua}. Let $I_a(t) = |u_a(t)|$, by
\eqref{estimtau} we get that $|J_a(t) - I_a(t)  |\leq  
C \varepsilon^{\frac{5}{2} + \frac{1}{24}}\langle a  \rangle^{-2s}$. We then deduce that for $t\leq\min(T,\eps^{-2r+2})$
\begin{align}\nonumber|I_a(t)-I_a(0) |&\leq| I_a(t)-J_a(t)  |+|\tilde I_a(t)-J_a(0)  |+|J_a(0)- I_a(0)  |\\
\label{II2}&\leq 2\eps^\frac{5}{2}\langle a  \rangle^{-2s}+T\eps^{2r+2}\langle a  \rangle^{-2s}\leq 3 \eps^{\frac{5}{2}}\langle a  \rangle^{-2s}, \end{align}
which shows \eqref{estimua} and conclude the proof of the Corollary. 
%
\end{proof}


\appendix
\section{The case of (NLSP)}\label{casNLSP}
As explain in section \ref{sketch}, the main difference between \eqref{nls} and \eqref{nlsp} appears when we calculate $Z_4$. Indeed, the resonant normal form procedure used in section \ref{killbill0} leads, in the \eqref{nlsp} case, to  the following formula (see \eqref{eqmagic} with $\hat V_a = a^2$, $a \neq 0$ and $\hat V_0 = 0$) 
$$Z_4 =  \varphi'(0)  \sum_{a\neq b \in \Z} \frac{1}{(a - b)^2} I_a I_b. 
$$
Thus the frequencies associated with this integrable Hamiltonian are  
$$
\lambda_a (I) = \frac{\partial Z_4}{\partial I_a} = 2\varphi'(0)  \sum_{b\neq a \in \Z} \frac{1}{(a - b)^2}  I_b. 
$$
For these frequencies we obtain a much better control of the small denominators that the one obtained for \eqref{nls}, in particular, contrary to the \eqref{nls} case (see \eqref{nlsbad}), the loss of derivative is independent of $s$.\\
For $\jb = (j_1,\ldots,j_{2m}) \in \U_2\times \Z$, if $j_\alpha = (\delta_\alpha,a_\alpha)$ for $\alpha = 1,\ldots,2m$, the small denominators in the \eqref{nlsp} case are given by 
$$
\omega_{\jb}(I)= \sum_{\alpha = 1}^{2m} \delta_\alpha \frac{\partial Z_4}{\partial I_{a_\alpha}} (I) = 2 \varphi'(0) \sum_{\alpha = 1}^{2m} \delta_\alpha  \sum_{b \neq a_\alpha  \in \Z} \frac{1}{(a_\alpha - b)^2}  I_b. 
$$
Let us remark that $\omega_\jb(I)$ has the same structure of the small denominator associated with $Z_6$ used to obtain non resonance estimates, except that $I_b^2$ is replaced by $I_b$ as it can be easily seen by comparing the previous formula with \eqref{nocturne}. By proceeding as in Section \ref{gene}, with the crucial use of Lemma \ref{proba_need_algebra}, we obtain the following result whose proof whose proof is left to the reader. 
\begin{lemma}
\label{lemma_nonresNLSP} Assume that $I_a$, $a\in \Z$ are independent random variable with $I_a$ uniformly distributed in $(0,\langle a \rangle^{-2s - 4})$, then
there exists a constant $c>0$ such that for all $\gamma\in (0,1)$ we have
$$
\mathbb{P}\left( \forall \kb \in \Ic, \ \length k \leq 2r \Rightarrow |\omega_{\kb}( I)| \geq \gamma \left( \prod_{\alpha=1}^{\length\kb } \langle k_\alpha \rangle^{-4} \right) \right) \geq 1 - c\gamma.
$$
\end{lemma}
The major difference with the \eqref{nls} case is that now the small denominator do not depend on $s$ (compare with Lemma \ref{lemma_nonres1}). Hence, the construction can be performed without having to distribute the derivative and we can apply a normal form procedure using only $Z_4$ (and not $Z_4+Z_6$ as in the \eqref{nls} case). 

Following the general strategy, for $\eps,\gamma >0$, $r\geq1$, $N \geq 1$ and $s\geq0$,
 we say that  $z\in \ell_s^1$ belongs to the non resonant set $ \mathcal{U}_{\gamma,\varepsilon,r,s}$, if for all $\kb\in \Ic$ of length $\length\kb \leq 2r$ we have
$$
|\omega_\kb(I)|> \gamma  \varepsilon^2 \left( \prod_{\alpha=1}^{\length\kb } \langle k_\alpha \rangle^{-4} \right);
$$
and the  that  $z\in \ell_s^1$ belongs to the truncated non resonant set $ \mathcal{U}_{\gamma,\varepsilon,r,s}^N$, if for all $\kb\in \Ic$ of length $\length\kb \leq 2r$ such that
$\langle \mu_1(\kb)\rangle \leq N^2$,  we have 
\begin{equation}
\label{goek}
|\omega_\kb(I)|>   \gamma \varepsilon^2  N^{-16r}.
\end{equation}

An adapted Proposition \ref{musset} remains valid, namely: for $N$ large enough depending on $\varepsilon$ and on $\gamma'<\gamma$ we have $\mathcal{U}_{\gamma,\varepsilon,r,s}\subset \mathcal{U}_{\varepsilon,\gamma',r,s}^N$. Moreover, by using the previous Lemma, if $z \in \ell_s^1$ depends on random actions $I_a$ independent and uniformly distributed in $(0,\langle a \rangle^{-2s - 4})$,  there exists a constant $c>0$ such that for all $\gamma\in(0,1)$ we have
\begin{equation}
\label{gasp}
 \ \mathbb{P}\left( \forall \varepsilon>0,\ \varepsilon z\in  \mathcal{U}_{\gamma,\varepsilon,r,s} \right) \geq 1 - c\gamma.
\end{equation}
Note that the difference with Proposition 
\ref{counting_eps_fixed} is that for one choice of non resonant actions, the non resonance condition holds for all $\varepsilon$. In other words, the phenomenon of resonances between $\varepsilon$ and $I$ cannot occur in the \eqref{nlsp} case. 

The class of rational Hamiltonians we need is also simpler: we only need to consider $\Hscr_\omega$ and $\Hscr_\omega^*$ defined in Section \eqref{rational}, {\em i.e.} functionals of the form 
$$
 Q_\Gamma[c](z)= \sum_{ \ell \in \mathbb{Z}^* } c_\ell (-1)^{n_\ell} \frac{z_{\pig_\ell}}{\displaystyle \prod_{\alpha =1}^{p_{\ell}} \omega_{\kb_{\ell,\alpha}}  } .
$$
with the same condition as in the \eqref{nls} case, 
but {\em without } the restrictive condition {\bf (vi)} on the distribution of derivatives, making the proof of the Poisson bracket estimate considerably much simpler, as can be seen in the next Appendix. 

By using the estimate \eqref{goek}, we can prove an equivalent of Lemma \eqref{XQG} for this class of functional (with $\alpha_r = 16r$) and the steps of the rational normal form construction can be then followed as in Section \ref{RNF} under the same condition \eqref{CFL2}. The optimization process in $N$ and $\gamma$ can then be done in the same way. \\
In the end, the probability estimate \eqref{gasp} gives Theorem \eqref{gilles}. 

\section{Proof of Lemma \ref{cap}}\label{appendixA}

This section is devoted to the proof of Lemma \ref{cap}. As in the statement of the Lemma, let $W = \omega$ or $W = \Omega$ and let
$\Gamma = (\pig,\kbsf,\hbsf,n)\in \Hscr_{r,W}^*$ and $\Gamma = (\pig',\kbsf',\hbsf',n')\in \Hscr_{r',W}$. 

To compute the poisson bracket between  $Q_\Gamma[c]$ and $Q_{\Gamma'}[c']$, we only need to calculate the poisson brackets of the summands (see the expression \eqref{Q}). Applying the Leibniz's rule we see that,
up to combinatorial factors and finite linear combinations depending on $r$, four kind of terms appear depending on which part of the Hamiltonians the Poisson bracket applies to: 
%
%
 
\medskip 
\noindent {\bf \emph{Type I.} }
The first type of terms we consider are those where the derivatives apply only on the  numerators. They are of the form 
$$
\frac{c_\ell c_{\ell'}(-i)^{p_\ell + q_\ell + p_{\ell'} + q_{\ell'}}}{\displaystyle \prod_{\alpha =1}^{n_{\ell}} \omega_{\kb_{\ell,\alpha}}  \prod_{\alpha =n_{\ell}+1}^{p_{\ell}} \Omega_{\kb_{\ell,\alpha}} \prod_{\alpha =1}^{q_{\ell}} \Omega_{\hb_{\ell,\alpha}}
 \prod_{\alpha =1}^{n_{\ell'}} \omega_{\kb'_{\ell',\alpha}}  \prod_{\alpha =n_{\ell'}+1}^{p_{\ell'}} \Omega_{\kb'_{\ell',\alpha}} \prod_{\alpha =1}^{q_{\ell'}} \Omega_{
 \hb_{\ell',\alpha}}
 }
\{ z_{\pig_\ell},z_{\pig'_{\ell'}}\}
$$
for some $\ell$ and $\ell'$ in $\Z^*$. Let us set $\jb = \pig_\ell$ and $\jb' = \pig'_{\ell'}$, i.e. $z_{\pig_\ell} = z_{j_1}\cdots z_{j_{2m}}$ and $z_{\pig'_{\ell'}} = z_{j'_1}\cdots z_{j'_{2m'}}$. The product $\{z_{\pig_\ell},z_{\pig_{\ell'}}\}$ is a linear combination of terms of the form $z_{\jb''}$ with $\jb'' \in \mathcal{R}_{m_\ell + m_{\ell'}-1}$ . 

Up to a combinatorial factor, linear combinations and renumbering to define the application $\pig''$, we can concentrate on terms $z_{\jb''}$ with $\jb'' = \pig''_{\ell''}$ of the form 
$$
z_{j_2}\cdots z_{j_{2m}}z_{j'_2}\cdots z_{j'_{2m'}}, 
$$
provided $\bar j_1 = j'_1$. Clearly, the produced term is of the good form with $r'' = r+ r' - 1$, $n_{\ell'} = n_\ell + n_\ell'$, $q_{\ell'} = q_\ell + q_{\ell'}$ and $p_{\ell''} = p_{\ell} + p_{\ell'}$. In particular the reality condition is easily verified by considering the terms corresponding to $-\ell$ and $- \ell'$ and imposing $\overline{\pi''_{\ell''}} = \pi''_{-\ell''}$, and the conditions ${\bf (i)-(v)}$ of the definition of the class are trivially satisfied. 
We can also verify that these terms fulfill the conditions defining the subclass $\Hscr_{r'',W}$. Indeed, in the case when $W = \omega$, we have $q_{\ell''} = q_\ell + q_{\ell'} = 0$ and $n_{\ell''} = n_{\ell} + n_{\ell'} \leq 2(r +1) - 5 + 2r' - 6 = 2 (r + r') - 9 \leq 2 r'' - 7$. 
In the case $ W  = \Omega$, we can set $\alpha_i'' = \alpha_i + \alpha_i'$ for $i = 1,\ldots,4$ and $\alpha_5'' = \alpha_5 + \alpha_5'+1$, and we can easily check that the relations \eqref{grizzli} are satisfied for $\Gamma''$. Moreover, using \eqref{ilpleut} and \eqref{beaucoup}, we check that $\alpha_5'' = \alpha_5 + \alpha_5'  +1 \leq (r + 2) - 4 + r' - 4 +1 \leq r''-4$, and similarly that the three conditions in \eqref{ilpleut} are satisfied.

It remains to prove the conditions ${\bf (vi)}$ and ${\bf (vii)}$ that are the most delicate. 
We analyze different cases according to which are the largest indices among $\jb$, $\jb'$ and $\jb''$. The three main case are $\langle j_1\rangle \leq \langle \mu_3(\jb)\rangle$, $ j_1 = \mu_2(\jb)$ and $j_1 = \mu_1(\jb)$, and by symmetry, we are left to the following cases to be studied:

\begin{center}
\begin{tabular}{|c|c|c|c|c|}
\hline
\multirow{6}{*}{$\langle j_1 \rangle \leq \langle \mu_3(\jb) \rangle$} & \multirow{2}{*}{$\langle j_1' \rangle \leq \langle \mu_3(\jb') \rangle$} & $\mu_1(\jb'')  = \mu_1(\jb)$ & $\mu_2(\jb'') = \mu_1(\jb')$ & 3.3.a \\
  &  & $\mu_1(\jb'')  = \mu_1(\jb)$ & $\mu_2(\jb'') = \mu_2(\jb)$ & 3.3.b \\
  \cline{2-5}
 & \multirow{3}{*}{$j_1' = \mu_2(\jb')$} & $\mu_1(\jb'')  = \mu_1(\jb')$ & $\mu_2(\jb'') = \mu_1(\jb)$ & 3.2.a\\
 & & $\mu_1(\jb'')  = \mu_1(\jb)$ & $\mu_2(\jb'') = \mu_2(\jb)$ & 3.2.b\\
 & & $\mu_1(\jb'')  = \mu_1(\jb)$ & $\mu_2(\jb'') = \mu_1(\jb')$ & 3.2.c \\
\cline{2-5}
&  $j_1' = \mu_1(\jb')$ & $\mu_1(\jb'') = \mu_1(\jb)$ & $\mu_2(\jb'') = \mu_2(\jb)$ & 3.1 \\
\hline 
\hline
\multirow{2}{*}{$j_1 = \mu_2(\jb)$} & $j_1' = \mu_2(\jb')$ & $\mu_1(\jb'') = \mu_1(\jb)$ & $\mu_2(\jb'') = \mu_1(\jb')$ & 2.2 \\
\cline{2-5}
& \multirow{2}{*}{$j_1' = \mu_1(\jb')$} & $\mu_1(\jb'') = \mu_1(\jb)$ & $\mu_2(\jb'') = \mu_3(\jb)$ & 2.1.a \\
&  & $\mu_1(\jb'') = \mu_1(\jb)$ & $\mu_2(\jb'') = \mu_2(\jb')$ & 2.1.b \\
\hline
\hline
\multirow{2}{*}{$j_1 = \mu_1(\jb)$} & \multirow{2}{*}{$j_1' = \mu_1(\jb')$} & $\mu_1(\jb'') = \mu_2(\jb)$ & $\mu_2(\jb'') = \mu_2(\jb')$ & 1.1.a \\
&  & $\mu_1(\jb'') = \mu_2(\jb)$ & $\mu_2(\jb'') = \mu_3(\jb)$ & 1.1.b \\
\hline
\end{tabular}
\end{center}

%
%
%
%
%
%
%
%
%
%
%
%
%
%
%
%

These cases are summarized in Figure \ref{mon_beau_dessin} where we try to visualize the different configurations.

\begin{figure}
\includegraphics[scale=0.75]{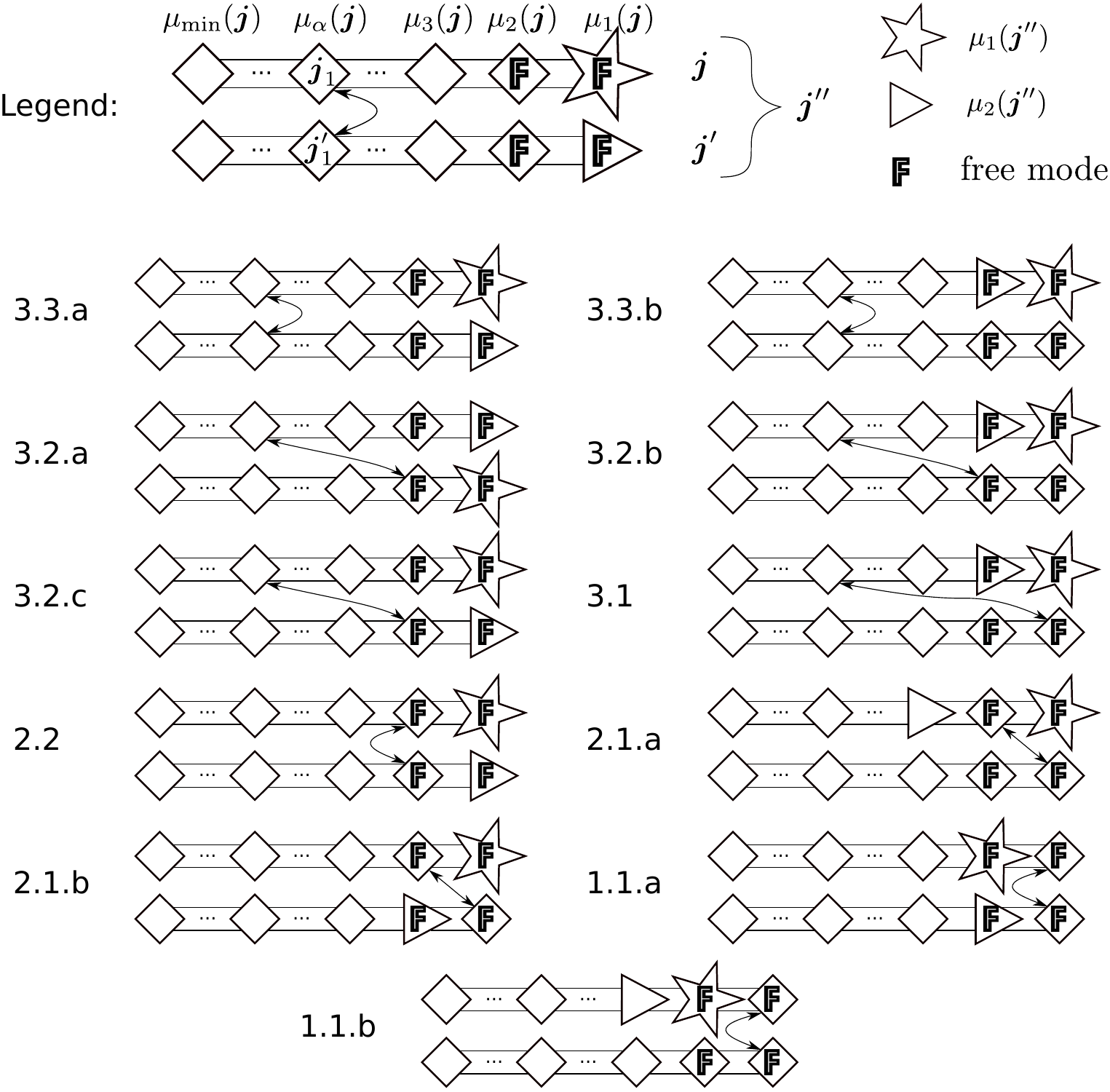}
\caption{Possible configurations arising from the calculation of $\{ z_{\jb}, z_{\jb'} \}$.}
\label{mon_beau_dessin}
\end{figure}

\medskip 
\noindent \ul{\bf Cases 3.3}. 
In these cases, we have $\langle j_1\rangle \leq \langle\mu_3(\jb)\rangle$ and $\langle j'_1\rangle \leq \langle\mu_3(\jb')\rangle$
and  ${\bf (vii)}$ is automatically satisfied as $\mu_2(\jb'')$ is always greater than $\mu_2(\jb)$ and $\mu_2(\jb')$. 

To prove ${\bf (vi)}$, we must contruct a fontion $\iota''$ that distributes the derivatives in $\jb''$ from the functions $\iota$ and $\iota'$ distributing the derivatives in $\jb$ and $\jb'$. 
Note that by induction hypothesis and the definition of the condition \eqref{crux}, the modes $\mu_1(\jb)$, $\mu_{2}(\jb)$, $\mu_1(\jb')$ and $\mu_2(\jb')$ are {\em free} in the sense that $1$ and $2$ are not in the image of $\iota$ and $\iota'$. 

We see that we can build $\iota''$ from $\iota$ and $\iota'$ easily if $j_1$ or $j_1'$ do not correspond to some $\mu_{\iota_{\alpha}}(\jb)$ or $\mu_{\iota'_{\alpha}}(\jb')$, as $j_1$ and $j_1'$ do not appear in $\jb''$. We thus see that the issue is to control $\langle j_1 \rangle = \langle j_1' \rangle$ by two free modes {\em and} by letting the two highest modes $\mu_1(\jb'')$ and $\mu_2(\jb'')$ free.  
 Indeed, in such a case, up to a reconfiguration of $\iota''$,  the relation \eqref{crux} will hold again for $\jb''$, by using the induction hypothesis on $\jb$ and $\jb'$. 
By symmetry, we thus are led to distinguish two cases: 

\medskip 
\noindent \ul{\em Case 3.3.a}.  $\mu_1(\jb'') = \mu_1(\jb)$ and $\mu_2(\jb'') = \mu_1(\jb')$. 
 In this case, $\langle j_1 \rangle \leq \langle \mu_{2}(\jb)\rangle$, $\langle j_1' \rangle \leq \langle\mu_{2}(\jb')\rangle$ and we can distribute the derivative to the free modes $\mu_2(\jb)$ and $\mu_2(\jb')$ by letting the two highest modes of $\jb''$ free. 
 
\medskip 
\noindent \ul{\em Case 3.3.b}. $\mu_1(\jb'') = \mu_1(\jb)$ and $\mu_2(\jb'') = \mu_2(\jb)$.
In this case, we use the fact that $\langle j_1 \rangle = \langle j_1' \rangle$ to control $\langle j_1 \rangle$ by $\langle \mu_2(\jb')\rangle$ and 
$\langle j_1' \rangle$ by $\langle \mu_1(\jb')\rangle$ which are modes always smaller that $\langle\mu_3(\jb'')\rangle$.


\medskip 
\noindent \ul{\bf Cases 3.2}. $\langle j_1\rangle \leq \langle\mu_3(\jb)\rangle$ and $j_1' = \mu_2(\jb')$. The main difference with the previous case is that condition {\bf (vii)} is not automatically satisfied. To prove it, we need a control of   $\langle \mu_2(\jb) \rangle$ and $\langle \mu_2(\jb') \rangle$ by $\langle \mu_2(\jb'')\rangle $. 
But on the other hand, we only need to control $\langle j_1 \rangle$ by one mode, as $j_1'$ was not used in the distribution of the derivative (condition \eqref{crux}) for $\jb'$. 
As necessarily the first to highest modes of $\jb''$ are in the set $\{ \mu_1(\jb), \mu_2(\jb),\mu_1(\jb')\}$ we are thus led to the following three cases: 

\medskip 
\noindent \ul{\em Case 3.2.a}. $\mu_1(\jb'')  = \mu_1(\jb')$ and $\mu_2(\jb'') = \mu_1(\jb)$. In this situation, we can easily control $\langle j_1 \rangle$ by $\langle \mu_2(\jb) \rangle$ which is free, and fulfill condition {\bf (vi)}. Moreover, we have $\langle\mu_2(\jb) \rangle\leq \langle\mu_1(\jb)\rangle =\langle \mu_2(\jb'')\rangle$ and $\langle \mu_2(\jb')\rangle = \langle j_1 \rangle \leq \langle\mu_2(\jb)\rangle$ and hence condition \eqref{HP_old_momo} for $\jb''$ is inherited from condition {\bf (vii)} for $\jb$ and $\jb'$. 

\medskip 
\noindent \ul{\em Case 3.2.b}. $\mu_1(\jb'')  = \mu_1(\jb)$ and $\mu_2(\jb'') = \mu_2(\jb)$. Here, we can control $\langle j_1 \rangle = \langle \mu_2(\jb') \rangle$ by $\langle\mu_1(\jb')\rangle$ which is free and smaller than $\mu_2(\jb'')$ which shows ${\bf (vi)}$. Moreover, in this situation, we have $\mu_{2}(\jb) = \mu_2(\jb'')$ and $\langle \mu_2(\jb') \rangle = \langle  j_1 \rangle \leq \langle \mu_2(\jb)\rangle =\langle \mu_2(\jb'') \rangle$ so that ${\bf (vii)}$ holds true for $\jb''$.  

\medskip 
\noindent \ul{\em Case 3.2.c}. $\mu_1(\jb'')  = \mu_1(\jb)$ and $\mu_2(\jb'') = \mu_1(\jb)$. In this situation ${\bf (vi)}$ can be easily shown as $\langle j_1 \rangle \leq \langle \mu_2(\jb)\rangle$ which free and smaller than $\langle \mu_2(\jb'')\rangle$. To prove ${\bf (vii)}$, we notice that $\mu_2(\jb) = \mu_2(\jb'')$  and $\langle \mu_2(\jb')\rangle = \langle j_1\rangle \leq \langle \mu_2(\jb) \rangle$. 

\medskip 
\noindent \ul{\bf Case 3.1}. $\langle j_1\rangle \leq \langle\mu_3(\jb)\rangle$ and $j_1' = \mu_1(\jb')$. In this situation we have $\mu_1(\jb'') = \mu_1(\jb)$ and $\mu_2(\jb'') = \mu_2(\jb)$.  As in the previous case, we only have to distribute derivative in one free mode, {\em i.e.} control $\langle j_1 \rangle$ by $\langle\mu_2(\jb')\rangle$. This is done by using the zero momentum condition: we have $\langle j_1 \rangle = \langle \mu_1(\jb') \rangle \leq C_{r'} \langle \mu_2(\jb') \rangle$ where $C_{r'}$ depends only on $r'$. This shows ${\bf (vi)}$ and ${\bf (vii)}$ is proved by noticing that $\mu_2(\jb) = \mu_2(\jb'')$ and $\langle \mu_2(\jb') \rangle \leq \langle \mu_1(\jb') \rangle  = \langle j_1 \rangle \leq \langle \mu_2(\jb) \rangle$. 

\medskip 
\noindent \ul{\bf Case 2.2}. $j_1 = \mu_2(\jb)$ and $j_1' = \mu_2(\jb')$ and by symmetry we can assume $\mu_1(\jb'') = \mu_1(\jb)$ and $\mu_2(\jb'')=\mu_1(\jb')$. 
In this case, ${\bf (vi)}$ for $\jb''$  is directly inherited from the condition for $\jb$ and $\jb'$ as $j_1$ and $j_1'$ were not involved in them. 
To prove ${\bf (vii)}$, we notice that 
$ \langle\mu_2(\jb)\rangle = \langle\mu_2(\jb')\rangle \leq \langle\mu_1(\jb')\rangle = \langle\mu_2(\jb'')\rangle$.  

\medskip 
\noindent \ul{\bf Cases 2.1}. $j_1 = \mu_2(\jb)$ and $j_1' = \mu_1(\jb')$. In this necessarily, we have $\mu_1(\jb'') = \mu_1(\jb)$. As in the previous case, ${\bf (vi)}$ is easily obtained. To prove ${\bf (vii)}$ we have to distinguish two cases: 

\medskip 
\noindent \ul{\em Case 2.1.a}. $\mu_2(\jb'') = \mu_3(\jb)$, which means in particular that $\langle\mu_2(\jb')\rangle \leq \langle\mu_3(\jb)\rangle = \langle \mu_2(\jb'')\rangle$. Moreover, by using the zero-momentum condition, we have $\langle \mu_2(\jb) \rangle  = \langle \mu_1(\jb') \leq C_{r'} \langle\mu_2(\jb')\rangle \leq C_{r'} \langle \mu_2(\jb'')\rangle$ and this shows the result. 

\medskip 
\noindent \ul{\em Case 2.1.a}. $\mu_2(\jb'') = \mu_2(\jb')$. In this situation we just need to prove that $\langle\mu_2(\jb)\rangle $ is controlled by $\langle\mu_2(\jb'')\rangle$ which is ensured by the fact that $\langle\mu_2(\jb)\rangle = \langle\mu_1(\jb')\rangle \leq C_{r'} \langle\mu_2(\jb')\rangle$ by using the zero momentum condition. 

\medskip 
\noindent \ul{\bf Cases 1.1}. $j_1 = \mu_1(\jb)$ and $j_1' = \mu_1(\jb')$. As before, ${\bf (vi)}$ is easily obtained. 
To verify ${\bf (vii)}$, by symmetry, we have only two cases to examine:
 
\medskip 
\noindent \ul{\em Case 1.1.a}. $\mu_1(\jb'') = \mu_2(\jb)$ and $\mu_2(\jb'') = \mu_2(\jb')$. In this situation, we have by using the zero momentum condition 
$\langle\mu_2(\jb)\rangle =\langle \mu_1(\jb'')\rangle \leq C_{r''}\langle \mu_2(\jb'')\rangle$ which shows ${\bf (vii)}$. 

\medskip 
\noindent \ul{\em Case 1.1.b}. $\mu_1(\jb'') = \mu_2(\jb)$ and $\mu_2(\jb'') = \mu_3(\jb)$. In this case we have necessarily $\langle\mu_2(\jb')\rangle \leq \langle\mu_3(\jb)\rangle \leq \langle\mu_2(\jb)\rangle = \langle \mu_1(\jb'') \rangle$ and we conclude by using the zero momentum condition as in the previous case. 
%
%
%
%
%
%
%

\medskip 
To conclude the analysis of this type, we just observe that \eqref{caribou} is a consequence of the fact that in all the previous cases, we have $\langle \mu_1(\jb'') \rangle \leq \max (\langle \mu_1(\jb'') \rangle, \langle \mu_1(\jb'') \rangle)$ and the definition \eqref{weight} of $\Nc_{\Gamma}(c)$.

\medskip
\noindent {\bf \emph{Type II.}}
The second type of terms we consider are those where one $\omega_{\kb_{\ell,\alpha}}$ appears in the Poisson bracket. They are of the form 
$$
\frac{c_\ell c_{\ell'}(-i)^{p_\ell + q_\ell + p_{\ell'} + q_{\ell'}}z_{\pig_\ell}}{\displaystyle \prod_{\alpha =1}^{n_{\ell} - 1} \omega_{\kb_{\ell,\alpha}}  \prod_{\alpha =n_{\ell}+1}^{p_{\ell}} \Omega_{\kb_{\ell,\alpha}} \prod_{\alpha =1}^{q_{\ell}} \Omega_{\hb_{\ell,\alpha}}
 \prod_{\alpha =1}^{n_{\ell'}} \omega_{\kb'_{\ell',\alpha}}  \prod_{\alpha =n_{\ell'}+1}^{p_{\ell'}} \Omega_{\kb'_{\ell',\alpha}} \prod_{\alpha =1}^{q_{\ell'}} \Omega_{
 \hb_{\ell',\alpha}}
 }
\{ \frac{1}{\omega_{\kb_{\ell,n_\ell}}},z_{\pig'_{\ell'}}\}
$$

Let us set $\jb^* =\kb_{\ell,n_\ell} = (j^*_1,\ldots,j_{\sharp \kb_{\ell,n_\ell}}^*)$. 
The Poisson bracket above is in general zero, except if one of the index of $\jb^*$ is conjugated to one of the index of $\jb' = \pig'_{\ell'}$. We can assume here that $\bar j^*_1 = j'_1$. In this case, we have 
\begin{equation}
\label{dnn}
\{  \frac{1}{\omega_{\jb^*}} ,z_{\jb'}\} = \pm i \frac{1}{\omega_{\jb^*}^2} z_{\jb'}
\end{equation}
So the new term is of the good form with $\jb'' = \jb \cup \jb'$ and up to a combinatorial factor, linear combinations and renumbering we can define the application $\pig''$ in such a way that  $\jb'' = \pig''_{\ell''}$. The term in the denominator has one more factor repeating the index $\kb_{\ell,n_\ell}$. Hence we have $m_{\ell''} = m_\ell + m_{\ell'}$, $n_{\ell''} = n_\ell + n_{\ell'} + 1$,  $p_{\ell''} = p_\ell + p_{\ell'} + 1$, $q_{\ell''} = q_\ell + q_{\ell'}$ and $r'' = r + r' - 1$. As in the Type I case, we can fulfill the reality condition by considering the terms corresponding to $-\ell$ and $-\ell'$ and imposing $\overline{\pi''_{\ell''}} = \pi''_{-\ell''}$, and the conditions ${\bf (i)-(v)}$ of the definition of the class are hence satisfied. 
Moreover, we can verify that these terms fulfill the conditions associated with the subclass $\Hscr_{r'',W}$. In the case when $W = \omega$, we have $q_{\ell''} = q_\ell + q_{\ell'} = 0$ and $n_{\ell''} = n_{\ell} + n_{\ell'} + 1 \leq 2(r +1) - 5 + 2r' - 5 = 2 (r + r') - 8 \leq 2 r'' - 6$. 
Moreover, in the case $ W  = \Omega$, we can set $\alpha_i'' = \alpha_i + \alpha_i'$ for $i \in\{ 1,3,4\}$ and $\alpha_i'' = \alpha_i + \alpha_i' + 1$ for $i \in \{ 2,5\}$ and check that the relations \eqref{grizzli} and \eqref{ilpleut} are satisfied for $\Gamma''$.

In this case $\mu_2(\jb'')$ is necessarily greater than $\mu_2(\jb)$ and $\mu_2(\jb')$, so that ${\bf (vii)}$ is easily proved. 

To prove ${\bf (vi)}$, we observe that the functions $\iota$ and $\iota'$ distribute the indices $\kb_{\ell,\alpha}$ and $\kb'_{\ell',\alpha}$ to some indices in $\jb$ and $\jb'$ respectively that are always lower than the third ones. Hence we have four free indices, and two new derivatives to distribute coming from the presence of the new term $\omega_{\jb^*}$.  We can distinguish two cases: 
\begin{itemize}
\item $\langle j_1'\rangle \leq \langle\mu_2(\jb')\rangle$. In this situation, we use ${\bf (vi)}$ saying that $\langle\mu_{\min}(\jb^*) \rangle \leq C \langle\mu_2(\jb)\rangle$. Hence as $\langle\mu_{\min}(\jb^*)\rangle \leq \langle j_1^*\rangle = \langle j_1' \rangle \leq \langle\mu_2(\jb')\rangle$, we can construct $\iota''$ from $\iota$ and $\iota'$ and by making $\jb^*$ correspond to the third and fourth largest indices amongst $\mu_1(\jb),\mu_2(\jb), \mu_1(\jb')$ and $\mu_2(\jb')$. 

\item $j_1' = \mu_1(\jb')$. We still have by ${\bf (vi)}$ that $\langle\mu_{\min}(\jb^*) \rangle \leq C \langle\mu_2(\jb)\rangle$. Moreover by zero momentum condition, we have $\langle\mu_{\min}(\jb^*)\rangle \leq \langle j_1^*\rangle = \langle j_1' \rangle = \langle \mu_1(\jb')\rangle \leq C_{r'}\langle\mu_2(\jb')\rangle$  and we are back the the previous case.  
\end{itemize}
\noindent {\bf {\emph{Type III.}}}
Now we consider terms where one $\Omega_{\kb_{\ell,\alpha}}$ appears in the Poisson bracket. They are of the form 
$$
\frac{c_\ell c_{\ell'}(-i)^{p_\ell + q_\ell + p_{\ell'} + q_{\ell'}}z_{\pig_\ell}}{\displaystyle \prod_{\alpha =1}^{n_{\ell}} \omega_{\kb_{\ell,\alpha}}  \prod_{\alpha =n_{\ell}+1}^{p_{\ell} - 1} \Omega_{\kb_{\ell,\alpha}} \prod_{\alpha =1}^{q_{\ell}} \Omega_{\hb_{\ell,\alpha}}
 \prod_{\alpha =1}^{n_{\ell'}} \omega_{\kb'_{\ell',\alpha}}  \prod_{\alpha =n_{\ell'}+1}^{p_{\ell'}} \Omega_{\kb'_{\ell',\alpha}} \prod_{\alpha =1}^{q_{\ell'}} \Omega_{
 \hb_{\ell',\alpha}}
 }
\{ \frac{1}{\Omega_{\kb_{\ell,p_\ell}}},z_{\pig'_{\ell'}}\}
$$

Let us set $\jb^* =\kb_{\ell,p_\ell} = (j^*_1,\ldots,j_{\sharp \kb_{\ell,p_\ell}}^*)$, $\jb = \pig_\ell$ and $\jb' = \pig'_{\ell'}$ as before. 
 To compute the Poisson bracket there two case to examine. 

\begin{itemize}
\item First, if $\overline{\jb^*} \cap \jb'= \emptyset $ then 
\[ \{  \frac{1}{\Omega_{\jb^*}} ,z_{\jb'}\} = \pm i P(I)\frac{z_{\jb'}}{\Omega_{\jb^*}^2}, \]
where, in view of the form $Z_6$ (see \eqref{Z6}) , $P$ is a polynomial of degree $1$ with real coefficients. 
Up to a combinatorial factor, linear combinations and renumbering we can define the application $\pig''$ satisfying the reality condition, and we can set 
$m_{\ell''} = m_{\ell} + m_{\ell'} + 1$,  $n_{\ell''} = n_\ell + n_{\ell'}$,  $p_{\ell''} = p_\ell + p_{\ell'}$ and $q_{\ell''} = q_\ell + q_{\ell'} + 1$. The conditions ${\bf (i)-(v)}$ of the definition of the class are hence satisfied. Moreover, we can set $\alpha_i'' = \alpha_i + \alpha_i'$ for $i \in\{ 1,2,3\}$, $\alpha_i'' = \alpha_i + \alpha_i'+1$ for $i \in\{ 4,5\}$  and check that the relations \eqref{grizzli} and \eqref{ilpleut} are satisfied for $\Gamma''$. 
Moreover, ${\bf (vii)}$ is satisfied as $\langle \mu_2(\jb)\rangle \leq \langle \mu_2(\jb'')\rangle$ and $\langle \mu_2(\jb')\rangle \leq \langle \mu_2(\jb'')\rangle$, and ${\bf (vi)}$ is also satisfied as there is no new derivative to distribute.

\item If one of the index of $\jb^*$ is conjugate of one of the index of $\jb'$, then we get 
\begin{equation}
\label{personnenelira}
 \{  \frac{1}{\Omega_{\jb^*}} ,z_{\jb'}\} = \pm i \frac{z_{\jb'}}{\Omega_{\jb^*}^2} +  \pm i P(I)\frac{z_{\jb'}}{\Omega_{\jb^*}^2}, 
 \end{equation}
where $P$ is a polynomial of degree $1$ with real coefficients. We thus treat the second term as previously. To treat the first term, we use the same analysis than the one in type II with $n_{\ell'} = n_{\ell} + n_{\ell'}$,  $p_{\ell''} = p_{\ell} + p_{\ell'} + 1$ $q_{\ell''} =q_{\ell} + q_{\ell'}$. The only difference is that we set  $\alpha_i'' = \alpha_i + \alpha_i'$ for $i \in\{ 1,2,4\}$ and $\alpha_i'' = \alpha_i + \alpha_i' + 1 $ for $i \in \{ 3,5\}$ but the distribution of derivatives is achieved in a similar way. 
\end{itemize}

\noindent {\bf \emph{Type IV.}} Finally  we consider terms where one $\Omega_{\hb_{\ell,\alpha}}$ appears in the Poisson bracket. They are of the form 
$$
\frac{c_\ell c_{\ell'}(-i)^{p_\ell + q_\ell + p_{\ell'} + q_{\ell'}}z_{\pig_\ell}}{\displaystyle \prod_{\alpha =1}^{n_{\ell}} \omega_{\kb_{\ell,\alpha}}  \prod_{\alpha =n_{\ell}+1}^{p_{\ell}} \Omega_{\kb_{\ell,\alpha}} \prod_{\alpha =1}^{q_{\ell} - 1} \Omega_{\hb_{\ell,\alpha}}
 \prod_{\alpha =1}^{n_{\ell'}} \omega_{\kb'_{\ell',\alpha}}  \prod_{\alpha =n_{\ell'}+1}^{p_{\ell'}} \Omega_{\kb'_{\ell',\alpha}} \prod_{\alpha =1}^{q_{\ell'}} \Omega_{
 \hb_{\ell',\alpha}}
 }
\{ \frac{1}{\Omega_{\hb_{\ell,q_\ell}}},z_{\pig'_{\ell'}}\}
$$

It is almost the same as type III except that to deal with the first term in the right-hand side of \eqref{personnenelira} we count one $\Omega_{\jb^*}$ in the denominator as $\Omega_{\hb_{\ell'',q_{\ell''}}}$ with $q_{\ell''} = q_{\ell} + q_{\ell'}+1$  and the other is counted as $\Omega_{\kb_{\ell'',p_{\ell''}}}$ with $p_{\ell''} = p_{\ell} + p_{\ell'}$. The analysis is then the same as in Type II for the distribution of derivatives.

\end{document}